\documentclass[12pt]{amsart}

\usepackage{amsfonts,amssymb,stmaryrd,amscd,amsmath,latexsym,amsbsy,amsthm}

\usepackage{url}

\usepackage{hyperref}

\numberwithin{equation}{section}

\newtheorem{theorem}{Theorem}[section]
\newtheorem{lemma}[theorem]{Lemma}
\newtheorem{proposition}[theorem]{Proposition}
\newtheorem{corollary}[theorem]{Corollary}
\newtheorem{conjecture}[theorem]{Conjecture}
\theoremstyle{definition} 
\newtheorem{definition}[theorem]{Definition}
\newtheorem{example}[theorem]{Example}
\newtheorem{question}[theorem]{Question}
\newtheorem{remark}[theorem]{Remark}

\newcommand{\End}{\text{End}}

\newcommand{\g}{\mathfrak{g}}

\newcommand{\h}{\mathfrak{h}}

\newcommand{\ov}{\overline}

\newcommand{\la}{\lambda}

\renewcommand{\star}{{\dag}}

\newcommand{\bs}{\backslash}

\newcommand{\ben}{\begin{enumerate}}
\newcommand{\een}{\end{enumerate}}

\newcommand{\CC}{{\mathbb{C}}}

\newcommand{\nc}{\newcommand}
\nc{\on}{\operatorname}
\nc{\what}{\widehat}
\nc{\wt}{\widetilde}
\nc{\sw}{{\mathfrak s}{\mathfrak l}}
\nc{\ghat}{\what{\g}}
\nc{\hhat}{\what{\h}}
\nc{\mc}{\mathcal}
\nc{\Bun}{\on{Bun}}
\nc{\ol}{\overline}
\nc{\OO}{\mathcal O}
\nc{\pone}{{\mathbb P}^1}
\nc{\pa}{\partial}
\nc{\Pic}{\on{Pic}}
\nc{\ga}{\gamma}
\nc{\orr}{\underline}
\nc{\mbb}{\mathbb}
\nc{\mbf}{\mathbf}

\theoremstyle{plain}

\newtheorem*{sol}{Solution}

\theoremstyle{definition}

\theoremstyle{remark}

\usepackage{color}

\newcommand{\solu}[1]{\begin{sol}{\bf (\ref{#1})}}
\newcommand{\mP}{\mathbb P}

\newcommand{\mcA}{\mathcal A}
\newcommand{\mcV}{\mathcal V}

\newcommand{\mcC}{\mathcal C}
\newcommand{\mcH}{\mathcal H}
\newcommand{\mcS}{\mathcal S}
\newcommand{\mcD}{\mathcal D}
\newcommand{\mcL}{\mathcal L}
\newcommand{\mcO}{\mathcal O}

\newcommand{\mcF}{\mathcal F}

\newcommand{\fm}{\mathfrak m}
\nc{\Bunc}{\on{Bun}^{\circ}}

\def\g{\mathfrak{g}}
\def\R{\mathbb{R}}

\def\D{\mathcal{D}}
\def\bO{\Omega}

\def\h{\mathfrak{h}}

\def\Z{\mathbb{Z}}

\def\End{\mathrm{End}}

\def\Fq{\mathbb{F}_q}
\def\LG{{}^L\hspace*{-0.4mm}G}
\def\LB{{}^L\hspace*{-0.4mm}B}

\begin{document}

\title{An analytic version of the Langlands correspondence for complex curves}

\author{Pavel Etingof}

\address{Department of Mathematics, MIT, Cambridge, MA 02139, USA}

\author{Edward Frenkel}

\address{Department of Mathematics, University of California,
  Berkeley, CA 94720, USA}

\author{David Kazhdan}

\address{Einstein Institute of Mathematics, Edmond J. Safra Campus,
  Givaat Ram The Hebrew University of Jerusalem, Jerusalem, 91904,
  Israel}

\dedicatory{In memory of Boris Dubrovin}

\begin{abstract}
  The Langlands correspondence for complex curves is traditionally
  formulated in terms of sheaves rather than functions. Recently,
  Langlands asked whether it is possible to construct a
  function-theoretic version. In this paper we use the algebra of
  commuting global differential operators (quantum Hitchin
  Hamiltonians and their complex conjugates) on the moduli space of
  $G$-bundles of a complex algebraic curve to formulate a
  function-theoretic correspondence. We conjecture the existence of a
  canonical self-adjoint extension of the symmetric part of this
  algebra acting on an appropriate Hilbert space and link its spectrum
  with the set of opers for the Langlands dual group of $G$ satisfying
  a certain reality condition, as predicted earlier by Teschner. We
  prove this conjecture for $G=GL_1$ and in the simplest non-abelian
  case.
\end{abstract}

\maketitle

\setcounter{tocdepth}{1}
\tableofcontents

\section{Introduction}

\subsection{The Langlands correspondence for curves over finite
  fields}

Let $X$ be a smooth projective curve over a finite field $\Fq$, $F$
the field of rational functions on $X$, ${\mathbb A}_F$ the ring of
adeles of $F$, $\OO_F \subset {\mathbb A}_F$ the ring of integer
adeles. Let $G$ be a connected reductive algebraic group over
$\Fq$. An invariant measure on $G({\mathbb A}_F)$ induces a measure on
the double quotient $G(F)\bs G({\mathbb
  A}_F)/G(\OO_F)$. Using this measure, one can define a Hilbert
space of $L^2$ functions on this double quotient. A commutative
algebra of self-adjoint integral operators, called the Hecke operators,
acts on this Hilbert space. Hence one can consider the corresponding
spectral problem.

The Langlands correspondence predicts the possibility of expressing
the joint eigenvalues of the Hecke operators in a way that links
them to objects of a different nature: the unramified homomorphisms
from the Weil group of $F$ to $\LG$, the Langlands dual group of
$G$. The Langlands correspondence (in the unramified and the
ramified cases) has been proved for $G=GL_n$ by V. Drinfeld
\cite{Dr1,Dr2} and L. Lafforgue \cite{Laf}, and for classical groups
by J. Arthur \cite{AR}. There has also been a lot of progress in other
cases.

\subsection{A categorical formulation of the
Langlands correspondence}

The double quotient $G(F)\bs G({\mathbb A}_F)/G(\OO_F)$ is in
bijection with the set of isomorphism classes of $G$-bundles on $X$,
i.e. the set of $\Fq$-points of an algebraic moduli stack $\Bun_G$ of
$G$-bundles on $X$. Motivated by this fact, A. Beilinson and
V. Drinfeld proposed a geometric formulation of the Langlands
correspondence for curves over algebraically closed fields of
characteristic zero. This formulation of the Langlands correspondence
is inherently categorical: instead of dealing with the spectral
problem on the space of functions on the set of points of $\Bun_G$, it
deals with a ``spectral problem'' of Hecke {\em functors} acting on
the {\em category} of $D$-modules on $\Bun_G$. Conjecturally, the
``categorical spectra'' are, in some sense, parametrized by flat
$\LG$-bundles on $X$, in agreement with the general Langlands
philosophy.

\subsection{The analytic version in the case of complex curves}

Recently, R. Langlands asked (see \cite{L:analyt}) whether it is
possible to develop an analytic theory of automorphic functions for
curves over $\CC$.

Let $G$ be a connected reductive group over $\CC$, $X$ a smooth
projective complex curve of genus greater than 1, and $\Bun_G$ the
moduli stack of principal $G$-bundles on $X$. Langlands' proposal was
to define a Hilbert space $\mcH$ of functions on the set of
$\CC$-points of $\Bun_G$ and develop a suitable spectral theory on
this space. He suggested to achieve this goal by constructing an
analogue of the theory of integral Hecke operators, as in the case of
curves over finite fields.

In this paper, we take a different approach: we use the existence of a
large commutative algebra of {\em global differential operators}
acting on a $C^\infty$ line bundle $\Omega^{1/2}$ of half-densities on
$\Bun_G$ (see \cite{F:analyt} for some further discussion of the
background and our motivation).

To simplify the exposition, we assume throughout this paper (except in
Section \ref{abelian}) that the group $G$ is simple and
simply-connected. Then the stack $\Bun_G$ is irreducible and there is
a unique up to an isomorphism square root $K^{1/2}$ of the canonical
line bundle on $\Bun_G$ (see \cite{BD}).

Let $D_G$ be the algebra of global regular differential operators
acting on $K^{1/2}$. Elements of this algebra are known as quantum
Hitchin Hamiltonians. The following result is due to Beilinson and
Drinfeld \cite{BD}.

\begin{theorem}    \label{BD1isom}
  There is a canonical algebra isomorphism
  $D_G \simeq \mcC_{\LG}$, where $\mcC_{\LG}$ is the algebra
  of regular functions on the variety $\on{Op}_{\LG}(X)$ of
  $\LG$-opers on $X$.
\end{theorem}

Opers are defined in \cite{BD,BD:opers}. We recall the definition in
Section \ref{opers} below. Denote by $\ol{K}^{1/2}$ the
anti-holomorphic line bundle on $\Bun_G$ obtained by complex
conjugation of $K^{1/2}$ and by $\bO ^{1/2}$ the $C^\infty$ line
bundle $K^{1/2} \otimes \ov K^{1/2}$ on $\Bun_G$.\footnote{For
  non-simply connected $G$, there are different choices for
  $K^{1/2}$. However, any two of them differ by a line bundle $L$ such
  that $L^{\otimes 2}$ is the trivial line bundle. But then $L \simeq
  \ol{L}$ and so $\bO ^{1/2}$ does not depend on the choice of
  $K^{1/2}$.} Let $\ol{D}_G$ be the algebra of global anti-holomorphic
differential operators acting on $\ol{K}^{1/2}$. Its elements are the
anti-holomorphic analogues of the quantum Hitchin Hamiltonians. The
algebra
$$
\mcA := D_G \otimes _\CC \ol{D}_G
$$
acts on sections of $\Omega ^{1/2}$.

The algebra $\mcA$ has a natural anti-linear involution. Let
$\mcA_{\mathbb R} \subset \mcA$ be the ${\mathbb R}$-subalgebra of
invariants of this involution. In this paper we propose an
interpretation of $\mcA_{\mathbb R}$ as a commuting family of
unbounded essentially self-adjoint operators on a Hilbert space $\mcH
=L^2(\Bun_G)$.

Recall that if the genus of $X$ is greater than 1, then there is an
open dense substack $\Bunc_G \subset \Bun_G$ of those stable
$G$-bundles whose group of automorphisms is equal to the center $Z(G)$
of $G$. The stack $\Bunc_G$ is a $Z(G)$-gerbe over a smooth algebraic
variety. In what follows, slightly abusing notation, we will use the
same notation $\Bunc_G$ for this algebraic variety and the
corresponding smooth complex manifold.

We define the space $\mcH =L^2(\Bun_G)$ as the completion of the space
$V$ of smooth compactly supported sections of $\bO ^{1/2}$ on the
complex manifold $\Bunc_G$. Observe that the algebra $\mcA$ naturally
acts on $V$.

The first step of our construction is the defintion of an
$\mcA$-invariant extension $S(\mcA) \subset \mcH$ of $V$. Our first
conjecture (Conjecture \ref{first}) is that $({\mc A}_{\mathbb R},
S(\mcA))$ is a strongly commuting (see Subsection \ref{usa}) family of
unbounded essentially self-adjoint operators on $\mcH$. Our second
conjecture (Conjecture \ref{second}) is that the joint spectrum
$\on{Spec}_\mcH(\mcA)$ of $\mcA$ on $\mcH$ is discrete.

Assuming the validity of these conjectures, Theorem \ref{BD1isom}
implies that the spectrum $\on{Spec}_\mcH(\mcA)$ is a countable subset
$\Sigma$ of the space of $\LG$-opers on $X$ and the corresponding
joint $\mcA$-eigen-sections of $\Omega^{1/2}$ over $\Bunc_G$ form a
basis of $L^2(\Bun_G)$ labeled by $\Sigma$.

Our third conjecture (Conjecture \ref{third}) is that $\Sigma$ is
contained in the set of $\LG$-opers on $X(\CC)$ defined over ${\mathbb
  R}$. In Section \ref{specopers} we sketch an argument showing that
this follows from Conjectures \ref{first} and \ref{second} if we also
assume the affirmative answer to a question of Beilinson and Drinfeld
(formulated as Conjecture \ref{RS} below). This assumption also
implies that the spectrum $\on{Spec}_\mcH(\mcA)$ is simple.

In this paper we prove Conjectures \ref{first}--\ref{third} in the
abelian case $G=GL_1$ and in the simplest non-trivial non-abelian case
(in a more general setting with Borel reductions, see Subsection
\ref{borelred}). Furthermore, we show that in these cases the
spectrum $\on{Spec}_\mcH(\mcA)$ is simple and coincides with the set
of $\LG$-opers on $X(\CC)$ that are defined over ${\mathbb R}$ (which
in these cases is equivalent to their monodromy taking values in the
split real form of $\LG$). We leave it as an open question (Question
\ref{question}) whether this last statement holds in general.

The idea of combining commuting holomorphic and anti-holomorphic
global differential operators on $\Bun_G$ and considering the
corresponding spectral problem in the framework of the Langlands
correspondence is due to J. Teschner \cite{Teschner}. A similar idea
was also proposed by one of us in \cite{F:MSRI}. Teschner considered
the problem of describing single-valued sections of $\Omega^{1/2}$
which are eigenfunctions of the algebra $\mcA$ for $G=SL_2$ and
conjectured that the corresponding eigenvalues are parametrized by
$PGL_2$-opers with real monodromy. He did not address the question
whether these single-valued eigenfunctions are in $L^2(\Bun_G)$. In
this paper we propose a conjectural self-adjoint extension for the
algebra $\mcA_{\mbb R}$ acting on $L^2(\Bun_G)$ and conjecture (and
prove in the simplest cases) a description of the spectrum of this
self-adjoint extension. In the case of $G=SL_2$ we obtain the set of
$PGL_2$-opers with real monodromy, which coincides with the set of
joint eigenvalues of $\mcA$ on single-valued sections of
$\Omega^{1/2}$ conjectured by Teschner. Thus, his conjecture is
compatible with ours in this case, and their compatibility means that
all single-valued eigenfunctions belong to $L^2(\Bun_{SL_2})$ and form
a basis there. We expect that the same is true for the generalization
$\Bun_{SL_2}(X,S)$ of $\Bun_{SL_2}$ defined in the next subsection. In
this paper we show that this is indeed so if $X=\pone$ and $|S|=4$.

\subsection{Borel reductions}    \label{borelred}

We will consider the following generalization of the moduli stack
$\Bun_G$. Let $S\subset X$ be a finite subset. We denote by
$\Bun_G(X,S)$ the moduli stack of pairs $(\mcF ,r_S)$, where $\mcF$ is
a $G$-bundle on $X$ and $r_S$ is a collection of Borel reductions at
the points $x \in S$, i.e. reductions of the fibers $\mcF _x,x\in S$,
to a Borel subgroup $B$ of $G$.\footnote{These are also known as
  parabolic structures.} Let $\Bunc_G(X,S)\subset \Bun_G(X,S)$ be the
substack of stable pairs $(\mcF ,r_S)$ whose group of automorphisms is
equal to $Z(G)$.

\begin{remark}    \label{assumption}
  Throughout this paper, we assume that either the genus of $X$ is
  greater than 1; or $X$ is an elliptic curve and $|S| \geq 1$; or
  $X=\mP^1$ and $|S| \geq 3$. Since $\Bun _G(\mP ^1,S)$ is a point
  when $|S| =3$, we will effectively only consider $\Bun _G(\mP ^1,S)$
  with $|S|\geq 4$. In these cases, $\Bunc_G(X,S)$ is open and dense
  in $\Bun_G(X,S)$.
\end{remark}

Let $\on{Op}_{\LG}(X,S)$ be the variety of $\LG$-opers on $X$ that are
regular outside $S$ and have regular singularities at the points $x\in
S$ with oper-residue $\varpi(0)$, see Section \ref{parab} for more
details (we note that $\on{Op}_{\LG}(X,S)$ is a particular component
of the space of $\LG$-opers with regular singularities at the points
$x\in S$ and unipotent monodromies around these points).

The following statement is a generalization of Theorem \ref{BD1isom}
and follows from the results of \cite{FF,F:wak} and \cite{FG1} (see
Section \ref{parab}).

\begin{theorem}    \label{mon}
  There exists a canonical algebra embedding $\mcC_{\LG}(X,S)
  \hookrightarrow
  D_G(X,S)$ where $\mcC_{\LG}(X,S)$ is the algebra of regular
  functions on
  $\on{Op}_{\LG}(X,S)$ and $D_G(X,S)$ is the algebra of global regular
  differential operators on the line bundle $K^{1/2}$ on
  $\Bun_G(X,S)$.
\end{theorem}

\subsection{Operators on $\mcH$}    \label{opsH}

Let $\ol{K}^{1/2}$ be the
anti-holomorphic line bundle on $\Bun_G(X,S)$ which is the complex
conjugate of $K^{1/2}$ (see Section \ref{gen} for more details) and
$\Omega ^{1/2}$ the $C ^\infty$ line bundle $K^{1/2}\otimes _\CC
\ol{K} ^{1/2}$ on $\Bun_G(X,S)(\CC)$. Let $\ol{D}_G(X,S)$ be the
algebra of global anti-holomorphic differential operators acting on
$\ol{K}^{1/2}$. The algebra
$$
\mcA := D_G(X,S) \otimes _\CC \ol{D}_G(X,S)
$$
is embedded into the algebra of smooth differential operators on
the line bundle $\Omega ^{1/2}$ on $\Bunc_G(X,S)$.

Recall that according to our assumption in Remark \ref{assumption},
$\Bunc_G(X,S)$ is open and dense in $\Bun_G(X,S)$. Moreover,
$\Bunc_G(X,S)$ is a $Z(G)$-gerbe over a smooth algebraic variety
(which is the corresponding coarse moduli space). From now on, we will
use the notation $\Bunc_G(X,S)$ for this algebraic variety and the
corresponding smooth complex manifold.

Let $V$ be the space of smooth compactly supported sections of $\Omega
^{1/2}$ on the complex manifold $\Bunc_G(X,S)$. We have a natural
positive-definite Hermitian form $\langle,\rangle$ on $V$ given by
$$
\langle v,w \rangle :=\int _{\Bunc_G(X,S)(\mbb C)} v \cdot \overline{w},
\qquad v,w\in V.
$$
We denote by $\mcH$ the Hilbert space completion of $V$.

Let $\what{V}$ be the space of smooth sections of $\Omega^{1/2}$ on
$\Bunc_G$. The algebra ${\mc A}$ acts on $\what{V}$. For a homomorphism
$\chi: {\mc A} \to \CC$, denote by $V_\chi \subset \what{V}$ the
$\chi$-eigenspace of ${\mc A}$. It follows from Theorem \ref{mon} that
every $\chi$ can be described as a pair $(\la,\mu)$, where $\la$ is a
holomorphic $\LG$-oper and $\mu$ is an anti-holomorphic $\LG$-oper on
$X$. Furthermore, we expect that for every $\chi = (\la,\mu)$ such
that $V_\chi \neq 0$, $\mu$ is uniquely determined by $\la$ (see
Lemma \ref{ollast}).

Let $\Delta_\la$ be the holomorphic $K^{1/2}$-twisted $D$-module on
$\Bun_G$ corresponding to $\la$ and $\ol\Delta_\mu$ the
anti-holomorphic $\ol{K}^{1/2}$-twisted $D$-module on $\Bun_G$
corresponding to $\mu$ (see Section \ref{diffBun}). These $D$-modules
are known to be holonomic. Let $U$ be an open dense subset of
$\Bunc_G$ on which $\Delta_\la$ and $\ol\Delta_\mu$ restrict to
$C^\infty$ vector bundles of finite rank with projectively flat
connections. Consider the restriction to $U$ of a non-zero element
$\gamma$ of $V_\chi$, where $\chi=(\la,\mu)$. It satisfies two
properties: (1) it is single-valued as a section of $\Omega^{1/2}$, and
(2) it can be expressed locally as a finite sum
\begin{equation}    \label{eigenfunction}
\gamma|_U = \sum_{i,j} a_{ij} \; \phi_i({\mbf z})
\ol\psi_j(\ol{\mbf z})
\end{equation}
where $\{ \phi_i \}$ is a local basis of sections of $\Delta_\la|_U$
and $\{ \ol\psi_j \}$ is a local basis of sections of
$\ol\Delta_\mu|_U$. This implies that $\on{dim} V_\chi < \infty$ for
all $\chi$. Moreover, if $\Delta_\la$ is irreducible and has regular
singularities (as anticipated in \cite{BD}, see Conjecture \ref{RS}
below), then $\Delta_\la|_U$ is also irreducible, and so if $V_\chi
\neq 0$, then $\on{dim} V_\chi = 1$.

We conjecture that the square-integrable eigenfunctions $\gamma$
constructed this way form an orthogonal basis of $\mcH =
L^2(\Bun_G)$. We reformulate this as follows. Let $V^0_\chi = V_\chi
\cap \mcH$ (we leave it as an open question whether $V^0_\chi =
V_\chi$ for all $\chi$, see Question \ref{question}).

\begin{conjecture}    \label{basic}
\hfill
\begin{enumerate}
\item We have an orthogonal decomposition
$$
\mcH = \bigoplus_\chi V^0_\chi.
$$

\item The set $\Sigma = \{ \la \in \on{Op}_{\LG}(X) \, | \,
  \exists \mu: V^0_{(\la,\mu)} \neq 0 \}$ is discrete in the complex
  topology on $\on{Op}_{\LG}(X)$.
\end{enumerate}
\end{conjecture}

In the abelian case, Conjecture \ref{basic} can be proved directly
(and moreover, $V^0_\chi = V_\chi$ for all $\chi$), see
Section \ref{abelian}. In the non-abelian case, we first need to
construct a self-adjoint extension of an ${\mbb R}$-subalgebra ${\mc
  A}_{\mbb R} \subset {\mc A}$. This is discussed in the next
subsection.

\subsection{Self-adjoint extension}

For the explanation of the notation and the proof of the following
result see Section \ref{gen}.

Let $C\mapsto C^*$ be the Verdier anti-involution on $D_G$ (i.e.,
taking the algebraic adjoint). We denote by $A\to A^\star$ the
anti-involution on $\mcA$ given by $C_1 \otimes \ol{C_2} \mapsto C_2^*
\otimes \ol{C_1^*}$. Integration by parts implies the following.

\begin{lemma} $\langle Av,w \rangle = \langle v,A^\star w \rangle$ for
  any $A\in \mcA , v,w\in V $.
\end{lemma}

We set $\mcA _{\mathbb R} =\{ A\in \mcA | A= A^\star \}$. 

\begin{corollary}
Every $A\in \mcA _{\mathbb R}$ is a symmetric operator on $V$.
\end{corollary}

Here ``symmetric'' means ``Hermitian symmetric''.

The following notions are discussed in detail in Subsection \ref{preli}. 

\begin{definition}\label{es1} The {\it Schwartz space} $S(\mcA)$ 
  is the space of $u\in \mathcal{H}$ such that the linear functional
  $v\mapsto \langle Av,u \rangle$ on $V$ extends to a continuous
  functional on $\mcH$ for any $A\in \mcA$. 
\end{definition}

\begin{definition}\label{es2b} We say that $\mcA_{\mbb R}$ is an {\it
    essentially self-adjoint algebra on} $V$ if
every element $A\in \mcA_{\mbb R}$ is essentially self-adjoint on $S(\mcA)$. 
\end{definition}

\subsection{The Main Conjectures}    \label{main}

\begin{conjecture}\label{first}
The algebra $\mcA_{\mbb R}$ is essentially self-adjoint on $V$. 
\end{conjecture}

As explained in Section \ref{preli} (see Proposition
\ref{spectrumexists}), this conjecture implies that one can define the
joint spectrum $\on{Spec}_{\mcA}(\mcH)$ of the algebra $\mcA$ on
$\mcH$.

\begin{conjecture}\label{second}
  The spectrum $\on{Spec}_{\mcA}(\mcH)$ is discrete. In other words,
  $\mathcal{H}$ is the (completed) direct sum of finite-dimensional
  eigenspaces of $\mcA$.
\end{conjecture}

According to Theorem \ref{mon}, $\on{Spec}_{\mcA}(\mcH)$ is naturally
embedded as a subset of \linebreak $\on{Op}_{\LG}(X,S)$, which we
denote by $\on{Op}_{\LG,\mcH}(X,S)$.

On the other hand, let $\on{Op}_{\LG}(X,S)_{{\mathbb R}}$ be the set
of $\LG$-opers $\la$ for which the underlying $C^\infty$ flat
$\LG$-bundle on $X$ is isomorphic to the $C^\infty$ flat $\LG$-bundle
underlying the complex conjugate anti-holomorphic $\LG$-oper $\ol\la$
(see the definition in Section \ref{opers}).

To each $\LG$-oper $\lambda $ we associate its monodromy
representation
$$
\rho_\lambda: \pi _1(X \bs S)\to \LG(\CC).
$$
The set $\on{Op}_{\LG}(X,S)_{{\mathbb R}}$ may be equivalently
described as the set of $\LG$-opers $\la$ such that that the
representation $\rho_\lambda$ and its complex conjugate representation
$\ol\rho_\lambda$ are isomorphic. We expect that this set is a
discrete subset of the space of $\on{Op}_{\LG}(X,S)$ (this is known
for $G=SL_2$, \cite{F}).

\begin{conjecture}    \label{third}
  The set $\on{Op}_{\LG ,\mcH}(X,S)\subset \on{Op}_{\LG}(X,S)$ is a subset
  of the set \linebreak $\on{Op}_{\LG}(X,S)_{{\mathbb R}}\subset
  \on{Op}_{\LG}(X,S)$.
\end{conjecture}

Conjectures \ref{first}--\ref{third} imply and refine Conjecture
\ref{basic}.

In this paper we prove the validity of the above three conjectures in
the following two cases:

\begin{enumerate}
\item $G=GL_1$, $X$ is an arbitrary curve of genus greater than 0, and
  $S=\emptyset$;

\item $G=SL_2, X=\mP ^1$, and $|S|=4$.
\end{enumerate}

In both cases, we actually prove a stronger statement:

\medskip

\begin{theorem}
In the above two cases, the set $\on{Op} _{\LG ,\mcH}(X,S) $ coincides
with the set $\on{Op}_{\LG}(X,S)_{{\mathbb R}}$.
\end{theorem}

\begin{question}\label{question}
  Is it true in general that the set $\on{Op}_{\LG ,\mcH}(X,S)\subset
  \on{Op}_{\LG}(X,S)$ coincides with $\on{Op}_{\LG}(X,S)_{{\mathbb R}}$?
\end{question}

\begin{remark}
  Consider the case $G=SL_2$ and $\LG=PGL_2$. In this case, a
  $PGL_2$-oper is the same as a projective connection on $X$ (see
  \cite{BD:opers}). Suppose that the genus of $X$ is greater than 1
  and $S=\emptyset$. It follows from \cite{GKM} that a $PGL_2$-oper
  $\lambda$ satisfies the condition $\rho_\la \simeq \ol\rho_\la$ if
  and only if the monodromy representation $\rho_\lambda$ takes
  values, up to an inner automorphism of $PGL_2(\CC)$, in the split
  real form $PGL_2(\R) \subset PGL_2(\CC)$.

  Thus, $\on{Op}_{PGL_2}(X)_{{\mathbb R}}$ may be equivalently
  described as the set of $PGL_2$-opers with monodromy in
  $PGL_2(\R)$. Such opers were studied in \cite{G,F,T}. We expect that
  the same is true for $PGL_2$-opers even if $S \neq
  \emptyset$. However, we are not aware of similar results for
  non-abelian groups $\LG$ other than $PGL_2$.
\end{remark}

\subsection{The structure of the paper}

The paper is organized as follows. 

In Part I we discuss the conjectural picture for general $X$ and
$G$. In Section \ref{gen} we collect some properties of differential
operators acting on line bundles which we use in this paper. In
Section \ref{diffBun}, we discuss the global holomorphic and
anti-holomorphic differential operators acting on the square root of
the canonical line bundle on $\Bun_G$, in the case when there are no
parabolic structures (i.e. $N=0$). In Section \ref{specopers}, we
describe the link between the spectra of (canonical self-adjoint
extension of) the global differential operators on $\Bun_G$ and
$\LG$-opers $\la$ on the curve $X$ satisfying a reality condition,
leaving aside the question of the self-adjoint extension.  In Section
\ref{abelian} we describe the spectrum explicitly and prove
Conjectures \ref{first}--\ref{third} in the abelian case $G=GL_1$. In
Section \ref{parab} we include the parabolic structures at finitely
many marked points on $X$. In Section \ref{SL2}, we discuss in detail
the case of $G=SL_2$, $X=\pone$, and $|S|>3$. Section \ref{App}
contains the proofs of two results that we need in Section
\ref{specopers}.

In Part II we prove Conjectures \ref{first}--\ref{third} in the
special case $G=SL_2, X=\mbb P^1$, and $|S|=4$. In the course of
the proof, we develop a theory of Sobolev and Schwartz spaces relevant
to this problem. See Section 11 for a general discussion of unbounded
self-adjoint operators and their spectral theory. A more detailed
description of the contents of Sections 9--19 is given at the
beginning of Part II.

\medskip

\noindent{\bf Acknowledgements.} We thank Dennis Gaitsgory for his
help with the proof of Proposition \ref{nu}. E.F. thanks Dima Arinkin
for useful discussions. The work of P.E. and D.K. on this project was
supported by ERC under grant agreement 669655. The work of P.E. was
partially supported by the NSF grant DMS-1502244, and he is grateful
to the Department of Mathematics of Hebrew University for hospitality.

\bigskip

\bigskip

\part*{\hspace*{70mm}Part I}

\bigskip

\bigskip

\section{Differential operators on line bundles}    \label{gen}

\subsection{$D$-modules on line bundles}    \label{Dmodlb}

Let $Y$ be a smooth connected complex algebraic variety. Abusing
notation, we will also denote by $Y$ the complex manifold $Y(\mbb
C)$. Let $K$ the canonical line bundle on $Y$. Let $\mcD$ be the sheaf
of holomorphic differential operators on $Y$ in the complex analytic
topology. For any algebraic line bundle $L$ on $Y$ we denote by
$\mcD_L$ the sheaf of holomorphic differential operators acting on
$L$. Then
$$
\mcD_L \simeq L \otimes_{\mcO_Y} \mcD \otimes_{\mcO_Y} L^{-1},
$$
To any $\mcD_L$-module $M$ on $Y$ we associate the $\mcD$-module $M^0 :=
L^{-1} \otimes_{\mcO_Y} M$ on $Y$. Conversely, $M \simeq L
\otimes_{\mcO_Y} M^0$. Hence we obtain an equivalence between the
categories of $\mcD_L$-modules and $\mcD$-modules on $Y$ (see
e.g. \cite{BB}).

If a $\mcD$-module $M^0$ is regular holonomic, then there exists an
open dense subset $U\subset Y$ such that the restriction $M^0|_U$ is a
holomorphic vector bundle with a holomorphic flat connection. We
denote by ${\mcS}_U $ the corresponding local system (with respect to
the analytic topology on $U$).

\begin{lemma}\label{reg}
An irreducible regular holonomic $\mcD$-module $M^0$ such that $M^0|_U\neq
\{0\}$ is uniquely determined by the local system ${\mcS}_U$.
\end{lemma}

\begin{proof} By the Riemann-Hilbert correspondence, the local system
  ${\mcS}_U $ defines a holomorphic vector bundle with a holomorphic
  flat connection with regular singularities, which we denote by
  $\mcV_U$. Since $M^0$ is regular we have $M^0|_U \simeq
  {\mcV}_U$. The irreducibility of $M^0$ then implies that $M^0$ is
  equal to the intermediate extension of ${\mcV}_U$.
\end{proof}

\subsection{Complex conjugation}    \label{cpxconj}

For an algebraic line bundle $L$ on $Y$, we denote by the same letter
$L$ the corresponding holomorphic line bundle on $Y$ and by $\ol{L}$
the complex conjugate anti-holomorphic line bundle on $Y$. To define
$\ol{L}$, pick an open covering $U_i,i\in I$, of $Y$ and
trivializations $\alpha_i:L|_{U_i}\cong U_i\times \mbb C$.  Then we
have $\mbb C^\times$-valued holomorphic gluing functions $\phi_{ij}$
on $U_{ij}:=U_i\cap U_j$. The line bundle $\ol{L}$ on $Y$ is then
defined by the same covering and the anti-holomorphic gluing functions
$\ol\phi_{ij}$. It is clear that this definition does not depend on
the covering $U_i$ and trivializations $\alpha_i$. Moreover, for any
local section $s\in \Gamma(U,L)$ on an open subset $U \subset Y$ (in
analytic topology), there is a canonically defined complex conjugate
section $\ol{s}\in \Gamma(U,\ol{L})$ such that
$\alpha_i(\ol{s})=\ol{\alpha_i(s)}$.

To any local section $D$ of $ \mcD_L$, we associate a differential
operator $\ol{D}$ on $\ol{L}$ acting by the formula
$\ol{D}(\ol{s})=\ol{D(s)}$ for every local section $s$ of
$L$. Thus, we obtain a sheaf $\ol{\mc D}_L$ of anti-holomorphic
differential operators acting on $\ol{L}$. 

\subsection{Canonical anti-involution}    \label{cananti}

The following realization of the Verdier duality is proved in
\cite{B}, Sect. 6.3 (see also \cite{BB}, Sect. 2.4).

\begin{proposition}
  Let $E,F$ be two vector bundles on $Y$ and $A:E\to F$ a differential
  operator. Then there exists a unique differential operator $A^* :
  K_Y\otimes F^* \to K_Y\otimes E^* $ such that for any local
  sections $e$ of $E$ and $f$ of $F^* $ the differential
  form $\langle Ae,f \rangle - \langle e,A^* f \rangle$ is
  exact.

Furthermore, we have
$$
(AB)^* =B^*  A^*, \qquad   (A ^* )^* =A.
$$
\end{proposition}

Namely, if $x_1,...,x_n$ are local coordinates on $Y$ and we fix local trivializations of $E,F$, as well as the trivialization of $K_Y$ 
using the volume form $dy_1\wedge...\wedge dy_n$, then the map $A\mapsto A^*$ is defined by the formula 
$x_i^*=x_i$, $\partial_i^*=-\partial_i$. In other words, $A^*$ is the algebraic adjoint of $A$. 

Thus, we obtain a canonical anti-involution
\begin{eqnarray*}
\nu: \mcD_L &\to &\mcD _{KL^{-1}} \\
A &\mapsto & A^*
\end{eqnarray*}
In particular, if we fix an isomorphism $L^{\otimes 2}\to K$, then
$\nu$ is an anti-involution of the sheaf $\mcD_L$. We will describe
this anti-involution explicitly in the case of $Y=\Bunc_G$ in
Proposition \ref{nu} below.

\subsection{Formal adjoint operators}\label{ratsing}

Let again $L$ be a line bundle on $Y$ and $\ol{L}$ the complex
conjugate line bundle. Let $V$ be
the space of smooth sections with compact support of
the $C^\infty$ line bundle $L\otimes_{\CC} \ol{KL^{-1}}$. For any section $\phi \in V$ we can view $\ol\phi$
as a section of the $C ^\infty$ line bundle $KL^{-1}\otimes_{\CC}
\ol{L}$. Then $\phi\ol\psi$ is a section of $K\otimes \ol K=\Omega$, the bundle of $C^\infty$ top differential forms on $Y$. 
We define a Hermitian form $\langle \cdot,\cdot \rangle$ on
$V$ by
$$
\langle\phi , \psi \rangle := \int_Y \phi \ol\psi.
$$

\begin{lemma}    \label{formaladj}
For any  $\phi ,\psi \in V$ we have
$$
\langle D \phi,\psi \rangle=\langle \phi,D^\dagger \psi \rangle 
$$
where
\begin{equation}    \label{Dstar}
D^\star :=\overline {\nu (D)}.
\end{equation}
\end{lemma}

\begin{proof}
  In the case when $Y$ is an open ball and $L \simeq \mathcal
  O$ we have $\nu (f)=f, \nu (v)=-v$ where $f$ is a
  holomorphic function and $v$ a holomorphic vector field. In
  this case the equality is obvious. Since the general statement is
  local, it follows immediately from this special case.
\end{proof}

\subsection{Solutions and pairings}

To simplify notation, we assume until the end of this section that $L$
is equipped with an isomorphism $L^{\otimes 2}\to K$ and denote by
$\Omega ^{1/2}$ the $C ^\infty$ line bundle $L\otimes \ol{L}$.

Let $\mcC \subset \Gamma (Y,\mcD _L)$ be a commutative subalgebra.
Let $U$ be the set of $y\in Y$ such that the common zero set of the
symbols of all $D\in \mcC$ on $T_y^*Y$ is $\lbrace{0\rbrace}$.  Then
$U$ is a Zariski open subset of $Y$, and we assume that it is dense.
Then for any homomorphism $\la: \mcC\to \mbb C$, the system of
equations $Df=\la(D)f$, $D\in \mcC$, is holonomic on $Y$. Thus the
left $\mcD _L $-module
$$
M(\la):=\mcD _L /(\mcD _L \cdot {\rm Ker}\la)
$$
and the left $\mcD$-module
$$
M(\la)^0 = L^{-1} \otimes_{\mcO_Y} \mcD _L/(\mcD _L \cdot {\rm Ker}\la)
\simeq (\mcD \otimes_{\mcO_Y} L^{-1})/(\mcD _L \cdot {\rm Ker}\la)
$$
are $\mathcal{O}$-coherent on $U$.

In particular, $M(\la)^0|_U$ is a holomorphic vector bundle
$\mcV(\la)$ (of some rank $n$ independent on $\la$) with a holomorphic
flat connection. Let us fix $y\in U$ and denote by $\rho_\la$ the
corresponding representation of the fundamental group $\pi_1(U,y)$ on
the fiber $M(\la)^0|_y$ of $M(\la)^0$ at $y$.

On the other hand, denote by $\mcF_\la $ the sheaf of local
holomorphic solutions $f$ of the system $\{ Df=\la(D)f, D\in \mcC \}$
on $U$. If $B$ is a small ball around $y$, then we have a natural
isomorphism $M(\la)|_y\to \mcF_\la(B), v\mapsto f_v$ given by solving the
initial value problem. This enables us to associate to
  $\mcF_\la $ a representation $\wt\rho_\la$ of the group $\pi_1(U,y)$
  on the fiber $M(\la)|_y$. Since $M(\la)|_y = L_y \otimes
  M(\la)^0|_y$, we have $\wt\rho_\la \simeq \rho_\la$.

\begin{definition} Given two homomorphisms $\la,\eta$, we denote by $V_{\la,\eta}$ the space of
  sections $\gamma$ of $\Omega ^{1/2}$ on $U$ such that
\begin{equation}\label{eq400}
D\gamma =\la(D)\gamma, \qquad \ol{D} \gamma =\ol{\eta(D)}\gamma, \qquad D\in \mcC. 
\end{equation} 
\end{definition}

\begin{proposition}    \label{Her} 
  Let $\la , \eta: \mcC\to \mbb C$ be two homomorphisms. Then there is a natural isomorphism
  $$
  V_{\la,\eta}\cong (\rho_\la\otimes \ol\rho_\eta)^{\pi_1(U,y)}.
  $$ 
  \end{proposition}

\begin{proof} Let $v\in M(\la)|_y$, $w\in M(\eta)|_y$. Then the local
  section $f_v\otimes \overline{f_w}$ of $\Omega^{1/2}$ satisfies
  \eqref{eq400}.
But \eqref{eq400} is a real holonomic system of rank $n^2$. Hence any
local solution $\gamma$ of \eqref{eq400} can be written as
\begin{equation}    \label{locsum}
\gamma =\sum_jf_{v_j} \overline{f_{w_j}}, \qquad v_j\in M(\la)|_y ,
\quad w_j \in M(\eta)|_y.
\end{equation}
This implies that the linear map $v\otimes \ol{w}
  \mapsto f_v\otimes \overline{f_w}$
defines an isomorphism from $\wt\rho_\la\otimes \ol{\wt\rho}_\eta \simeq
\rho_\la\otimes \ol\rho_\eta$ to the space $W_{\la,\eta}$ of local
solutions of \eqref{eq400} near $y$ 
(in particular, all such solutions are real analytic). This isomorphism clearly commutes with the action of $\pi_1(U,y)$ (given by analytic continuation), hence defines an isomorphism 
$$
(\rho_\la\otimes \ol\rho_\eta)^{\pi_1(U,y)}\cong W_{\la,\eta}^{\pi_1(U,y)}=V_{\la,\eta},
$$ 
as claimed. 
\end{proof} 

\begin{remark}    \label{Her1}
Let again $\mcV(\la) = M(\la)^0|_U$. We denote by the same symbol
the corresponding $C^\infty$ vector bundle with a flat connection on
$U$. Proposition \ref{Her} is equivalent to the existence of an
isomorphism
\begin{equation}    \label{Her2}
V_{\la,\eta} \simeq \on{Hom}(({\mc C}^\infty_U,d),{\mc V}(\la) \otimes
\ol{{\mc V}(\eta)}),
\end{equation}
where $({\mc C}^\infty_U,d)$ denotes the $C^\infty$ trivial flat line
bundle on $U$.
\end{remark}

\section{Differential operators on $\Bun_G$}    \label{diffBun}

In this section, we describe some properties of the algebra $\mcA$ of
differential operators acting on the line bundle $\Omega^{1/2}$ on
$\Bun_G$ in the case when there are no parabolic structures
(i.e. $N=0$). We will extend these results to the case of parabolic
structures in Section \ref{parab}.

\subsection{Opers}    \label{opers}

Recall from \cite{BD:opers} that an (algebraic, hence holomorphic)
$\LG$-oper on $X$ is a triple $({\mc F},\nabla,{\mc F}_{\LB})$, where
${\mc F}$ is a (holomorphic) $\LG$-bundle on $X$, $\nabla$ is a
(holomorphic) connection on ${\mc F}$, and ${\mc F}_{\LB}$ is a
(holomorphic) reduction of ${\mc F}$ to a Borel subgroup $\LB
\subset \LG$ satisfying the oper transversality condition (as defined
in \cite{BD:opers}). Denote the  variety of $\LG$-opers on 
$X$ by $\on{Op}_{\LG}(X)$.

Recall that we are under the assumption that $G$ is simple and
simply-connected. Then $\LG$ is of adjoint type. In this case,
$\on{Op}_{\LG}(X)$ is the affine space of all holomorphic connections
on a particular holomorphic $\LG$-bundle ${\mc F}_0$ on $X$. In other
words, for every holomorphic connection $\nabla$ on ${\mc F}_0$, there
is a unique Borel reduction of ${\mc F}$ satisfying the oper condition
(see \cite{BD}).

Given $\lambda \in \on{Op}_{\LG}(X)$, we denote the corresponding flat
$\LG$-bundle $({\mc F}_0,\nabla_\la)$ by ${\mc E}(\la)$.

\subsection{Global holomorphic differential operators on
  $\Bun_G$}    \label{holom}

Let ${\mc D}_G$ be the sheaf of holomorphic differential operators
acting on the line bundle $K^{1/2}$ and $D_G=\Gamma (\Bun_G, \mcD_G)$.
In the case we are considering now (there are no marked points on $X$,
i.e.  $S=\emptyset$) the commutative subalgebra of $D_G$ discussed in
the Introduction coincides with $D_G$ itself. Furthermore, Beilinson
and Drinfeld proved \cite{BD} that there is a canonical isomorphism
\begin{equation}    \label{BDisom}
D_G \simeq \mcC_{\LG},
\end{equation}
where $\mcC_{\LG}$ is the algebra of regular functions on
$\on{Op}_{\LG}(X)$.

\begin{remark} The generalization of the Beilinson-Drinfeld
  construction to the moduli stack of $G$-bundles on $X$ with
  parabolic structures at $N$ marked points is discussed in Section
  \ref{parab} below. Then we obtain a commutative subalgebra of the
  algebra of global differential operators, which is isomorphic to the
  space of functions on an appropriate space of $\LG$-opers with
  regular singularities at the marked points.
\end{remark}

For an $\LG$-oper $\la \in \on{Op}_{\LG}(X)$ we denote by $\fm_\la$ the
corresponding maximal ideal in the commutative algebra $D_G$ and by
$I_\la = {\mc D}_G \cdot \fm_\la$ the corresponding left ideal in the
sheaf ${\mc D}_G$. Beilinson and Drinfeld \cite{BD} constructed the
$K^{1/2}$-twisted holonomic $D$-module on $\Bun_G$
$$
\Delta_\la := {\mc D}_G/I_\la
$$
They proved that the corresponding untwisted $D$-module 
$$
\Delta_\la^0 = K^{-1/2} \otimes_{\mc O} \Delta_\la
$$
is a {\em Hecke eigensheaf} with respect to ${\mc E}(\la)$.

For completeness, we recall this notion in the next subsection.

\subsection{Hecke eigensheaves}    \label{Hecke1}

Let $D(\Bun_G)$ be the category of $D$-modules on $\Bun_G$ and $D(X
\times \Bun_G)$ the category of $D$-modules on $X \times \Bun_G$.

To any algebraic finite-dimensional representation $R$ of $\LG$ one
can associate (see \cite{BD}) the {\it Hecke correspondence} $\mcH_R$
over $\Bun_G\times (X \times \Bun _G)$ and a $D$-module ${\mc F}_R$ on
$\mcH_R$ which gives rise to the {\it Hecke functor} $H_R: D(\Bun_G)
\to D(X \times \Bun_G)$.

On the other hand, let ${\mc E}$ be an algebraic (equivalently,
holomorphic) flat $\LG$-bundle on $X$. For each representation $R$
of $\LG$, let
\begin{equation}    \label{assocbdle}
R_{\mc E} := {\mc E} \underset{\LG}\times R
\end{equation}
be the associated flat vector bundle on $X$. We view it as a
$D$-module on $X$.

A $D$-module $M$ on $\Bun_G$ is called a {\em Hecke eigensheaf} with
respect to ${\mc E}$ if for every $R \in \on{Rep} \LG$ there is an
isomorphism
\begin{equation}    \label{Hecke}
H_R(M) \simeq R_{{\mc E}} \boxtimes M,
\end{equation}
and these isomorphisms are compatible with the monoidal structures on
both sides.

Now let $\la \in \on{Op}_{\LG}(X)$ and ${\mc E}(\la)$ the
corresponding flat $\LG$-bundle on $X$. Let $\Delta^0_\la$ be the
$D$-module on $\Bun_G$ defined in the previous subsection. The
following statement is proved in \cite{BD}.

\begin{theorem}    \label{Heckeeig}
$\Delta^0_\la$ is a Hecke eigensheaf with respect to 
${\mc E}(\la)$.
\end{theorem}

\begin{remark}
  The $D$-module $\Delta^0_\la$ (or the corresponding perverse sheaf)
  on $\Bun_G$ may be viewed as an analogue of an unramified Hecke
  eigenfunction of the classical theory (over $\Fq$) in the
  sheaf-theoretic version of the geometric Langlands correspondence
  over $\CC$.
\end{remark}

Applying the Riemann-Hilbert correspondence to the $D$-module
$\Delta_\la^0$, we obtain a perverse sheaf ${\mc K}_{{\mc E}(\la)}$ on
$\Bun_G$. In the same way as above, one defines the notion of the
Hecke functors on the category of perverse sheaves on $\Bun_G$. A
Hecke eigensheaf in this category is then a perverse sheaf $M$
equipped with a system of isomorphisms \eqref{Hecke} with respect to
these functors (compatible with the monoidal structures) in which
instead of the flat vector bundle $R_{\mc E}$ on $X$ we take the
locally constant sheaf of its horizontal sections (in the analytic
topology). We obtain that the sheaf ${\mc K}_{{\mc E}(\la)}$ is a
Hecke eigensheaf with respect to the sheaf of horizontal sections of
${\mc E}(\la)$.

Recall that for an $\LG$-oper $\la$ on $X$ we denote by $\rho_\la$ the
corresponding monodromy representation $\pi_1(X) \to \LG$.

Let $U$ be an open dense subset of $\Bun_G$ such that
$\Delta^0_\la{}|_U$ is a vector bundle with a flat connection, which
we denote by $\mcV_\la$. We denote the corresponding locally constant
sheaf on $U$ by ${\mc S}_\la$. Theorem \ref{Heckeeig} and Lemma
\ref{reg} imply the following statements.

\begin{corollary}\label{eig}
\hfill 
\begin{enumerate}
\item The flat $\LG$-bundle ${\mc E}(\la)$ and the monodromy
  representation $\rho_\la$ are uniquely determined by the $D$-module
  $\Delta^0_\la$ and the corresponding perverse sheaf ${\mc K}_{{\mc
      E}(\la)}$.

\item If the $D$-module $\Delta^0_\la$ is regular holonomic and
  irreducible, then ${\mc E}(\la)$ and $\rho_\la$ are uniquely
  determined by the flat vector bundle $\mcV_\la$ and the
  corresponding locally constant sheaf ${\mc S}_\la$.
\end{enumerate}
\end{corollary}

\begin{remark} It is known \cite{BD} that $\Delta^0_\la|_U$ is a
  vector bundle with a flat connection if we take as $U$ the subset of
  $G$-bundles ${\mc F}$ such that the vector bundle $\g_{\mc F}
  \otimes K_X$, where $K_X$ is the canonical line bundle on $X$, does
  not admit non-zero sections taking nilpotent values everywhere on
  $X$.
\end{remark}

\subsection{The involution $\nu$}

Recall the anti-involution $\nu$ from Section \ref{cananti}. In our
case, we have an anti-involution $\nu$ on the sheaf ${\mc D}_G$ and
hence on its algebra $D_G$ of global sections.
Hence it defines an involution on the space
$\on{Op}_{\LG}(X)$, which we also denote by $\nu$.

On the other hand, let $\tau$ be a Chevalley involution of $\LG$, see
\cite{AV}, Sect. 2 for a precise definition. This is an automorphism
of $\LG$ that is well-defined up to an inner automorphism.\footnote{To
  fix $\tau$, we need to fix a ``pinning'' of $\g$, i.e. a Borel
  subgroup $B$ of $G$, a Cartan subgroup $T \subset B$ and non-zero
  nilpotent generators of the nilpotent radical of $\on{Lie}(B)$
  corresponding to the simple roots. Then there is a unique $\tau$
  preserving these data \cite{AV}.} We recall the following properties
of $\tau$:

\begin{enumerate}

\item For any representation $R$ of $\LG$, the representation
  $\tau^*(R)$ is isomorphic to $R^*$.

\item If $\LG$ is abelian, then $\tau(g) = g^{-1}$.

\item If $\LG=PGL_n$, then we can take $\tau(g) = (g^t)^{-1}$.

\item If $\LG$ is simple, then $\tau$ gives rise to the automorphism
  $\sigma$ of the Dynkin diagram that sends the node $a$ to
  $\sigma(a)$, where $\alpha_{\sigma(a)} = -w_0(\alpha_a)$.

\end{enumerate}

Now we can describe explicitly the action of $\nu$ on
$\on{Op}_{\LG}(X)$. We will assume that in the definition of
$\LG$-opers, we will use a Borel subgroup $\LB$ stable under $\tau$.
Given an $\LG$-oper $({\mc F},\nabla,{\mc F}_{\LB})$, we obtain a new
$\LG$-oper $(\tau({\mc F}),\tau(\nabla),\tau({\mc F}_{\LB}))$. It is
easy to see that both the oper bundle ${\mc F}$ and its Borel
reduction ${\mc F}_{\LB}$ are stable under $\tau$. Therefore, we
obtain an involution on the space of opers, which we denote by
$\tau$.

Recall that $\on{Op}_{\LG}(X)$ consists of all holomorphic connections
$\nabla$ on a particular $\LG$-bundle ${\mc F}_0$ (the Borel reduction
of ${\mc F}_0$ is uniquely determined by $\nabla$). The bundle
${\mc F}_0$ is preserved by $\tau$. Hence $\tau$ sends the flat bundle
$({\mc F}_0,\nabla)$ to $({\mc F}_0,\tau(\nabla))$.

\begin{proposition}    \label{nu}
  The action of $\nu$ on $\on{Op}_{\LG}(X)$ coincides with the action
  of the Chevalley involution $\tau$.
\end{proposition}

We sketch the proof in Section \ref{proofnu}; it relies on Theorem 5.4
of \cite{FG}.

For an $\LG$-oper $\la$, we denote $\tau(\la)$ by $\la^*$. The above
property (1) of $\tau$ 
implies that the flat $\LG$-bundles ${\mc E}(\la)$ and ${\mc
  E}(\la^*)$ corresponding to $\la$ and $\la^*$, respectively, have
the following property:
\begin{equation}    \label{dual}
R_{{\mc E}(\la^*)} \simeq R^*_{{\mc E}(\la)}, \qquad \forall \; R \in
\on{Rep} \LG.
\end{equation}
(see formula \eqref{assocbdle} for the notation).

Likewise, applying $\tau$ to an anti-holomorphic $\LG$-oper $\mu$, we
obtain another anti-holomor\-phic $\LG$-oper, which we denote by
$\mu^*$. Clearly, $\ol{\la^*} = (\ol\la)^*$, i.e. it doesn't matter in
which order we apply $*$ and the complex conjugation, so we'll just
use the notation $\ol\la^*$.

\subsection{Global anti-holomorphic differential operators on
  $\Bun_G$}

We define an anti-holomorphic counterpart of the space of
$\LG$-opers: the space $\ol{\on{Op}}_{\LG}(X)$ of {\em
  anti-holomorphic} $\LG$-opers on $X$.

\begin{remark}
Instead of considering anti-holomorphic objects on a variety $Y$, we
may consider holomorphic objects on the complex conjugate variety
$\ol{Y}$ (see Section \ref{cpxconj} for more details).
\end{remark}

By definition, an anti-holomorphic $\LG$-oper $\mu$ on $X$ is a triple
$(\ol{\mc F},\ol\nabla,\ol{\mc F}_{\LB})$, where $\ol{\mc
  F}$ is an anti-holomorphic $\LG$-bundle on $X$, $\ol\nabla$ is an
anti-holomorphic connection on $\ol{\mc F}$, and $\ol{\mc F}_{\LB}$ is
an anti-holomorphic reduction of $\ol{\mc F}$ to a Borel subgroup $\LB
\subset \LG$ which satisfies the obvious analogue of the oper
transversality condition (as defined in \cite{BD:opers}).

We will denote by ${\mc E}(\mu)$ the flat $\LG$-bundle $(\ol{\mc
  F},\ol\nabla)$ on $X$ obtained by forgetting the Borel reduction in
$\mu$. As in the holomorphic case, under our assumption that $\LG$ is
of adjoint type, the map $\mu \mapsto {\mc E}(\mu)$ is an
embedding, and the image of $\ol{\on{Op}}_{\LG}(X)$ is the space of
all anti-holomorphic connections on the anti-holomorphic $\LG$-bundle
$\ol{\mc F}_0$ obtained by complex conjugation of the holomorphic
$\LG$-bundle ${\mc F}_0$ discussed above.

We will also denote by $\rho_\mu$ the monodromy representation
$\pi_1(X) \to \LG$ corresponding to ${\mc E}(\mu)$.

Note that the complex conjugate of every $\la \in \on{Op}_{\LG}(X)$,
which we denote by $\ol\la$, is in $\ol{\on{Op}}_{\LG}(X)$, and vice
versa.

\medskip

Now we discuss the anti-holomorphic counterpart of the
Beilinson-Drinfeld construction.

We apply the construction of Section \ref{gen} to the case of the line
bundle $K^{1/2}$ on $\Bun_G$. Then we obtain an anti-holomorphic line
bundle $\ol{K}^{1/2}$ on $\Bun_G$ and the sheaf $\ol{\mc D}_G$ of
anti-holomorphic differential operators on it. Denote by $\ol{D}_G$
the corresponding commutative subalgebra of the algebra of its global
sections. The Beilinson-Drinfeld isomorphism \eqref{BDisom} implies
\begin{equation}    \label{BDisomanti}
\ol{D}_G \simeq \ol{\mcC}_{\LG},
\end{equation}
where $\ol{\mcC}_{\LG}$ is the algebra of regular functions on the
space $\ol{\on{Op}}_{\LG}(X)$ of anti-holomorphic opers on $X$
introduced in Section \ref{opers}.

Now, given $\mu \in \ol{\on{Op}}_{\LG}(X)$, we construct the maximal
ideal $\ol{\fm}_\mu$ in $\ol{D}_G$, the left ideal $\ol{I}_\la =
\ol{\mc D}_G \cdot \ol{\fm}_\mu$ in $\ol{\mc D}_G$, and the $\ol{\mc
  D}_G$-module
$$
\ol\Delta_\mu = \ol{\mc D}_G/\ol{I}_\mu.
$$
Applying the anti-holomorphic version of the Riemann-Hilbert
correspondence to $\ol\Delta_\mu^0 = \ol{K}^{-1/2} \otimes
\ol\Delta_\mu$, we obtain a perverse sheaf ${\mc K}_{{\mc
    E}(\mu)}$. Theorem \ref{Heckeeig} then implies the following
result (in which we use the same notation ${\mc E}(\mu)$ for the
locally constant sheaf on $X$ of horizontal sections of the flat
$\LG$-bundle ${\mc E}(\mu)$).

\begin{theorem}
  The sheaf ${\mc K}_{{\mc E}(\mu)}$ is a Hecke eigensheaf with
  respect to ${\mc E}(\mu)$.
\end{theorem}

Let $U$ be an open dense subset of $\Bun_G$ such that ${\mc
  K}_{{\mc E}(\mu)}{}|_U$ is locally constant sheaf on $U$, which
we denote by ${\mc S}_\mu$. We have the following analogue of
Corollary \ref{eig}.

\begin{corollary}\label{eig1}
\hfill
\begin{enumerate}
\item The locally constant sheaf ${\mc E}(\mu)$ and the corresponding
  monodromy representation $\rho_\mu$ are uniquely determined by
  $\ol\Delta_\mu^0$ and ${\mc K}_{{\mc E}(\mu)}$.

\item If $\ol\Delta^0_\mu$ is regular holonomic and irreducible, then
  ${\mc E}(\mu)$ and $\rho_\mu$ are uniquely determined by the
  corresponding locally constant sheaf ${\mc S}_\mu$.
\end{enumerate}
\end{corollary}

\section{The spectrum and opers}    \label{specopers}

As in Section \ref{opsH}, let $\what{V}$ be the space of smooth
sections of $\Omega^{1/2}$ on $\Bunc_G$. The algebra ${\mc A} = D_G
\otimes \ol{D}_G$ acts on $\what{V}$. Given a homomorphism $\chi: {\mc
  A} \to \CC$, denote by $V_\chi \subset \what{V}$ the
$\chi$-eigenspace of ${\mc A}$. It follows from the isomorphisms
\eqref{BDisom} and \eqref{BDisomanti} that every $\chi$ is a pair
$(\la,\mu)$, where $\la$ is a holomorphic $\LG$-oper and $\mu$ is an
anti-holomorphic $\LG$-oper on $X$.

Let $\Xi$ be the set of $\chi = (\la,\mu)$ such that $V_\chi \neq
0$. We would like to show that (1) the projection onto the first
factor $\chi \mapsto \la$ from $\Xi$ to $\on{Op}_{\LG}(X)$ is an
embedding, and (2) the image $\Sigma$ of this embedding is contained
in the subset $\on{Op}_{\LG}(X)_{{\mathbb R}}$ of $\LG$-opers $\la$ on
the curve $X$ satisfying the following reality condition: the
$C^\infty$ flat $\LG$-bundles on $X$ corresponding to $\la$ and to its
complex conjugate oper $\ol\la$ are isomorphic.

Let $\mcA_{\mbb R} \subset \mcA$ be the ${\mbb R}$-subalgebra of
symmetric operators. In order to establish assertion (1), we need
additional information; namely, we need to know that the eigenvalues
of the operators corresponding to elements of $\mcA_{\mbb R}$ are real
numbers. In the abelian case (see Section \ref{abelian}) these
operators are obviously self-adjoint, so we can prove both assertions
directly. But in the non-abelian case we need another argument, and
this is where we need to know that the eigenfunction is in
$\mcH=L^2(\Bun_G)$. More precisely, we need to know that it belongs to
a dense subspace of $\mcH$ on which the elements of $\mcA_{\mbb R}$
are symmetric operators. For the subspace $S(\mcA) \subset \mcH$ (see
Definition \ref{es1}), this follows from our Conjecture
\ref{first}. In this section we will assume this conjecture and will
use it (as well as Conjecture \ref{RS}) to derive the above properties
(1) and (2) of the spectrum $\on{Spec}_{\mcA}(\mcH)$ of the
self-adjoint extension of $\mcA_{\mbb R}$ to $S(\mcA)$.\footnote{In
  the case of $G=SL_2$, $X=\pone$, and $|S|=4$, an alternative
  argument is given in the proof of Proposition \ref{l2prop} (it
  can probably be generalized to $|S|>4$).}

However, according to Conjecture \ref{basic}, $\on{Spec}_{\mcA}(\mcH)$
coincides with the above set $\Xi$. Therefore we do expect that the
above properties (1) and (2) also hold for $\Xi$.

\subsection{From the spectrum to opers: step I}

We start with the following observation.

\begin{lemma}
The operators
\begin{equation}    \label{sym}
D+\ol{\nu(D)} \quad \on{and} \quad (D-\ol{\nu(D)})/i, \qquad D \in
D_G,
\end{equation}
acting on the space $V$ of compactly supported sections of
$\Omega^{1/2} = K^{1/2} \otimes \ol{K}^{1/2}$ are symmetric.
\end{lemma}

\begin{proof} Follows from Lemma \ref{formaladj} and Proposition
  \ref{nu}.
\end{proof}

Let $\mcA_{\mbb R}$ be the ${\mbb R}$-span of operators \eqref{sym} in
$\mcA$.

Conjecture \ref{first} states that there is a canonical way to extend
the domain $V \subset \mcH$ to a bigger domain $S(\mcA) \subset \mcH$
so that all operators $A \in \mcA_{\mbb R}$ are essentially
self-adjoint on $S(\mcA)$. Furthermore, Conjecture \ref{second} states
that their joint spectrum $\on{Spec}_{\mcA}(\mcH)$ is discrete. (We
shall prove this in the abelian case in Section \ref{abelian} and in
the case of $\pone$ and $N=4$ points in Part II.) Therefore we can
identify $\on{Spec}_{\mcA}(\mcH)$ with a subset
$\on{Op}_{\LG,\mcH}(X)$ of pairs
$$
(\la,\mu) \in
\on{Op}_{\LG}(X) \times \ol{\on{Op}}_{\LG}(X)
$$
which is discrete with respect to the analytic topology on the right
hand side. For every $(\la,\mu) \in \on{Op}_{\LG,\mcH}(X)$ we have a
non-zero solution $\Phi \in S(\mcA)$ of the equations
\begin{equation}    \label{lamu}
(D-\la(D))\Phi = 0, \quad D \in D_G; \qquad (\ol{D} - \mu(\ol{D}))\Phi
= 0, \quad \ol{D} \in \ol{D}_G.
\end{equation}

Assuming Conjectures \ref{first} and \ref{second}, as well as
Conjecture \ref{RS} below, we will now prove Conjecture \ref{third}
which we restate as follows:

\begin{theorem}    \label{incl}
If $(\la,\mu) \in \on{Spec}_{\mcA}(\mcH)$, then (i) $\mu=\ol\la^*$; and
(ii) the monodromy representation
$\rho_\la:\pi_1(X) \to \LG$ associated to $\la$ is isomorphic to its
complex conjugate representation $\ol\rho_\la$.
\end{theorem}


Thus, we show that the spectrum $\on{Op}_{\LG,\mcH}(X)$ can be
identified with a subset of the set of $\LG$-opers $\la$ on $X$
satisfying the reality condition $\rho_\la \simeq \ol\rho_\la$. This
is the set $\on{Op}_{\LG}(X)_{{\mathbb R}}$ introduced in Section
\ref{main}.

\medskip

The first step in the proof of Theorem \ref{incl} is to use the fact
that the eigenvalues of self-adjoint operators are real numbers. Lemma
\ref{formaladj} and Proposition \ref{nu} imply the following result.

\begin{lemma}    \label{ollast}
If $(\la,\mu) \in \on{Op}_{\LG,\mcH}(X)$, then $\mu = \ol\la^*$.
\end{lemma}

We have established property (i) of Theorem \ref{incl}. Thus, we can
identify $\on{Op}_{\LG,\mcH}(X)$ with a subset of $\on{Op}_{\LG}(X)$.

\subsection{From the spectrum to opers: step II}

The second step of the construction is to choose an open dense subset
$U$ of $\Bunc_G$ such that

\begin{enumerate}
\item $K^{-1/2} \otimes \Delta_\la|_U$ is a holomorphic vector bundle
  with a holomorphic flat connection ${\mc V}_\la$;

\item $\ol{K}^{-1/2} \otimes \ol\Delta_{\ol\la^*}|_U$ is an
  anti-holomorphic vector bundle with an anti-holomorphic flat
  connection ${\mc V}'_{\ol\la^*}$.
\end{enumerate}

According to Proposition \ref{Her} and Remark \ref{Her1}, the
existence of a non-zero single-valued solution to the equations
\eqref{lamu} with $\mu=\ol\la^*$ gives rise to an embedding of
$C^\infty$ flat vector bundles
\begin{equation}    \label{emb0}
({\mc C}^\infty_U,d) \hookrightarrow {\mc V}_\la \otimes {\mc
    V}'_{\ol\la^*}
\end{equation}
on $U$. We claim that ${\mc V}'_{\ol\la^*}$ can be described in terms
of ${\mc V}_\la$.

First, it is clear from the definition that there is an isomorphism of
anti-holomorphic flat vector bundles
\begin{equation}    \label{Vbar}
{\mc V}'_{\ol\la} \simeq \ol{\mc V_\la}, \quad \forall \; \la \in
\on{Op}_{\LG}(X).
\end{equation}

Second, we have the following isomorphism.

\begin{lemma}    \label{V*}
We have an isomorphism of flat holomorphic vector bundles
\begin{equation}    \label{Vdual}
{\mc V}_{\la^*} \simeq ({\mc V}_\la)^*, \quad \forall \; \la \in
\on{Op}_{\LG}(X).
\end{equation}
\end{lemma}

The proof is given in Section \ref{proof*}.

Combining \eqref{Vbar} and \eqref{Vdual}, we obtain an isomorphism of
anti-holomorphic flat vector bundles
\begin{equation}    \label{bardual}
{\mc V}'_{\ol\la^*} \simeq \ol{{\mc V}_\la}^*.
\end{equation}
In words: the flat vector bundle given by the restriction of the
anti-holomorphic $D$-module $\ol{K}^{-1/2} \otimes
\ol\Delta_{\ol\la^*}$ to $U$ is isomorphic to the flat vector bundle
obtained from ${\mc V}_\la$ by applying both the complex conjugation
and taking the dual (it does not matter in which order we apply these
operations).

Therefore, the embedding \eqref{emb0} becomes an embedding of
$C^\infty$ flat vector bundles
\begin{equation}    \label{emb}
({\mc C}^\infty_U,d) \hookrightarrow {\mc V}_\la \otimes (\ol{{\mc V}_\la})^*.
\end{equation}
Hence we have an embedding of the corresponding locally constant
sheaves
\begin{equation}    \label{emb1}
\underline\CC_U \hookrightarrow {\mc S}_{\la} \otimes (\ol{{\mc
    S}_\la})^*,
\end{equation}
where $\underline\CC_U$ is the constant sheaf on $U$, and ${\mc
  S}_{\la}$ and $(\ol{{\mc S}_\la})^*$ denote the local systems on $U$
(with respect to the analytic topology) corresponding to ${\mc V}_\la$
and $(\ol{{\mc V}_\la})^*$, respectively.

The following conjecture was formulated by Beilinson and Drinfeld in
\cite{BD} as an open question.

\begin{conjecture}    \label{RS}
$\Delta^0_\la$ is irreducible and has regular singularities on each
connected component of $\Bun_G$.
\end{conjecture}

\begin{remark}
  According to D. Gaitsgory, for $G=SL_n$ this conjecture follows by
  comparing the  Beilinson-Drinfeld construction with the
  construction of \cite{Dr1,FGV}, if one assumes the validity of the
  quasi-theorems in \cite{Gai:outline}.

\end{remark}

Assuming Conjecture \ref{RS}, we obtain that ${\mc V}_\la$ and ${\mc
  S}_{\la}$ are irreducible. Hence formula \eqref{emb1} is equivalent
to the existence of an isomorphism of local systems on the open dense
subset $U$:
\begin{equation}    \label{isoS}
{\mc S}_{\la} \simeq {\mc S}_{\ol\la}
\end{equation}
According to Corollaries \ref{eig} and \ref{eig1}, this implies
the existence of an isomorphism of the corresponding monodromy
representations
\begin{equation}    \label{rhola}
\rho_\la \simeq \ol\rho_\la
\end{equation}
(indeed, $\rho_{\ol\la} = \ol\rho_\la$). Thus, we have established
property (ii) of Theorem \ref{incl}.

Note that \eqref{rhola} is equivalent to the existence of an
isomorphism of the $C^\infty$ flat $\LG$-bundles on $X$ corresponding
to ${\mc E}(\la)$ and ${\mc E}(\ol\la)$.

Thus, Conjectures \ref{first}, \ref{second}, and \ref{RS} imply
Theorem \ref{incl} (which is in turn equivalent to Conjecture
\ref{third}).



\begin{remark}
  Note that the property ${\mc E}(\la) \simeq {\mc E}(\ol\la)$ (as
  $C^\infty$ flat $\LG$-bundles) is equivalent to the existence of an
  embedding
  \begin{equation} \label{embX} ({\mc C}^\infty_X,d) \hookrightarrow R_\la
    \otimes \ol{R}^*_\la, \qquad \forall \; R \in \on{Rep} \LG
\end{equation}
(here we use the notation introduced in formula
\eqref{assocbdle}). Comparing it with formula \eqref{emb}, one could
say that we have related the existence of a kind of ``weird pairing''
(it is neither on ${\mc V} \otimes \ol{\mc V}$ nor on ${\mc V} \otimes
{\mc V}^*$, but rather on ${\mc V} \otimes \ol{\mc V}^*$) on an open
dense subset $U$ of $\Bun_G$ and on the curve $X$ itself.
\end{remark}

Finally, let us discuss the inverse map from the set
$\on{Op}_{\LG}(X)_{{\mathbb R}}$ of $\LG$-opers satisfying the
property ${\mc E}(\la) \simeq {\mc E}(\ol\la)$ (as $C^\infty$ flat
$\LG$-bundles) to the spectrum $\on{Spec}_{\mcA}(\mcH)$. So, let $\la$
be an $\LG$-oper $\on{Op}_{\LG}(X)_{{\mathbb R}}$. Reversing the above
argument, we obtain an embedding \eqref{emb}. This implies the
existence of a single-valued solution $\Phi$ to the equations
\eqref{lamu} with $\mu=\ol\la^*$ on the open subset $U \subset
\Bun_G$. If this solution is square-integrable (i.e. belongs to
$\mcH = L^2(\Bun_G)$) then $\la$ belongs to the spectrum
$\on{Spec}_{\mcA}(\mcH)$.

Thus, the existence of the inverse map $\on{Op}_{\LG}(X)_{{\mathbb R}}
\to \on{Spec}_{\mcA}(\mcH)$ (and hence a bijection between the two
sets) depends on whether the above solutions $\Phi$ are square
integrable for all $\la \in \on{Spec}_{\mcA}(\mcH)$. Question
\ref{question} is therefore equivalent to the question whether this is
always so. In the two cases in which we have proved our Conjectures
\ref{first}--\ref{third} ($G=GL_1$ in the next section and $G=SL_2,
X=\pone, |S|=4$ in Part II) the answer is affirmative.

\section{The abelian case}    \label{abelian}

In this section, we describe explicitly the spectrum of the algebra
${\mc A}$ in the case of the group $G=GL_1$ and prove Conjectures
\ref{first}--\ref{third}. Furthermore, we show that the spectrum is in
one-to-one correspondence with the $GL_1$-opers (which are in this
case holomorphic connections on the trivial line bundle on $X$) with
{\em real monodromy}. And we show that the Fourier harmonics form a
basis of eigenfunctions of the algebra ${\mc A}$.

The stack $\Bun_{GL_1}(X)$ of $GL_1$-bundles, or equivalently, line
bundles on a smooth projective curve $X$ is the quotient of the Picard
variety $\Pic(X)$ by the trivial action of the multiplicative group
${\mathbb G}_m=GL_1$, which is the group of automorphisms of every
line bundle on $X$. The Picard variety $\Pic(X)$ is in this case a
fine moduli space of line bundles on $X$. It is a union of components
labeled by integers (corresponding to the degree of the line bundles).

The Hilbert space $\mcH$ that we consider in this case is the space of
$L^2$ functions on the neutral component $\Pic^0(X)$ of $\Pic(X)$,
which is the Jacobian variety of $X$.

\subsection{The case of an elliptic curve}

Let's start with the case of the elliptic curve
$$
X = E_i \simeq \CC/(\Z + \Z i).
$$
This is already a representative example, and we can make contact with
the classical Fourier analysis.

Choosing a reference point $p_0$ on $E_i$ enables us to identify
$\Pic^0(E_i)$ with $E_i$ via the Abel-Jacobi map; namely, we map a
point $p \in E_i$ to the degree 0 line bundle
$\OO(p-p_0)$. Furthermore, we identify $E_i$ with $\CC/(\Z + \Z i)$ by
sending the point $p_0$ to $0 \in \CC/(\Z + \Z i)$. Thus $\Pic^0(E_i)$
is identified with $\CC/(\Z + \Z i)$.
 
Under this identification, the algebra $D_{GL_1}$
(resp. $\ol{D}_{GL_1}$) of global holomorphic (resp. anti-holomorphic)
differential operators on $\Pic^0(E_i)$ becomes the algebra of
constant holomorphic (resp. anti-holomorphic) differential operators
on $E_i$:
$$
D_{GL_1} = \CC[\pa_z], \qquad \ol{D}_{GL_1} = \CC[\pa_{\ol{z}}],
$$
where $z$ is the natural coordinate on $\CC/(\Z + \Z i)$.

Note that in this case the involution $\nu$ on $D_G$ (see Section
\ref{cananti}) sends $\pa_z$ to $-\pa_z$, and we have the standard
formula $\pa_z^\star = - \pa_{\ol{z}}$ (see equation \eqref{Dstar}).

The commutative algebra ${\mc A}$ is the tensor product
$$
{\mc A} = D_{GL_1} \otimes \ol{D}_{GL_1},
$$
and
$$
{\mc A}_{{\mathbb R}} = {\mathbb
  R}[\pa_z-\pa_{\ol{z}},(\pa_z+\pa_{\ol{z}})/i].
$$
In this case, $V$ is the space of smooth functions on $E_i$, $S({\mc
  A}) = V$, and ${\mc A}_{{\mathbb R}}$ is essentially self-adjoint on
$V$ (see Example \ref{exam}(1)).

The eigenfunctions of ${\mc A}_{{\mathbb R}}$ are the standard Fourier
harmonics $f_{m,n}$ on $\CC/(\Z + \Z i)$ given by the formula
\begin{equation}    \label{fnm}
f_{m,n} = e^{2\pi i mx} \cdot e^{2\pi i ny}, \qquad m,n \in \Z,
\end{equation}
where we set $z=x+yi$. We rewrite them in terms of $z$ and $\ol{z}$:
\begin{equation}    \label{fnm1}
f_{m,n} = e^{\pi z (n+im)} \cdot e^{-\pi \ol{z}(n-im)},
\end{equation}
to find the eigenvalues of $\pa_z$ and $\pa_{\ol{z}}$ on
$f_{m,n}$. They are equal to $\pi(n+im)$ and $-\pi(n-im)$,
respectively. Let us recast these eigenvalues in terms of the
corresponding $GL_1$-opers.

By definition, a $GL_1$-oper on $X$ is a holomorphic connection on the
trivial line bundle on $X$. The space of such connections is
canonically isomorphic to the space of holomorphic one-forms on
$X$. We write such a connection as
\begin{equation}    \label{GL1op}
\lambda(a) = \nabla_z - a \, dz, \qquad a \in \CC,
\end{equation}
where $\nabla_z = dz \otimes \pa_z$ is the trivial holomorphic
connection on the trivial line bundle and $a \, dz$ is a holomorphic
one-form (the reason for the minus sign will become clear
below). Together with $\nabla_{\ol{z}} = d\ol{z} \otimes
\pa_{\ol{z}}$, it gives rise to the flat connection
\begin{equation}    \label{GL1op1}
\nabla(a) = d - a \, dz, \qquad a \in \CC
\end{equation}
on the trivial line bundle on $E_i$.

The isomorphism \eqref{BDisom} specializes in the case of $GL_1$ to
the isomorphism
$$
\on{Spec} D_{GL_1} \simeq \on{Op}_{GL_1}(E_i)
$$
under which the oper \eqref{GL1op} (and the connection \eqref{GL1op1})
corresponds to the eigenvalue $a$ of $\pa_z$. Note that the Chevalley
involution $\tau$ acting on the $GL_1$-oper \eqref{GL1op} sends
$\lambda(a) \mapsto \lambda(-a)$ in accordance with the action of
$\nu$ sending $\pa_z \mapsto -\pa_z$ (see Proposition \ref{nu}). In
other words, we have
\begin{equation}    \label{last}
\la(a)^* = \la(-a).
\end{equation}

Next, consider the space $\ol{\on{Op}}_{GL_1}(E_i)$ of
anti-holomorphic $GL_1$-opers on $E_i$. These are anti-holomorphic
connections
\begin{equation}    \label{antiGL1op}
\mu(b) = \nabla_{\ol{z}}-b \, d\ol{z}, \qquad b \in \CC,
\end{equation}
on the trivial line bundle on $E_i$, which together with $\nabla_z$
give rise to the flat connections
\begin{equation}    \label{antiGL1op1}
\ol\nabla(b) = d - b \, d\ol{z}, \qquad b \in \CC.
\end{equation}
The isomorphism \eqref{BDisomanti} specializes to
$$
\on{Spec} \ol{D}_{GL_1} \simeq \ol{\on{Op}}_{GL_1}(E_i),
$$
under which the oper \eqref{antiGL1op} (and the connection
\eqref{antiGL1op1}) corresponds to the eigenvalue $b$ of
$\pa_{\ol{z}}$.

We have found above that the eigenvalues of $\pa_z$ and $\pa_{\ol{z}}$
on $f_{m,n} \in L^2(\Bun_{GL_1})$ are $\pi(n+im)$ and $-\pi(n-im)$,
respectively, where $m,n \in \Z$. The corresponding holomorphic and
anti-holomorphic opers are
\begin{equation}    \label{lamn}
\lambda_{m,n} = \nabla_z - \pi(n+im) \; dz
\end{equation}
\begin{equation}
\mu_{m,n} = \nabla_{\ol{z}} + \pi(n-im) \; d\ol{z},
\end{equation}
We have
$$
\mu_{m,n} = \ol{\la}_{m,n}^*
$$
(see formula \eqref{last} for the action of $*$ on $\la$). This
is the statement of Lemma \ref{ollast} in this case.

It turns out that the $GL_1$-opers $\lambda_{m,n}$ given by formula
\eqref{lamn} are precisely the $GL_1$-opers with real monodromy
(i.e. the corresponding monodromy representation of $\pi_1(E_i)$ takes
values in $\R^\times \subset \CC^\times$).

\begin{lemma}    \label{realmono}
The connection $\nabla(a)$ (see formula \eqref{GL1op1}) on the trivial
line bundle on $E_i = \CC(\Z + \Z i)$ has real monodromy if and only
if $a = \pi(n+im)$, where $m,n \in \Z$.
\end{lemma}

\begin{proof}
Let ${\mc L}_a$ be the trivial line bundle on $E_i$ with the
connection $\nabla(a) = d-adz, a \in \CC$. The monodromy
representation $\pi_1(E_i) \to \CC^\times$ associated to the flat line
bundle ${\mc L}_a$ takes values in $\R^\times$ if and only if the
monodromy representation associated to the flat line bundle ${\mc L}_a
\otimes \ol{{\mc L}_a}{}^{-1}$ is trivial. The latter means that
$e^{a-\ol{a}}=1$ and $e^{i(a+\ol{a})}=1$, which is equivalent to $a =
\pi(n+im)$, where $m,n \in \Z$.
\end{proof}

Now we prove a stronger version of our main Conjectures
\ref{first}--\ref{third} in the case $G=GL_1, X=E_i$:

\begin{theorem}    \label{specell}
  The spectrum of the algebra ${\mc A}$ on $\mcH = L^2(\Pic^0(E_i))$
  is in one-to-one correspondence with the set of $GL_1$-opers on
  $E_i$ with real monodromy (i.e. the monodromy takes values in
  $\R^\times \subset \CC^\times$). The eigenfunctions are the Fourier
harmonics $f_{m,n}, m, n \in \Z$.
\end{theorem}

We give two proofs of this theorem.

\medskip

\noindent{\em First proof.} We use an explicit formula for the
eigenfunction $f_{m,n}$ and the corresponding eigenvalues of $\pa_z$
and $\pa_{\ol{z}}$ on $f_{m,n}$ (they are $\pi(n+im)$ and
$-\pi(n-im)$, respectively, where $m,n \in \Z$), which we found
above. Then we apply Lemma \ref{realmono}.

\medskip

\noindent{\em Second proof.} This is a more conceptual proof, in which
we do not use an explicit formula for the eigenfunctions and
eigenvalues.

Suppose that there is an eigenfunction of the operators $\pa_z$ and
$\pa_{\ol{z}}$ with eigenvalues $a$ and $b$, respectively. We then use
the factorization formula \eqref{locsum} which in this case expresses
(locally) this eigenfunction as the product of horizontal sections of
holomorphic and anti-holomorphic flat bundles (in fact, this is
formula \eqref{fnm1}, but we do not want to rely on an explicit
formula for the eigenfunctions). There is no summation in this case
because these flat bundles have rank one. In fact, they are defined by
the connections $\nabla(a)$ (formula \eqref{GL1op1}) and
$\ol\nabla(b)$ (formula \eqref{antiGL1op}), respectively, on the
trivial line bundle on $E_i$. This means that the product of non-zero
solutions $\Phi_a(z)$ and $\Psi_b(\ol{z})$ of the equations
$$
(\pa_z - a) \Phi_a(z) = 0 \qquad \on{and} \qquad (\pa_{\ol{z}} - b)
\Psi_b(\ol{z}) = 0
$$
is single-valued, and is an eigenfunction of ${\mc A}$.

Since the operators $(\pa_z-\pa_{\ol{z}})$ and
$(\pa_z+\pa_{\ol{z}})/i$ are essentially self-adjoint, we find that
$b=-\ol{a}$. Therefore $\Psi_b(\ol{z}) = \ol{\Phi_a(z)}{}^{-1}$ (up to
a scalar). Thus, we find that $a$ must be such that the function
$$
f_a(z,\ol{z}) = \Phi_a(z) \ol{\Phi_a(z)}{}^{-1}
$$
is a single-valued function on $E_i$. This means that the monodromy of
$\Phi_a(z)$ (which takes values in $\CC^\times$) must coincide with
the monodromy of $\ol{\Phi_a(z)}$. Equivalently, the monodromy of
$\Phi_a(z)$ takes values in $\R^\times \subset \CC^\times$. Then we
apply Lemma \ref{realmono} to describe explicitly the values of $a$
satisfying this property; namely, $a = \pi(n+im)$, where $m,n \in
\Z$. Note that we did all this without using an explicit formula for
the eigenfunctions $f_{m,n}$.\qed

\subsection{General elliptic curves}

Now we consider the general elliptic curve $E_\tau = \CC/(\Z+\Z\tau)$,
where $\on{Im}\tau>0$. As before, we identify the Jacobian
$\Pic^0(E_\tau)$ with $E_\tau$ using a reference point $p_0$
corresponding to $0 \in \CC$. The algebras of differential operators
are the same as in the case $\tau=i$. Their joint eigenfunctions in
$L^2(\Pic^0(E_\tau))$ are the Fourier harmonics
\begin{equation}    \label{fmntau}
f^\tau_{m,n} = e^{2\pi
  im(z\ol\tau-\ol{z}\tau)/(\ol\tau-\tau)} \cdot e^{2\pi
  in (z-\ol{z})/(\tau-\ol\tau)}, \qquad m,n \in \Z
\end{equation}
(compare with formula (2.24) from \cite{F:analyt}).

The corresponding holomorphic $GL_1$-opers are
$$
_\tau\lambda_{m,n} = \nabla_z - 2\pi i
\frac{n-m\ol\tau}{\tau-\ol\tau} \qquad m,n \in \Z
$$
(compare with formula (2.27) of \cite{F:analyt}). One checks by a
direct computation that these are precisely the holomorphic
$GL_1$-opers on $E_\tau$ with real monodromy.

\begin{remark}
  One simplification that occurs in the case of elliptic curves is
  that $\Pic^0(E_\tau)$ is isomorphic to the curve $E_\tau$
  itself. For curves of higher genus this is not so. However, as we
  will see below, the isomorphism $H^1(X,\CC) \simeq
  H^1(\Pic^0(X),\CC)$ suffices for our purposes.
\end{remark}

\subsection{General curves}

We generalize these results to a curve $X$ of an arbitrary genus
$g>0$. The Jacobian $\Pic^0(X)$ is then a real $2g$-dimensional torus
(see, e.g., \cite{GH})
\begin{equation}    \label{Pic}
\Pic^0(X) \simeq H^0(X,\Omega^{1,0})^*/H_1(X,\Z).
\end{equation}
We give an explicit formula for the Fourier harmonics in
$L^2(\Pic^0(X))$ following \cite{F:analyt}, Sect. 2.4. Recall the
Hodge decomposition
\begin{equation}    \label{Hodge}
H^1(X,\CC) = H^0(X,\Omega^{1,0}) \oplus H^0(X,\Omega^{0,1}) =
H^0(X,\Omega^{1,0}) \oplus \ol{H^0(X,\Omega^{1,0})}.
\end{equation}
It enables us to identify $H^0(X,\Omega^{1,0})$, viewed as an
$\R$-vector space, with $H^1(X,\R)$ by the map
\begin{equation}    \label{ident vsp}
\omega \in H^0(X,\Omega^{1,0}) \mapsto \omega + \ol\omega.
\end{equation}
Under this identification, any class $c \in H^1(X,\R)$ is represented
by a unique real-valued {\em harmonic} one-form $\omega_c +
\ol{\omega}_c$, where $\omega_c \in H^0(X,\Omega^{1,0})$.

Formula \eqref{Pic} then implies that, as a real torus,
\begin{equation}   \label{Picident}
\Pic^0(X) \simeq H^1(X,\R)^*/H_1(X,\Z),
\end{equation}
where $H_1(X,\Z)$ is embedded into $H^1(X,\R)^*$ by sending $\beta \in
H_1(X,\Z)$ to the linear functional on $H^1(X,\R)$
\begin{equation}    \label{embed}
H^1(X,\R) \; \ni \; c \quad \mapsto \quad \int_\beta c = \int_\beta
(\omega_c + \ol{\omega}_c).
\end{equation}

Now, for each $\gamma \in H^1(X,\Z)$ we denote by $\varphi_\ga$ the
harmonic representative of its image in $H^1(X,\R)$, i.e.
\begin{equation}    \label{omegaga}
\varphi_\ga = \omega_\ga + \ol\omega_\ga, \qquad \omega_\ga \in
H^0(X,\Omega^{1,0}).
\end{equation}

Consider $\varphi_\ga$ as a linear functional on the dual vector space
$H^1(X,\R)^*$. Then the Fourier harmonics on $\Pic^0(X)$, which we
identify with a real $2g$-dimensional torus via formula
\eqref{Picident}, can be written as
\begin{equation}    \label{Fourier}
f_\ga = e^{2 \pi i \varphi_\ga}, \qquad \ga \in H^1(X,\Z).
\end{equation}
This formula is well-defined because
$$
\varphi_\ga(\beta) \in \Z, \qquad \forall \ga \in H^1(X,\Z), \quad
\beta \in H_1(X,\Z).
$$

It is clear that the Fourier harmonics $f_\ga, \ga \in H^1(X,\Z)$,
form an orthogonal basis of the Hilbert space $L^2(\Pic^0(X))$. We
claim that they form an eigenbasis of our algebra
$$
{\mc A} = D_{GL_1} \otimes \ol{D}_{GL_1}
$$
of global differential operators on $\Pic^0(X)$.

Indeed, since we have an abelian group structure on $\Pic^0(X)$, the
tangent bundle to $\Pic^0(X)$ is isomorphic to the trivial vector
bundle with the fiber being the tangent space at the point
corresponding to the trivial line bundle. This space is isomorphic to
$H^1(X,\OO_X) \simeq H^0(X,\Omega^{1,0})^*$, which can be viewed as
the space of translation vector fields on $\Pic^0(X)$. Since
$\Pic^0(X)$ is compact, we obtain that $D_{GL_1}$ is isomorphic to the
symmetric algebra of the space of translation vector fields,
i.e. $\on{Sym} H^1(X,\OO_X)$. Likewise, we find that $\ol{D}_{GL_1}
\simeq \on{Sym} H^1(X,\OO_X)$. Hence ${\mc A}$ is freely generated by
the $C^\infty$ translation vector fields on $\Pic^0(X)$. Therefore the
Fourier harmonics $f_\ga, \ga \in H^1(X,\Z)$, form a basis of joint
eigenfunctions of the algebra ${\mc A}$.

Next, we interpret the eigenvalues in terms of $GL_1$-opers on $X$
with real monodromy. Since
$$
D_{GL_1} = \on{Sym} H^1(X,\OO_X) \simeq \on{Fun} H^0(X,\Omega^{1,0}),
$$
we obtain that
$$
\on{Spec} D_{GL_1} = H^0(X,\Omega^{1,0}),
$$
and so a point in the spectrum of $D_{GL_1}$ is a holomorphic one-form
on $X$, or equivalently a holomorphic connection on the trivial line
bundle on $X$, i.e. a $GL_1$-oper on $X$. We write this connection as
\begin{equation}    \label{GL1opgen}
\lambda({\mathbf a}) = \nabla_z - {\mathbf a}, \qquad {\mathbf a} \in
H^0(X,\Omega^{1,0}).
\end{equation}
Together with $\nabla_{\ol{z}}$, it gives rise to the flat connection
\begin{equation}    \label{GL1opgen1}
\nabla({\mathbf a}) = d - {\mathbf a}, \qquad {\mathbf a} \in
H^0(X,\Omega^{1,0}).
\end{equation}
on the trivial line bundle on $X$.

Thus, the isomorphism \eqref{BDisom} specializes for $G=GL_1$ to the
isomorphism
$$
H^0(X,\Omega^{1,0}) = \on{Op}_{GL_1}(X)
$$
under which ${\mathbf a} \in \on{Spec} D_{GL_1} = H^0(X,\Omega^{1,0})$
goes to the oper \eqref{GL1opgen} (and the connection
\eqref{GL1opgen1}).

For each ${\mathbf a} \in \on{Spec} D_{GL_1}$ we have the $D$-module
$\Delta_{\la({\mathbf a})}$ on $\Pic(X)$. It follows from the
definition that its restriction to the neutral component $\Pic^0(X)$
is the trivial line bundle on $\Pic^0(X)$ with the (holomorphic) flat
connection $d-\phi({\mathbf a})$, where $\phi$ denotes the isomorphism
\begin{equation}    \label{isom phi}
\phi: H^1(X,\CC) \overset{\sim}\rightarrow H^1(\Pic^0(X),\CC).
\end{equation}
restricted to the $(1,0)$ Hodge summands on both sides (see the
decomposition \eqref{Hodge}):
$$
H^0(X,\Omega^{1,0}) \overset{\sim}\rightarrow
H^0(\Pic^0(X),\Omega^{1,0}).
$$

Next, we consider the anti-holomorphic counterpart. An
anti-holomorphic $GL_1$-oper is an anti-holomor\-phic connection on
the trivial line bundle on $X$ of the form
\begin{equation}    \label{antiGL1opgen}
\mu({\mathbf b}) = \nabla_{\ol{z}}-{\mathbf b}, \qquad {\mathbf b} \in
H^0(X,\Omega^{0,1}).
\end{equation}
Together with $\nabla_z$, it gives rise to the flat connection on the
trivial line bundle
\begin{equation}    \label{antiGL1op1gen}
\ol\nabla({\mathbf b}) = d - {\mathbf b}, \qquad {\mathbf b} \in
H^0(X,\Omega^{0,1}).
\end{equation}
The isomorphism \eqref{BDisomanti} specializes for $G=GL_1$ to
$$
\on{Spec} \ol{D}_{GL_1} = H^0(X,\Omega^{0,1}) = \ol{\on{Op}}_{GL_1}(X),
$$
under which ${\mathbf b} \in H^0(X,\Omega^{0,1})$ goes to the oper
\eqref{antiGL1opgen} (and the connection \eqref{antiGL1op1gen}).

The restriction of the anti-holomorphic $D$-module
$\ol\Delta_{\mu({\mathbf b})}$ to $\Pic^0(X)$ is the trivial line
bundle with the (anti-holomorphic) flat connection $d-\phi({\mathbf
  b})$, where $\phi$ is the isomorphism \eqref{isom phi} restricted
to the $(0,1)$ Hodge summands on both sides (see \eqref{Hodge}):
\begin{equation}    \label{isom phi bar}
H^0(X,\Omega^{0,1}) \overset{\sim}\rightarrow
H^0(\Pic^0(X),\Omega^{0,1}).
\end{equation}

As in the case of an elliptic curve, $V$ is the space of smooth
functions on $\Pic^0(X)$ and $S({\mc A}) = V$. The algebra ${\mc
  A}_{\mathbb R}$ spanned by the operators $(D+\ol{\nu(D)})$ and
$(D-\ol{\nu(D)})/i$ is essentially self-adjoint on $V$.

The involution $\nu$ acts on an element of $\on{Sym}^i H^1(X,\OO_X)
\simeq D_{GL_1}$ by $(-1)^i$. The corresponding action of the
Chevalley involution $\tau$ on the $GL_1$-oper \eqref{GL1opgen} is
given by the formula $\lambda({\mathbf a}) \mapsto \lambda(-{\mathbf
  a})$. Thus, we have
\begin{equation}    \label{lastgen}
\la({\mathbf a})^* = \la(-{\mathbf a}).
\end{equation}

To a point in the spectrum of ${\mc A}_{\mathbb R}$ in
$L^2(\Pic^0(X))$ we attach a holomorphic $GL_1$-oper $\la({\mathbf
  a})$ (formula \eqref{GL1opgen}), and an anti-holomor\-phic
$GL_1$-oper $\mu({\mathbf b})$ (formula \eqref{antiGL1opgen}) for some
${\mathbf a} \in H^0(X,\Omega^{1,0})$ and ${\mathbf b} \in
H^0(X,\Omega^{0,1})$.

The self-adjointness of ${\mc A}_{\mathbb R}$ implies that
$$
\mu({\mathbf b}) = \ol{\la({\mathbf a})}{}^* = \ol{\la(-{\mathbf a})}
$$
(this is a special case of Lemma \ref{ollast}), i.e.
$$
{\mathbf b} = - \ol{{\mathbf a}}.
$$

We will now prove a stronger version of our main Conjectures
\ref{first}-\ref{third} in the case $G=GL_1$ and a general curve $X$
of genus greater than 0.

\begin{theorem}\hfill    \label{realmonogen}
\begin{enumerate}
\item
  The spectrum of the algebra ${\mc
    A}$ acting on $L^2(\Pic^0(X))$ is discrete and in one-to-one
  correspondence with the set of $GL_1$-opers on $X$ with real
  monodromy.

\item
  The $GL_1$-opers with real monodromy have the form
$$
\lambda(2\pi i \, \omega_\ga) = \nabla_z - 2\pi i \,
  \omega_\ga, \qquad \gamma \in H^1(X,\Z).
$$
  The corresponding eigenfunction of ${\mc A}$ in $L^2(\Pic^0(X))$ is
the Fourier harmonic $f_\ga$ given by formula \eqref{Fourier}.
\end{enumerate}
\end{theorem}

\begin{proof}
As explained above, the algebra ${\mc A}$ is freely generated by the
$C^\infty$ translation vector fields on $\Pic^0(X)$, and therefore the
Fourier harmonics $f_\ga, \gamma \in H^1(X,\Z)$, form a basis of joint
eigenfunctions of ${\mc A}$. Furthermore, it follows from the
construction of $f_\ga$ that
\begin{equation}
(d - 2\pi i\phi(\omega_\ga+\ol\omega_\ga)) f_\ga = 0
\end{equation}
where $\phi$ is the isomorphism \eqref{isom phi}. In other words,
$f_\ga$ is a horizontal section of the trivial line bundle on
$\Pic^0(X)$ with the flat connection $d - 2\pi
i\phi(\omega_\ga+\ol\omega_\ga)$ (this flat line bundle is nothing but
the tensor product of the flat line bundles $\Delta_{\la(2\pi i
  \omega_\ga)}$ and $\ol\Delta_{\mu(2\pi i \ol\omega_\ga)}$ on
$\Pic^0(X)$ described above). Thus, the joint eigenvalues of ${\mc A}$
on $f_\ga$ correspond to the $GL_1$-oper $\lambda(2\pi i \,
\omega_\ga)$ (see formula \eqref{GL1opgen}). As the following lemma
shows, these are precisely the $GL_1$-opers with real monodromy.

\begin{lemma}    \label{real mon line}
  The connection $d - {\mathbf a}$, where ${\mathbf a} \in
  H^0(X,\Omega^{1,0})$, on the trivial line
  bundle on $X$ has real monodromy if and only if ${\mathbf a} = 2 \pi
  i \, \omega_\ga$ for some $\ga \in H^1(X,\Z)$.
\end{lemma}

\begin{proof}
Denote by ${\mc L}_{{\mathbf a}}$ the trivial line bundle with the
connection $d - {\mathbf a}$. The monodromy representation associated
to this flat line bundle is real if and only if the monodromy
representation associated to the flat line bundle ${\mc L}_{{\mathbf
    a}} \otimes \ol{{\mc L}_{{\mathbf a}}}{}^{-1}$ is trivial. The
latter is the trivial line bundle with the flat connection $d -
({\mathbf a} - \ol{\mathbf a})$. The corresponding monodromy along a
cycle $\beta \in H_1(X,\Z)$ is $e^{\int_\beta ({\mathbf a} -
  \ol{\mathbf a})}$. Hence it is equal to $1$ if and only if
$$
\int_\beta ({\mathbf a} - \ol{\mathbf a}) \in 2\pi i \, \Z, \qquad
\forall \beta \in H_1(X,\Z).
$$
This means that ${\mathbf a} - \ol{\mathbf a}$ is equal to $2\pi i$
times a one-form representing a class in $H^1(X,\Z)$, which is
equivalent to the statement of the lemma.
\end{proof}

This proves the theorem.
\end{proof}

\begin{remark} For each ${\mathbf a} \in H^0(X,\Omega^{1,0})$ we have
the flat line bundle ${\mc L}_{\mathbf a}$ on $X$ and the flat line
bundle $\Delta_{\la({\mathbf a})}$ on $\Pic(X)$, which is the Hecke
eigensheaf corresponding to ${\mc L}_{\mathbf a}$. The abelian
quotients of the fundamental groups of $X$ and $\Pic^0(X)$ are
naturally isomorphic to $H_1(X,\Z)$ and $H_1(\Pic^0(X),\Z)$,
respectively, and hence to each other. The one-dimensional monodromy
representations of $\pi_1(X)$ and $\pi_1(\Pic^0(X))$ corresponding to
${\mc L}_{\mathbf a}$ and the restriction of $\Delta_{\la({\mathbf
    a})}$ to $\Pic^0(X)$, respectively, are isomorphic to each other
(see \cite{F:rev}, Sect. 4.3). Hence the former is real if and only if
the latter is real. This property is also equivalent to the property
that the monodromy representations associated to the flat line bundles
${\mc L}_{\mathbf a} \otimes \ol{{\mc L}_{\mathbf a}}{}^{-1}$ and
$\Delta_{\la({\mathbf a})} \otimes \ol\Delta_{\mu(-\ol{\mathbf a})}$
are trivial. According to Lemma \ref{real mon line}, these properties
are satisfied if and only if ${\mathbf a} = 2 \pi i \, \omega_\ga$ for
some $\ga \in H^1(X,\Z)$, and then $f_\ga$ is a horizontal section of
the last flat line bundle.
\end{remark}

\section{Bundles with parabolic structures}    \label{parab}

The Beilinson-Drinfeld construction discussed in Section \ref{holom}
can be generalized to the moduli stack $\Bun_G(X,S)$ of $G$-bundles on
a smooth projective curve $X$ with parabolic structures
(i.e. reductions to a Borel subgroup $B$ of $G$) at the points from a
finite set $S=\{ x_i, i \in J \}$, see Section \ref{borelred}. We then
obtain a commutative subalgebra in the algebra of global twisted
differential operators on $\Bun_G(X,S)$, which is isomorphic to the
algebra of regular functions on an appropriate space of opers with
regular singularities and unipotent monodromy at the points of $S$.

Let us briefly describe this construction in the case of a simple
simply-connected group $G$ (the key elements of this construction are
discussed in \cite{F:icmp,FG1}, see also \cite{F:rev}, Sect. 9.8).

The starting point is the isomorphism
\begin{equation}    \label{double}
\Bun_G(X,S) \simeq G(\CC[X \bs S]) \bs \prod_{i\in J} G(F_{x_i})/
\prod_{i\in J} G({\mc O}_{x_i}),
\end{equation}
where $F_x$ is the formal completion of the field $\CC(X)$ of rational
functions on $X$ and ${\mc O}_x$ is its ring of integers. Using this
isomorphism, we construct a localization functor. To do that, we need
to introduce some notation:

\begin{enumerate}
\item We denote by $\Lambda$ the group of characters of a Cartan
  subgroup $H \subset B \subset G$ of $G$. A pair $(x,\la)$ where
  $x\in S$ and $\la \in \Lambda$ defines a line bundle $\mcL
  _{(x,\la)}$ on $\Bun _G(X,S)$.
\item For $\orr\la: J\to \Lambda$ we define a line bundle
  $\pi^*(K^{1/2}) \otimes \mcL(\orr\la)$ on $\Bun_G(X,S)$,
  where $\pi$ is the natural projection $\Bun_G(X,S) \to \Bun_G$ and
$$
\mcL(\orr\la):= \bigotimes _{i\in J} \mcL
_{(x_i,\orr\la(i))}.
$$
\item We denote by $D_{\orr\la}(\Bun_G(X,S))$ the category of modules
  over the sheaf ${\mc D}_{S,\orr\la}$ of (holomorphic) differential
  operators acting on the line bundle $\pi^*(K^{1/2}) \otimes
  \mcL(\orr\la)$. This category is in fact well-defined for all
  $\orr\la: S \to \Lambda \otimes_{\Z} \CC = \h^*$, where
  $\h=\on{Lie}(H)$.
\item The line bundle $\pi^*(K^{1/2}) \otimes \mcL(\orr\la)$, where
  $\orr\la(i) = -\rho$ for all $i \in J$, is a square root of the
  canonical line bundle on $\Bun_G(X,S)$. In this paper we will mostly
  study this case. However, some of our results can be generalized to
  other values of $\orr\la$.
\end{enumerate}

For $\la \in \h^*$, introduce the $\ghat_{\on{crit}}$-module of
critical level
$$
{\mathbb M}_{\la,\on{crit}} = \on{Ind}_{\g[[t]] \oplus \CC{\mathbf
    1}}^{\ghat_{\on{crit}}} M_\la,
$$
where $M_\la$ is the Verma module of highest weight $\la$ over
$\g$. Let $Z(\ghat_{\on{crit}})$ be the center of the completed
enveloping algebra of
$\ghat_{\on{crit}}$. We have a natural homomorphism
$$
\xi_\la: Z(\ghat_{\on{crit}}) \to \on{End}_{\ghat_{\on{crit}}}
{\mathbb M}_{\la,\on{crit}}.
$$
It is proved in \cite{FG1} (see \cite{F:book}, Theorem 9.5.3) that this
homomorphism is surjective and there is a canonical isomorphism
\begin{equation}    \label{locisom}
\on{Im} \xi_\la \simeq \on{Fun}
\on{Op}^{\on{RS}}_{\LG}(D)_{\varpi(-\la-\rho)},
\end{equation}
where $\on{Op}^{\on{RS}}_{\LG}(D)_{\varpi(-\la-\rho)}$ is the space of
opers on the disc $D = \on{Spec} \CC[[t]]$ with regular singularity at
the origin and residue $\varpi(-\la-\rho)$, as defined in \cite{BD}
(see also \cite{F:book}, Sect. 9.1). Here $\varpi$ denotes the
projection $\h^* \to \h^*/W$. Thus, in the basic case when all
$\orr\la(i)=-\rho$ the residue is equal to $0 \in \h^*/W$ for all $i
\in J$. (This does not mean that the oper is regular at these points;
on the contrary, it has regular unipotent monodromy around each of
these points, see \cite{F:book}, Sect. 9 for more details.)

Let $I$ be the Iwahori subgroup of the formal loop group
$G(\!(t)\!)$. For $\la \in \h^*$, denote by
$\ghat_{\on{crit}}\on{-mod}_{\la}^{I}$ the category of
$(\ghat_{\on{crit}},I)$ Harish-Chandra modules on which the
  action of the center $Z(\ghat_{\on{crit}})$ factors through the
  algebra $\on{End}_{\ghat_{\on{crit}}} {\mathbb
    M}_{\la,\on{crit}}$.

Given $\orr\la$ as above, we have a localization functor
$$
\Delta_{S,\orr\la}: \prod_{i \in J}
\ghat_{\on{crit}}\on{-mod}_{\orr\la(i)}^{I} \to
D_{\orr\la}(\Bun_G(X,S)).
$$
It sends $\bigotimes_{i\in J} {\mathbb M}_{\orr\la(i),\on{crit}}$ to
the sheaf ${\mc D}_{S,\orr\la}$ of differential operators on the line
bundle $\pi^*(K^{1/2}) \otimes \mcL(\orr\la)$. Hence we obtain a
homomorphism
$$
\bigotimes_{i\in J} \on{End} {\mathbb M}_{\orr\la(i),\on{crit}} \to
\Gamma(\Bun_G(X,S)),{\mc D}_{S,\orr\la})^{\on{opp}}.
$$
Denote the image of this homomorphism by $D_{G,S,\orr\la}$. Since this
is a commutative subalgebra of $\Gamma(\Bun_G(X,S)),{\mc
  D}_{S,\orr\la})^{\on{opp}}$, it can also be viewed as a commutative
subalgebra of $\Gamma(\Bun_G(X,S)),{\mc D}_{S,\orr\la})$.

The following result is an analogue of the Beilinson-Drinfeld Theorem
\ref{BD1isom}:

\begin{theorem}    \label{isomparab}
There is an isomorphism
\begin{equation}    \label{BDparab}
D_{G,S,\orr\la} \simeq \on{Fun}
\on{Op}^{\on{RS}}_{\LG}(X)_{S,\orr\la},
\end{equation}
where $\on{Op}^{\on{RS}}_{\LG}(X)_{S,\orr\la}$ is the space of
$\LG$-opers on $X$ that are regular outside $S$ and have regular
singularity at each of the points $x_i \in S$ with residue
$\varpi(-\orr\la(i)-\rho) \in \h^*/W$.

This isomorphism fits in a commutative diagram
\begin{equation}    \label{fits}
\begin{CD}
\bigotimes_{i\in J} \on{End} {\mathbb M}_{\orr\la(i),\on{crit}} @>>> \bigotimes_{i
  \in J} \on{Fun} \on{Op}^{\on{RS}}_{\LG}(D_{x_i})_{\varpi(-\orr\la(i)-\rho)},
 \\
@VVV @VVV \\
D_{G,S,\orr\la} @>>> \on{Fun}
\on{Op}^{\on{RS}}_{\LG}(X)_{S,\orr\la}
\end{CD}
\end{equation}
where the right vertical arrow is induced by the natural embedding
$$
\on{Op}^{\on{RS}}_{\LG}(X)_{S,\orr\la} \hookrightarrow 
\prod_{i \in J}
\on{Op}^{\on{RS}}_{\LG}(D_{x_i})_{\varpi(-\orr\la(i)-\rho)}
$$
sending an $\LG$-oper on $X$ to its restrictions to the discs $D_{x_i}
= \on{Spec} {\mc O}_{x_i}$, $i \in J$.
\end{theorem}

The proof is obtained by combining the proof of Theorem
\ref{BD1isom} in \cite{BD} with the isomorphism \eqref{locisom} for
each $\la(i), i \in J$.

In a similar way, we define the complex conjugate commutative algebra
$\ol{D}_{G,S,\orr\la}$ of anti-holomorphic differential operators
acting on the line bundle $\ol{\pi^*(K^{1/2}) \otimes
  \mcL(\orr\la)}$ and its isomorphism with the algebra of functions on
the space of anti-holomorphic opers
$\ol{\on{Op}}^{\on{RS}}_{\LG}(D)_{\varpi(-\la-\rho)}$.

\section{The case of $SL_2$ and $\pone$ with marked
  points}    \label{SL2}

In what follows, we will focus on the case $G=SL_2$, $X=\pone$, and
$S$ the set of $N+1$ distinct marked points which we will denote by
$z_1,\ldots,z_N$ and $\infty$ using a coordinate $t$ on $\pone$. We
choose as Borel subgroup $B\subset PGL_2$, the subgroup of upper
triangular matrices. The moduli stack $\Bun_{SL_2}(\pone,S)$ of
principal $SL_2$-bundles on $\mbb P^1$ with Borel reductions
(parabolic structures) at these points has an open dense subset
${\mathcal M}^{(N)}$ classifying the pairs $(\mcF_{\on{triv}},r_S)$,
where $\mcF_{\on{triv}}$ is the trivial $SL_2$-bundle on $\pone$ (see
Section \ref{borelred}). We have an isomorphism of stacks
\begin{equation}    \label{MN}
{\mathcal M}^{(N)} \simeq (\mbb P^1)^{N+1}/PGL_2^{\text{diag}} = 
(\mbb P^1)^N/B,
\end{equation}
The moduli space of stable pairs $(\mcF_{\on{triv}},r_S)$ is the
corresponding GIT quotient. Our Hilbert space $\mcH$ can be defined as
the completion of the space $V$ of compactly supported functions on
this GIT quotient. The corresponding holomorphic differential
operators can be identified with the {\it quantum Gaudin
  Hamiltonians}. In the following discussion we use the results of
\cite{F:icmp}, Sect. 5.

Let us identify $\h^*$ with $\CC$ by sending the fundamental
weight to $1 \in \CC$. For any $\la \in \CC$, we have the line bundle
$\mcL_\la$ on $\pone$, and the Lie algebra $\mathfrak{sl}_2$ maps to
the algebra $D_\la$ of global differential operators acting on
$\mcL_\la$. Let $H_i, i=1,\ldots,N$, be the following elements of the
algebra $U(\mathfrak{sl}_2)^{\otimes N}$:
\begin{equation}    \label{hi}
H_i = \sum_{j\neq i} \frac{\Omega_{ij}}{z_i-z_j},
\end{equation}
where
$$
\Omega=e\otimes f + f\otimes e + \frac{1}{2} h\otimes h.
$$
These are the Gaudin Hamiltonians. They commute with each other
and with the diagonal action of $SL_2$. They also satisfy the
relations
\begin{equation}    \label{rel1}
\sum_{i=1}^N H_i = 0,
\end{equation}
\begin{equation}    \label{rel2}
\sum_{i=1}^N z_iH_i = \frac{1}{2} C_{\on{diag}}-\sum_{i=1}^N
\frac{\la_i(\la_i+2)}{4},
\end{equation}
where $C_{\on{diag}}$ is the Casimir element
$$
C = ef + fe + \frac{h^2}{2}
$$
acting via the diagonal action of $\mathfrak{sl}_2$.

Let us fix $\la_i \in \mbb C$ and view $H_i$ as elements of
$\bigotimes_i \mathcal{D}_{\lambda_i}$. We also fix $\la_\infty$ and
define $\orr\la: J \to \h^*$ by the formulas $\orr\la(i)=\la_i$ and
$\orr\la(\infty)=\la_\infty$. Recall the algebra
$D_{SL_2,S,\orr\la}$ introduced in the previous section. The following
result follows from \cite{F:icmp}:

\begin{theorem}\hfill
\begin{enumerate}
\item The algebra $D_{SL_2,S,\orr\la}$ is equal to the quotient of
$\CC[H_i]_{i=1,\ldots,N}$ by the relations \eqref{rel1}, \eqref{rel2},
and $C_{\on{diag}}=\frac{\la_\infty(\la_\infty+2)}{2}$.

\item The space $\on{Op}^{\on{RS}}_{PGL_2}(X)_{S,\orr\la}$ is the space
  of second order operators (projective connections)
\begin{equation}    \label{qt}
\partial_t^2 - q(t) = \partial_t^2  - \sum_{i=1}^N
\frac{\la_i(\la_i+2)/4}{(t-z_i)^2} - \sum_{i=1}^N \frac{\mu_i}{t-z_i}.
\end{equation}
such that the leading term at $\infty$ is
$-\frac{\la_\infty(\la_\infty+2)}{4s^2}$, where $s=t^{-1}$ is a local
coordinate at $\infty$.

\item The isomorphism \eqref{BDparab} from $D_{SL_2,S,\orr\la}$ to
  $\on{Fun} \on{Op}^{\on{RS}}_{PGL_2}(X)_{S,\orr\la}$ maps $H_i \mapsto
  \mu_i$ for all $i=1,\ldots,N$.
\end{enumerate}
\end{theorem}

Part (3) of the theorem implies that the $\mu_i$'s appearing in
formula \eqref{qt} satisfy the relations \eqref{rel1}, \eqref{rel2} in
which we replace the $H_i$'s with the $\mu_i$'s. Hence the system of
differential equations
\begin{equation}\label{eq12}
H_i\psi=\mu_i\psi, \qquad i=1,\ldots,N,
\end{equation}
where the $\mu$'s are appear in \eqref{qt}, is well-defined. It gives
rise to a $B$-equivariant $(\lambda_1,...,\lambda_N)$-twisted
$\mathcal{D}$-module
$\Delta(\overrightarrow\lambda,\overrightarrow\mu)$ on $(\mbb P^1)^N$
which is freely generated by one generator $\psi$ satisfying the
relations \eqref{eq12}. The system \eqref{eq12} is a $D$-module
realization of the quantum Gaudin system.

We can write down the operators $H_i$ quite explicitly on an open
subset of ${\mathcal M}^{(N)}$. Namely, let $U_i$ be the big cell on
the $i$th copy of $\mbb P^1$ and $x_i$ the usual coordinate on
$U_i$. The algebra of $\la_i$-twisted differential operators on $U_i$
can then be naturally identified with the Weyl algebra with respect to
the variable $x_i$. The $\lambda$-twisted differential operators on
the big cell $\mbb C\subset \mbb P^1 \bs \{ \infty \}$ corresponding
to the standard basis elements of $\mathfrak{sl}_2$ are
\begin{equation}    \label{sl2}
e = -x^2 \partial_x + \la x, \quad h= 2x
\partial_x - \la, \quad f = \partial_x.
\end{equation}
Thus the restriction
of $H_i$ to the open subset $\prod_{i=1}^N U_i$ of
$\left( \mbb P^1 \right)^N$, is given by 
\begin{equation}    \label{GaudinHam}
H_i=\sum_{j\neq i} \frac{1}{z_i-z_j} \left(- (x_i-x_j)^2
 \partial_{x_i}\partial_{x_j} + (x_i-x_j)(\la_i\partial_{x_j} -
 \la_j\partial_{x_i}) + \frac{\la_i \la_j}{2}\right)
\end{equation}
for all $i=1,\ldots,N$.

Since we consider the action of $H_i$ on translation-invariant
functions, we have $\sum_i \partial_{x_i}=0$.  Using this equality, we
obtain
$$
\sum_i z_iH_i=E^2-(\lambda+1)E+\frac{\lambda^2-\sum_i \lambda_i^2}{4},
$$
where $\lambda:=\sum_i \lambda_i$ and $E=\sum_i x_i\partial_i$ is the
Euler vector field.  This means that system \eqref{eq12} has a
homogeneous solution $\psi$ (i.e., such that $E\psi=\beta \psi$ for
$\beta\in \mbb C$) if and only if
$\beta=\frac{1}{2}(\lambda-\lambda_\infty)$, where as above
$\lambda_\infty$ satisfies the equation
$$
\frac{\lambda_\infty(\lambda_\infty+2)}{4}=\sum_{i=1}^N
\frac{\lambda_i(\lambda_i+2)}{4}+\sum_{i=1}^N z_i\mu_i.
$$
So we add the relation $E\psi=\beta\psi$ to the relations of
$\Delta(\overrightarrow{\la},\overrightarrow{\mu})$.

Denote the corresponding $B$-equivariant twisted $D$-module on
$(\pone)^N$ by $\Delta(\orr{\lambda},\orr{\mu})$. Equivalently, we may
view it as a twisted $D$-module on the stack ${\mathcal M}^{(N)} =
(\pone)^N/B$. The following result follows from \cite{F:icmp}.

\begin{proposition} \label{hecke}
  The $D$-module
  $\Delta(\orr{\lambda},\orr{\mu})$ is the
  restriction to ${\mathcal M}^{(N)}$ of the Hecke eigensheaf
  corresponding under the geometric Langlands correspondence to the
  flat $PGL_2$-bundle on $\mbb P^1$ with regular singularities at
  $z_1,\ldots,z_N,\infty$ defined by the $PGL_2$-oper (projective
  connection) on $\mbb P^1$ given by formula \eqref{qt}.
\end{proposition}

Considering the symbols of $H_i$, it is easy to see that the D-module
${\Delta}(\overrightarrow{\la},\overrightarrow{\mu})$ has rank
$2^{N-2}$.

\section{Proofs of two results}    \label{App}

\subsection{Outline of the proof of Proposition \ref{nu}}    \label{proofnu}

The argument below was formulated jointly with D. Gaitsgory.

Let ${\mathfrak D}_{G,\on{crit}}$ be the vertex algebra of chiral
differential operators (CDO) on the group $G$, see \cite{ArG} for the
precise definition. It carries two commuting actions of the affine
Kac-Moody algebra $\ghat_{\on{crit}}$ of critical level. Using these
actions, we apply the localization functor to ${\mathfrak
  D}_{G,\on{crit}}$ and obtain an object of the category of twisted
$D$-modules on $\Bun_G \times \Bun_G$, twisted by $K^{1/2}$ along both
factors. It is known that the restriction of this object to $U \times
U$, where $U$ is an open dense substack of $\Bun_G$, is isomorphic to
the restriction of $\Delta_! := (K^{1/2} \boxtimes K^{1/2})
\Delta_*(\OO)$, where $\Delta$ is the diagonal morphism and $\Delta_*$
is the direct image functor for left $D$-modules.

Recall that we denote by ${\mc D}_G$ the sheaf of $K^{1/2}$-twisted
differential operators on $\Bun_G$ and by $D_G$ the commutative
algebra of its global sections. Thus, $\Delta_!$ has two commuting
left actions of ${\mc D}_G$ corresponding to the two factors in
$\Bun_G \times \Bun_G$, sending a local section $D$ of ${\mc D}_G$ to
$D_\ell$ and $D_r$, respectively.

\begin{lemma}\hfill    \label{omega}

(1) There is a canonical section $\omega$ of $\Delta_!$, such
  that with respect to each of the two actions of ${\mc D}_G$,
  $\Delta_!$ is a free rank one module generated by $\omega$.

(2) For any local section $D$ of $\D_G$, let $\wt{D}$ be defined by
  the formula
\begin{equation}    \label{duality}
D_\ell\, \omega = \wt{D}_r\, \omega.
\end{equation}
Then $\wt{D}=\nu(D)$, where $\nu$ is the canonical anti-involution on
${\mc D}_G$.
\end{lemma}

\begin{proof} Part (1) follows from the definitions. The fact that the
  map $D \mapsto \wt{D}$ is an anti-involution follows from formula
  \eqref{duality}:
$$
(D_1 D_2)_\ell \, \omega = (D_1)_\ell (D_2)_\ell \, \omega = (D_1)_\ell
(\wt{D_2})_r \, \omega = (\wt{D_2})_r (D_1)_\ell \, \omega = (\wt{D_2})_r
(\wt{D_1})_r \, \omega.
$$
Therefore to prove part (2), it is sufficient to check that this
anti-involution coincides, locally, with $\nu$ for functions and
vector fields, which is easy to verify directly.
\end{proof}

Now we are ready to prove Proposition \ref{nu}. Let ${\mathbb
  V}_{\on{crit}}$ be the vertex algebra associated to
$\ghat_{\on{crit}}$ and ${\mathfrak z}_{\on{crit}}$ its center. As
explained in \cite{FG}, the commuting left and right actions of
$\ghat_{\on{crit}}$ on the vertex algebra ${\mathfrak
  D}_{G,\on{crit}}$ give rise to two vertex algebra embeddings
\begin{equation}    \label{V to D}
{\mathbb V}_{\on{crit}} \to {\mathfrak D}_{G,\on{crit}},
\end{equation}
which we denote by $\imath_\ell$ and $\imath_r$, and hence two vertex
algebra embeddings
\begin{equation}    \label{z to D}
{\mathfrak z}_{\on{crit}} \to {\mathfrak D}_{G,\on{crit}}.
\end{equation}
Applying the localization functor to ${\mathbb V}_{\on{crit}}$, we
obtain the sheaf ${\mc D}_G$ on $\Bun_G$. It follows that under the
localization functor, the two morphisms $\imath_\ell$ and $\imath_r$
of \eqref{V to D} correspond to the two morphisms ${\mc D}_G \to
\Delta_!$ sending $D \mapsto D_\ell \, \omega$ and $D \mapsto D_r
\, \omega$, respectively, where $D \in {\mc D}_G$ and we use the
notation introduced before Lemma \ref{omega}.

Thus, we obtain a commutative diagram
\begin{equation}    \label{diag}
\begin{CD}
{\mathfrak z}_{\on{crit}} @>>> {\mathfrak z}_{\on{crit}}
 \\
@VVV @VVV \\
D_G @>>> D_G
\end{CD}
\end{equation}
where the top horizontal map corresponds to the identification of the
images of the two embeddings \eqref{z to D} in ${\mathfrak
  D}_{G,\on{crit}}$ and the bottom horizontal map corresponds to the
identifications of the images of $D_G$ in $\Delta_!$ under the maps $D
\mapsto D_\ell \, \omega$ and $D \mapsto D_r \, \omega$, where $D \in
D_G$. By Lemma \ref{omega},(2), the bottom horizontal map is $\nu: D_G
\to D_G$.

The following statement is proved in \cite{FG} (Theorem 5.4).

\begin{proposition}    \label{images of z}
  The images of the two embeddings \eqref{z to D} coincide, and the
  corresponding automorphism of ${\mathfrak z}_{\on{crit}}$ is the one
  induced by the Chevalley involution $\tau$.
\end{proposition}

By Proposition \ref{images of z}, the top horizontal map of the above
diagram is induced by the Chevalley involution $\tau$. Therefore, we
obtain that the action of $\nu$ on $D_G$ corresponds to the action of
$\tau$ on ${\mathfrak z}_{\on{crit}}$.

Finally, it follows from the construction \cite{FF,F:wak} of the
isomorphism
$$
{\mathfrak z}_{\on{crit}} \simeq \on{Fun} \on{Op}_{\LG}(D)
$$
that the action of $\tau$ on the left hand side corresponds to the
action of $\tau$ on the right hand side.

The commutative diagram 
$$
\begin{CD}
{\mathfrak z}_{\on{crit}} @>{\sim}>> \on{Fun} \on{Op}_{\LG}(D)
 \\
@VVV @VVV \\
D_G @>{\sim}>> \on{Fun} \on{Op}_{\LG}(X)
\end{CD}
$$
which is proved in \cite{BD}, then implies that the action of $\nu$ on
$D_G$ corresponds to the action of $\tau$ on $\on{Op}_{\LG}(X)$. This
completes the proof of Proposition \ref{nu}.

\subsection{Proof of Lemma \ref{V*}}    \label{proof*}

We start with the following general result. Let $Z$ be a smooth
variety of complex dimension $d$, $K^{1/2}$ a square root of the
canonical line bundle on $Z$, and ${\mc D}_{1/2}$ the sheaf of
holomorphic differential operators acting on $K^{1/2}$. Let $\nu$ be
the canonical anti-involution on ${\mc D}_{1/2}$. For any right ${\mc
  D}_{1/2}$-module ${\mc M}$, we have the left ${\mc D}_{1/2}$-module
$\nu^*({\mc M})$.

Recall that for any holonomic ${\mc D}_{1/2}$-module ${\mc F}$ on $Z$
we define its Verdier dual ${\mathbb D}({\mc F})$ as the left ${\mc
  D}$-module obtained by applying $\nu^*$ to the right ${\mc
  D}_{1/2}$-module
\begin{equation}    \label{Vdef}
R^d{\mc Hom}({\mc F},{\mc D}_{1/2}).
\end{equation}
In particular, if ${\mc F} = K^{1/2} \otimes {\mc V}$, where ${\mc V}$
is a vector bundle with a holomorphic flat connection on $Z$, then
${\mathbb D}({\mc F}) \simeq K^{1/2} \otimes {\mc V}^*$. We also have
the following result:

\begin{lemma}    \label{exact}
Suppose that $D_{1/2} = \Gamma(Z,{\mc
  D}_{1/2}) = \CC[D_1,\ldots,D_d]$ and that the functor
\begin{equation}    \label{functor}
M \mapsto {\mc D}_{1/2} \underset{D_{1/2}}\otimes M
\end{equation}
from $D_{1/2}$-modules to ${\mc D}_{1/2}$-modules is exact. Given $\la
\in \on{Spec} D_{1/2}$, let
$$
\Delta_\lambda = {\mc D}_{1/2}/{\mc D}_{1/2} \cdot (D_i -
\la(D_i))_{i=1,\ldots,d}.
$$
Then
\begin{equation}    \label{nuDi}
{\mathbb D}(\Delta_\lambda) \simeq {\mc D}_{1/2}/{\mc D}_{1/2} \cdot
(\nu(D_i) - \la(D_i))_{i=1,\ldots,d}.
\end{equation}
\end{lemma}

\begin{proof} Since the above functor is exact, applying it to
the Koszul resolution of the one-dimensional representation of
$D_{1/2}$ corresponding to $\la$, we obtain a free resolution of
$\Delta_\lambda$. We then obtain that
$$
R^d{\mc Hom}({\mc F},{\mc D}_{1/2}) \simeq {\mc D}_{1/2}/
(D_i - \la(D_i))_{i=1,\ldots,d} \cdot {\mc D}_{1/2}
$$
as a {\em right} ${\mc D}_{1/2}$-module (indeed, in the above formula
we take the quotient by the right ideal). Therefore ${\mathbb
  D}(\Delta_\lambda)$, which is by definition obtained by applying
$\nu^*$ to this right ${\mc D}_{1/2}$-module, is given by the right
hand side of \eqref{nuDi}.
\end{proof}

Now we can prove Lemma \ref{V*}. Our ${\mc D}_G$-module $\Delta_\la$
satisfies the conditions of Lemma \ref{exact} (the exactness of the
above functor \eqref{functor} is known; it follows from the flatness
of the Hitchin map). Applying Lemma \ref{exact} and Proposition
\ref{nu}, we obtain that
$$
{\mathbb D}(\Delta_\lambda) \simeq \Delta_{\lambda^*}
$$
Restricting this isomorphism to the above open subset $U$, we obtain
that $({\mc V}_\la)^* \simeq {\mc V}_{\la^*}$.

\bigskip

\bigskip

\part*{\hspace*{70mm}Part II}

\bigskip

\bigskip

In this part we give a proof of Conjectures \ref{first}--\ref{third}
in the special case when $X=\mbb P^1$ with four marked points and
$G=SL_2$. We also develop a theory of Sobolev and Schwartz spaces
relevant to this problem. Namely, in Section 9 we show that the Gaudin
system for $N=4$ reduces to the joint eigenvalue problem for a Darboux
operator $L$ and its complex conjugate $L^\star$, and then discuss the
Darboux operators in detail. In Section 10 we study joint
eigenfunctions of Darboux operators and their complex conjugates, and
show that they correspond to those eigenvalues that give rise to real
monodromy. In Section 11 we give preliminaries on unbounded
self-adjoint operators. In particular, we formulate Nelson's theorem
giving sufficient conditions for strong commutativity of such
operators (i.e. commutativity of the spectral resolutions). We then
discuss the notion of an essentially self-adjoint commutative algebra
of unbounded operators on a Hilbert space, which we use in the
functional analysis used in the formulation of our main results and
conjectures. In Section 12 we state the main theorem of Part II
(Theorem \ref{mainthe}), which in particular implies that $L$ is an
unbounded normal operator with discrete spectrum, and its
eigenfunctions are joint eigenfunctions of $L,L^\dagger$ and form a
basis of the Hilbert space. In Section 13 we develop a theory of
Sobolev and Schwartz spaces attached to a Darboux operator $L$, which
is interesting in its own right and is also a tool for proving the
main theorem. In Sections 14-16 we give three different proofs of
Theorem \ref{mainthe}, two of them using Green's function (for the
operators $L$ and $L^\dagger L$, respectively) and one using Nelson's
theorem. In Section 17 we use our main result to give a spectral
description of Sobolev and Schwartz spaces attached to $L$. Section 18
consists of several remarks. Finally, in Section 19 we explain what
happens in degenerate cases, when the elliptic curve underlying the
Darboux operator degenerates to a nodal rational curve or to the union
of two projective lines meeting transversally at two points. In this
case, the spectrum of $L$ becomes continuous.

\section{Darboux operators} 

In this section we will show that the Gaudin system in the case $N=4$ reduces to the eigenvalue problem for the Darboux
operator. So let us discuss the theory of Darboux operators. 

\subsection{Definition and classification of Darboux operators} 

\begin{definition}\label{dar} A Darboux operator is a twisted holomorphic second order differential operator $L$ on $\mbb P^1$ such that for every $\Lambda\in \mbb C$ the equation $L\psi=\Lambda \psi$ has regular singularities. 
\end{definition} 

Let us classify Darboux operators up to symmetries (action of $PGL_2(\mbb C)$, adding a constant and scaling). The symbol $P$ of $L$ is a section of $S^2T\mbb P^1=O(4)$, so there are at most four singularities (zeros of $P$), and since they are regular, the order of each zero is $\le 2$. Let $z$ be the standard complex coordinate on $\mbb P^1$. Let us use the $PGL_2(\mbb C)$-action to place the singularity of the largest order (which is $1$ or $2$) at $\infty$. Outside $z=\infty$ we can write $L$ as a usual differential operator with polynomial coefficients: 
$$
L=P(z)\partial^2+Q(z)\partial +R(z),
$$
where $\partial:=\partial_z$ is the derivative with respect to $z$, and $2\le \deg P\le 3$. 

Let us write $L$ near $z=\infty$ using the local coordinate $w=1/z$. Let $s$ be the twisting parameter (such that 
if $s$ is an integer then $L$ acts on sections of $O(s)$). Then we have 
$$
L=P(w^{-1})(w^2\partial_w-sw)^2-Q(w^{-1})(w^2\partial_w-sw)+R(w^{-1})=
$$
$$
P(w^{-1})w^4\partial_w^2+(2(1-s)P(w^{-1})w+Q(w^{-1}))w^2\partial_w+
$$
$$
+s(s-1)P(w^{-1})w^2+sQ(w^{-1})w+R(w^{-1}). 
$$
We see that $\deg Q\le \deg P-1$ and that $R(z)$ is the regular part of $\frac{s(1-s)P(z)-szQ(z)}{z^2}$ up to adding a constant. Finally, if $P$ is quadratic with a root of order $2$ at $z=0$ then $Q(0)=0$. 

Also, it is easy to check that these conditions are sufficient. Thus, we obtain the following proposition. 

\begin{proposition}\label{da1} The Darboux operators up to symmetries are as follows: 

(a) Reducible: two singularities at $0$ and $\infty$. Then up to symmetries $P=z^2$ (i.e, $y^2=P(z)$ is a reducible curve), so $Q(z)=(c+2)z$ and 
$$
L=\left(z\partial+\frac{1}{2}\right)^2+c\left(z\partial+\frac{1}{2}\right).
$$ 

(b) Trigonometric: three singularities at $-1,1$ and $\infty$. Then up to symmetries $P=z^2-1$ (i.e., $y^2=P(z)$ is a smooth conic), so $Q(z)$ is linear and 
$$
L=\partial(z^2-1)\partial+(c_0z+c_1)\partial.  
$$

(c) Elliptic: four singularities at $e_1,e_2,e_3,\infty$, where 
$e_1+e_2+e_3=0$ (and $e_i$ are defined up to simultaneous scaling). Then up to symmetries 
$$
P(z)=(z-e_1)(z-e_2)(z-e_3)
$$
 (i.e., $y^2=P(z)$ is a smooth elliptic curve), so $Q(z)$ is quadratic and 
$$
L=\partial (z-e_1)(z-e_2)(z-e_3)\partial+(c_0z^2+c_1z+c_2)\partial-s(s+2+c_0)z.
$$ 
\end{proposition}

It is clear that case (a) is a degeneration of case (b), and case (b) is a degeneration of case (c). 
Moreover, it is easy to see that any twisted second order differential operator on $\mbb P^1$ is 
a Darboux operator or its confluent degeneration (in which three or four zeros of $P$ collide, destroying the 
regularity of singularities). 

\subsection{The elliptic case} Consider the elliptic case (c) in more detail.  
It is useful to write the operator $L$ in this case in the flat coordinate $u$ on the corresponding elliptic curve $E$ using that $\mbb P^1=E/\mbb Z_2$, i.e, $z=\wp(u)$, where $\wp(u)=\wp(u,\tau)$ is the Weierstrass $\wp$-function, so that $(\wp')^2=4(\wp-e_1)(\wp-e_2)(\wp-e_3)$. Then $\partial=\partial_z=\wp'(u)^{-1}\partial_u$. 

To simplify formulas, we replace $L$ by $4L$. Then we obtain 
$$
L=\partial_u^2 +\frac{\wp''+4(c_0\wp^2+c_1\wp(u)+c_2)}{\wp'}\partial_u-4s(s+2+c_0)\wp. 
$$
The coefficient of $\partial_u$ in $L$ is an odd elliptic function with simple poles at points of order $2$ on $E$, so it has the form 
$$
h(u)=2\sum_{p\in \lbrace{0,\frac{1}{2},\frac{\tau}{2},\frac{1+\tau}{2}\rbrace}}b_p\zeta(u-p)+C, 
$$
where 
$$
\zeta(u)=\zeta(u,\tau)=\int \wp(u,\tau)du
$$
is the Weierstrass $\zeta$-function, $b_p\in \mbb C$ are such that $\sum_p b_p=0$, and $C$ is determined from the condition that $h$ is odd. So we get 
$$
L=\partial_u^2 +(2\sum_pb_p\zeta(u-p)+C)\partial_u-4s(s+2+c_0)\wp(u). 
$$
Thus, conjugating $L$ by the function 
$\prod_p\theta(u-p)^{b_p}e^{Cz}$ and adding a constant, we get the operator
$$
\widetilde L=\partial_u^2+\sum_{p: 2p=0}\left(\frac{1}{4}-a_p^2\right)\wp(u-p), 
$$
where $a_p=b_p+\frac{1}{2}$ for $p\ne 0$ and $a_0=b_0+\frac{1}{2}+2s$ are such that 
$$
\sum_{p: 2p=0}a_p=2(s+1).
$$ 

This is the classical Darboux operator \cite{Da}, see also \cite{TV,V}. This motivates Definition \ref{dar}. We see that the Darboux operator in (c) depends on $5$ parameters $\mbf a=(a_p)\in \mbb C^4$ and $\tau$. We will denote it by $L(\mbf a,\tau)$. 

Let $\D_s=\D_s(\mbb P^1)$ be the algebra of differential operators on $\mbb P^1$ with twisting parameter $s$. 
Then to every $L\in \D_s$ we can attach its algebraic adjoint $L^*\in \D_{-2-s}$ (cf. Subsection \ref{cananti}). 
The following proposition is easy to prove by a direct computation: 

\begin{proposition}\label{da2} 
One has $L(\mbf a,\tau)^*=L(-\mbf a,\tau)$. 
\end{proposition}

Recall that we have a natural isomorphism $\D_s\cong \D_{-2-s}$ (both are identified with a quotient of $U(\mathfrak{sl}_2)$). It is easy to show that under this identification $L(\mbf a,\tau)$ goes to $L(\mbf a',\tau)$, where $\mbf a'=(a_p')$ and $a_p'=a_p-s-1$. Thus we can talk about algebraically self-adjoint Darboux operators, i.e. those for which
$L(\mbf a,\tau)^*=L(\mbf a',\tau)$. In view of Proposition \ref{da2}, this holds if and only if $a_p=\frac{1}{2}(s+1)$ for all $p$. In this case, setting $v=2u$, 
we get: 
$$
\frac{1}{4}\widetilde{L}=\partial_v^2-\frac{s}{2}\left(\frac{s}{2}+1\right)\wp(v). 
$$ 
So for each $s$ there is a unique up to symmetries algebraically self-adjoint Darboux operator, which is the classical Lam\'e operator. In the $z$ coordinate it looks as follows: 
$$
L=\partial P\partial -\frac{s+1}{2}P'\partial+\frac{s(s-1)}{2}z. 
$$
In particular, if $a_p=0$ we obtain the algebraically self-adjoint Darboux operator on the bundle $O(-1)$ of half-forms:
$$
\frac{1}{4}\widetilde{L}=\partial_v^2-\frac{1}{4}\wp(v) 
$$ 
and in the $z$ coordinate
$$
L=\partial P\partial+z. 
$$

\subsection{The case $N=4$ for Gaudin operators}

Let us now consider Gaudin operators in the special case $N=4$, and set $\lambda_4:=\lambda_\infty$. 
In this case we are supposed to get a rank $2$ D-module on $\mbb P^1$ with regular singularities at four points $z_1,z_2,z_3,z_4=\infty$. Since this D-module is invariant under the affine linear transformations $z\to b_1z+b_2$, we may 
set $z_1=a,z_2=1,z_3=0$. Let us consider the equation 
\begin{equation}\label{h1eq}
H_1\psi=\mu_1\psi
\end{equation}
where $\psi$ is translation invariant and homogeneous of degree $\beta$. 
Let $\eta(x_1,x_2):=\psi(x_1,x_2,0)$. We write $\partial_i$ for $\partial_{x_i}$ for brevity. 
Using that $\partial_3=-\partial_1-\partial_2$, we can write equation \eqref{h1eq} as 
$$
(-(x_1-x_2)^2\partial_1\partial_2+(x_1-x_2)(\la_1\partial_2-\la_2\partial_1)
+(1-a^{-1})x_1^2\partial_1(\partial_1+\partial_2)-
$$
$$
-(1-a^{-1})x_1(\la_1(\partial_1+\partial_2)+\la_3\partial_1)-\frac{1}{2}\la_1(\la_2+(1-a^{-1})\la_3))\eta=(1-a)\mu_1\eta. 
$$
Let us set $x_1=z$, $x_2=1$. Homogeneity of $\eta$ implies that $x_2\partial_2\eta=-z\partial \eta+\beta\eta$, where $\partial:=\partial_z$. Thus setting $\phi(z):=\eta(z,1)$, equation \eqref{h1eq} takes the form
$$
(-(z-1)^2\partial (-z\partial+\beta)+(z-1)(\la_1(-z\partial+\beta)-\la_2\partial)
+(1-a^{-1})z^2\partial(\partial-z\partial+\beta)-
$$
$$
-(1-a^{-1})z(\la_1(\partial-z\partial+\beta)+\la_3\partial)-\frac{1}{2}\la_1(\la_2+(1-a^{-1})\la_3))\phi=(1-a)\mu_1\phi,
$$
which can be simplified to 
$$
L\phi=\Lambda \phi, 
$$
where 
$$
L=z(z-1)(z-a)\partial^2-
$$
$$
-((\beta+\lambda_1-1) z^2-(a(2\beta-\la_2-\la_3-2)+\lambda_1+\lambda_3)z+a(\beta-\lambda_2-1))\partial
+\beta\lambda_1z
$$
and 
$$
\Lambda=a(1-a)\mu_1+\frac{1}{2}\la_1(a(2\beta+\la_2)+(a-1)\la_3). 
$$
Note that $L$ is a Darboux operator. Let us compute its parameters $a_i$ in terms of $a_i^*:=\la_i+1$. 

Near $z=a$ we have $L\sim a(a-1)((z-a)\partial^2-(\la_1+\la_2+\la_3-\beta+1)\partial)$, so 
$$
a_1=\la_1+\la_2+\la_3-\beta+2=\frac{\la_1+\la_2+\la_3+\la_4}{2}=\frac{a_1^*+a_2^*+a_3^*+a_4^*}{2}. 
$$
Near $z=1$ we have $L\sim (a-1)((z-1)\partial^2-(\la_3-\beta-1)\partial)$, so 
$$
a_2=\la_3-\beta=\frac{-\la_1-\la_2+\la_3+\la_4}{2}=\frac{-a_1^*-a_2^*+a_3^*+a_4^*}{2}. 
$$
Near $z=0$ we have $L\sim a(z\partial^2-(\beta-\la_2-1)\partial)$, so 
$$
a_3=\beta-\la_2=\frac{\la_1-\la_2+\la_3-\la_4}{2}=\frac{a_1^*-a_2^*+a_3^*-a_4^*}{2}. 
$$
Finally, near $z=\infty$ we have $L\sim w\partial_w^2+(\beta+\la_1+1)\partial_w+\beta\la_1w^{-1}$, where $w=1/z$, so 
$$
a_4=\beta-\la_1=\frac{-\la_1+\la_2+\la_3-\la_4}{2}=\frac{-a_1^*+a_2^*+a_3^*-a_4^*}{2}. 
$$

Consider now the "untwisted" special case $a_i=0$. In this case we also have $a_i^*=0$. 
Moreover, it is easy to show by a direct computation that the equation $L\phi=\Lambda\phi$ is equivalent to 
the equation 
$$
\left(\partial_t^2  - \sum_{i=1}^N \frac{\la_i(\la_i+2)/4}{(t-z_i)^2} - \sum_{i=1}^N \frac{\mu_i}{t-z_i}\right)\Phi=0
$$
corresponding to the oper \eqref{qt}. 

\subsection{Relation to Okamoto transformations of Painlev\'e VI} 
\label{realmonod} For general $a_i$ the null space of the oper \eqref{qt} may
  also be interpreted as an eigenspace of a Darboux operator, but with
  ``dual'' parameters $a_i^*$ instead of $a_i$. Also note that together
  with the ``obvious'' symmetries $S_4\ltimes \mbb Z_2^4$ (the Weyl
  group of type $B_4$) acting on the parameters $a_i$, the above
  transformation $\mbf a=\sigma(\mbf a_*)$ generates the
  Weyl group of type $F_4$.  This is a manifestation of the hidden
  $F_4$ symmetry of the Painlev\'e VI equation discovered by Okamoto
  \cite{Ok1,Ok2,AL}.\footnote{We thank D. Arinkin for pointing out this connection.} 
  
 Note that by \cite{CM}, Proposition 5.1, Okamoto transformations preserve real monodromy (since real monodromy is, at least locally, characterized by the condition that the functions $q_i,q_{ij}$ of the monodromy from \cite{CM}, Proposition 5.1 are real-valued). Thus, for $N=4$ and $X=\mbb P^1$ the condition that the oper attached the Gaudin system has real monodromy is equivalent to the Gaudin system itself having real monodromy. 

\section{Eigenfunctions and monodromy for Darboux operators} 

\subsection{Local holomorphic eigenfunctions of $L$}\label{holo} 
Let us study (multivalued) local holomorphic eigenfunctions of Darboux operators $L$ near their singularities. 

Fix one of the singular points $p\in E/\mbb Z_2=\mbb P^1$ and let $a_p=\alpha$ and $w$ be the local coordinate at $p$ on $\mbb P^1$. Assume first that the local monodromy around $p$ is semisimple (e.g., $\alpha\notin \mbb Z$). Then it is easy to see that the holomorphic eigenfunctions of $L$ are locally linear combinations of 
\begin{equation}\label{eq1}
\psi_1(w)=f_1(w),\quad \psi_2(w)=f_2(w)w^{\alpha},
\end{equation}
where $f_i$ are holomorphic near $0$ and equal $1$ at $0$. 
Here for $p=\infty\in \mbb P^1$ the local coordinate near $p$ on $\mbb P^1$ is $w=1/z$, so we should be careful to take into account that we are working not with functions but with ``sections of $O(s)$". 
 
On the other hand, if the local monodromy is not semisimple (so in particular, $\alpha\in \mbb Z$) then similarly for $\alpha<0$ we have  
$$
\psi_1(w)=f_1(w), \quad \psi_2(w)=f_1(w)\log w+f_2(w)w^{\alpha},
$$
while for $\alpha\ge 0$ we have 
\begin{equation}\label{eq2}
\psi_1(w)=f_1(w)w^{\alpha}, \quad \psi_2(w)=f_1(w)w^{\alpha}\log w+f_2(w).
\end{equation}

\subsection{Irreducibility of monodromy} 

\begin{lemma}\label{irre} (i) The monodromy representation of the equation $L\Psi=\Lambda\Psi$ is reducible for some $\Lambda$ if and only if there exist $\varepsilon_p=\pm 1$ such that $\frac{1}{2}\sum_{p:2p=0} \varepsilon_p a_p$ is a positive integer. 

(ii) For fixed $\mbf a$ this monodromy representation is irreducible for all but finitely many values of $\Lambda$. 
\end{lemma}  

\begin{proof} For the proof we may replace the operator $L$ with $\widetilde{L}$, since these operators are obtained from each other by conjugation by a (multivalued) function, so the corresponding monodromy representations differ by tensoring with a character of $\pi_1$. 

(i) Let $R$ be the set of $\Lambda\in \mbb C$ for which the monodromy of the equation $\widetilde{L}\Psi=\Lambda\Psi$ is reducible. Let $\Lambda\in R$, and let $\Psi$ be a solution of this equation which is an eigenvector for the monodromy. Consider the function $f=\partial_u \Psi/\Psi$ on $E$. Then $f$ is an elliptic function with only first order poles satisfying the Riccati equation 
$$
f'+f^2=\Lambda-\sum_{p\in E: 2p=0} \left(\frac{1}{4}-a_p^2\right)\wp(u-p).
$$
Since $\Psi$ is an eigenvector of $\mbb Z_2$ acting by $u\to -u$, we see that $f$ is odd. Also, the Riccati equation for $f$ implies that the residue of $f$ at any pole $q\in E$ such that $2q\ne 0$ must equal $1$. These poles come in pairs symmetric under $u\mapsto -u$, so the number of these poles is even; let us denote it by $2m$.  
Also let $\alpha_p$ be the residue of $f$ at a point $p\in E$ such that $2p=0$. Then $\alpha_p=\frac{1}{2}-\varepsilon_p a_p$, where $\varepsilon_p=\pm 1$. Thus we have 
$$
\frac{1}{2}\sum_{p:2p=0} \varepsilon_p a_p=m+1,
$$
as claimed. 

To prove the converse, it suffices to assume that $\frac{1}{2}\sum_{p: 2p=0}a_p=s+1$ is a positive integer, i.e., 
$s\in \mbb Z_{\ge 0}$ (indeed, $\widetilde{L}$ depends on $a_p^2$, so we can assume that $\varepsilon_p=1$ for all $p$). In this case we claim that the operator $L$ preserves the space $\mbb C[z]_{\le s}$ of polynomials of $z$ of degree $\le s$. Indeed, $L$ clearly maps $\mbb C[z]_{\le n}$ to $\mbb C[z]_{\le n+1}$ for any $n\ge 0$, so it preserves $\mbb C[z]_{\le n}$ if and only if the coefficient of $z^{n+1}$ in $Lz^n$ vanishes. This boils down to the equation 
$$
n(n+2)+c_0n-s(s+2+c_0)=0,  
$$
which is satisfied for $n=s$. This implies the statement, since any eigenvector of $L$ in $\mbb C[z]_{\le s}$
is an invariant of the monodromy of the equation $L\psi=\Lambda\psi$.   

(ii) First, we claim that $R$  is a proper closed subset of $\mbb C$. Indeed, it is clearly closed, since reducibility of monodromy is a closed condition. Now let $\Lambda\in R$ and let $\Psi$ be a (multivalued) eigenfunction of $\widetilde{L}$ on $E$  with eigenvalue $\Lambda$ which is also an eigenvector of the monodromy. Let $\xi\in \mbb C^\times$ be the multiplier of $\Psi$ along a cycle $C$ in $E$. Since every branch of $\Psi(u)$ is a multiple of a branch of $\Psi(-u)$ and the map $u\mapsto -u$ reverses $C$, we have $\xi=\xi^{-1}$, so $\xi=\pm 1$. This means that $R\ne \mbb C$, as claimed.\footnote{In fact, it follows that $R$ is discrete, since it is an analytic subset of $\mbb C$.}

Now note that we have finitely many choices of $\varepsilon_p$. So if $R$ is infinite then in particular 
there are infinitely many eigenvalues in $R$ for one of such choices, hence an infinite collection $T$ of positions of the $2m$ poles of $f$ which give rise to (uniquely determined) odd solutions of the Riccati equation with fixed values of $\alpha_p$ and $m$. Since $T$ is infinite, its Zariski closure $\overline{T}$ has positive dimension. But $\Lambda$ is a rational function on $\overline{T}$. Thus $R$ must be a 1-dimensional constructible subset of $\mbb C$, which contradicts the fact that $R$ is a proper closed subset. 
\end{proof} 

\begin{remark} Another proof of Lemma \ref{irre}(ii) together with a
  more precise description of the set of $\Lambda$ with reducible
  monodromy can be obtained using the fact that if the monodromy is
  reducible then there is an elementary eigenfunction, which can be
  computed explicitly. For example, in the case when all $a_i$ are
  equal, it follows from the above that reducibility can only occur
  if $s$ is a half-integer, and the corresponding eigenfunctions and
  eigenvalues are classified and discussed in \cite{GV} and references
  therein. 
\end{remark}

\subsection{Joint eigenfunctions of $L$ and $L^\dagger$}
Now set $L^\dagger:=\overline{L^*}$. The operators $L,L^\dagger$ can be viewed as commuting operators acting on sections of the complex (non-holomorphic) line bundle $O(s)\otimes \overline{O(-2-s)}$ on $\mbb P^1$. This bundle makes sense if and only if the function 
$$
z^{s}\overline{z^{-2-s}}=|z|^{-2+2i{\rm Im}(s)}e^{2i{\rm Re}(s+1){\rm arg}(z)}
$$ 
is single valued, i.e. when ${\rm Re}(s)$ is a half-integer. This means that 
$$
s\in \frac{n}{2}+i\mbb R
$$
for some integer $n$. So we impose this restriction on $s$ from now on. 
 
Let us now consider the behavior of joint eigenfunctions\footnote{Strictly speaking, the $\psi$'s should be called eigensections (of the line bundle $O(s)\otimes \overline{O(-2-s)}$), but we will often slightly abuse terminology and call them eigenfunctions, trivializing this line bundle on $\mbb P^1\setminus \infty=\mbb A^1$.} $\psi$ of $L,L^\dagger$ near one of the singularities $p\in E/\mbb Z_2=\mbb P^1$.
 
Consider first the case $a_p\notin \mbb Z$. 

\begin{lemma}\label{not2Z} Any joint eigenfunction of $L,L^\dagger$ which is single-valued near $p$ is a linear combination of $\eta_1,\eta_2$, where $\eta_1$ is smooth near $p$ and 
$$
 \eta_2\sim w^{a_p}\overline{w}^{-\overline{a_p}}=|w|^{\pm 2i{\rm Im}(a_p)}e^{\pm 2i{\rm Re}(a_p){\rm arg}(w)}.
$$ 
\end{lemma}

\begin{proof} Follows from formula \eqref{eq1} in Subsection \ref{holo}. 
\end{proof} 

We thus obtain the following proposition. 

\begin{proposition}\label{loceig} Suppose the operators  
$L,L^\dagger$ have infinitely many joint eigenfunctions (up to scaling). Then ${\rm Re}(a_p)$ is a half-integer for every $p$ and ${\rm Re}(\sum_p a_p)=2{\rm Re}(s+1)$ is an integer. 
\end{proposition} 

\begin{proof} Note that $\eta_2$ is single-valued with respect to $w$ if and only if ${\rm Re}(a_p)$ is a half-integer. Thus, if ${\rm Re}(a_p)\notin \frac{1}{2}\mbb Z$, 
we have just a 1-dimensional space of single-valued local joint eigenfunctions spanned by $\eta_1$, which is a product of a holomorphic and anti-holomorphic function. But this cannot happen in the situation of irreducible monodromy. Thus, the statement follows from Lemma \ref{irre}(ii). 
\end{proof} 

Now consider the behavior of joint eigenfunctions $L,L^\dagger$ near $p$ 
in the case $a_p=0$. We have 

\begin{lemma}\label{a_p=0} If $a_p=0$ then any joint eigenfunction of $L,L^\dagger$ which is single-valued 
near $p$ is a linear combination of $\eta_1$, $\eta_2$, where 
\begin{equation}\label{eq7}
\eta_1\sim 1,\quad \eta_2\sim \log|w|. 
\end{equation} 
\end{lemma} 

\begin{proof} Follows from formula \eqref{eq2} in Subsection \ref{holo}. 
\end{proof} 

\subsection{Global eigenfunctions} \label{gloeig}

From now on for simplicity let us restrict ourselves to the special case $a_p\in i\mbb R$; thus, $s\in -1+i\mbb R$. 
In this case by Lemma \ref{irre}(i) the monodromy of $L\psi=\Lambda\psi$ is irreducible for any $\Lambda$. 
Also the local monodromy operators of this equation have positive eigenvalues. 
 
There is a natural positive definite 
Hermitian inner product on sections of the line bundle $O(s)\otimes \overline{O(-2-s)}$ given by 
$(f,g)=\int_{\mbb P^1}f\overline{g}$ (using that the canonical bundle of $\mbb P^1$ is $O(-2)$). 
Thus we can define the Hilbert space $\mathcal{H}:=L^2(\mbb P^1,O(s)\otimes \overline{O(-2-s)})$, where 
$||f||^2:=(f,f)$.  
 
 \begin{proposition}\label{l2prop} (i) Any joint eigenfunction of $L$ and $L^\dagger$ belongs to $\mathcal{H}$. 

(ii) If $\psi$ is a joint eigenfunction of $L,L^\dagger$ such that $L\psi=\Lambda\psi$ then $L^\dagger \psi=\overline{\Lambda}\psi$. 

(iii) Joint eigenfunctions of $L,L^\dagger$ with different eigenvalues are orthogonal. 

(iv) Joint eigenspaces of $L$ and $L^\dagger$ are at most $1$-dimensional. 

(v) The set of joint eigenfunctions of $L,L^\dagger$ (up to scaling) is countable. 
\end{proposition} 

\begin{proof} (i) follows from Lemma \ref{not2Z} and Lemma \ref{a_p=0}. 

(ii), (iii) follow as usual by integration by parts, using Lemma \ref{not2Z} and Lemma \ref{a_p=0}. 

(iv) Every joint eigenfunction with given eigenvalues corresponds to an invariant Hermitian\footnote{Here and below Hermitian forms are {\bf not} assumed to be positive definite unless specified otherwise.} pairing between the monodromy representations of the equations $L\psi=\Lambda\psi$ and $L^\dagger\psi=\overline{\Lambda}\psi$. But such a pairing, if exists, is unique up to scaling since the monodromy representation is irreducible. 

Finally, (v) follows from (iii) and separability of $\mathcal{H}$. 
\end{proof} 
  
Note that since the conjugacy classes of monodromies of the equation $L\Psi=\Lambda\Psi$
at its singular points are fixed by fixing $\mbf a=(a_p)$, the possible monodromy representations  form a (2-dimensional) affine complex algebraic surface $X(\mbf a)$. Namely, $X(\mbf a)$ is the {\it character variety}, which consists of homomorphisms $\pi_1(\mbb P^1\setminus \lbrace{e_1,e_2,e_3,\infty\rbrace})\to SL_2(\mbb C)$ with prescribed conjugacy classes of images of the loops around the four punctures; in fact, it is well known that this is a cubic surface with three lines forming a triangle removed, see e.g. \cite{IIS,O} and references therein. The actual monodromy representation of the equation $L\Psi=\Lambda\Psi$ is a certain point $\rho(\Lambda)\in X(\mbf a)$, which gives a (transcendental) parametrized  complex curve in $X(\mbf a)$ (here $\rho(\Lambda)$ is holomorphic in $\Lambda$).\footnote{We make the monodromy representation unimodular by renormalizing it by the character that sends 
the loop around $p$ to $\exp(\pi ia_p)$.}  
 
 \begin{proposition}\label{eigecond} The system $L\psi=\Lambda\psi$, $L^\dagger \psi=\overline{\Lambda}\psi$ has a nonzero single-valued solution if and only if the representation $\rho(\Lambda)$ preserves a nonzero (equivalently, nondegenerate) Hermitian form. 
 \end{proposition}  
 
 \begin{proof} Since $L(\mbf a,\tau)^*=L(-\mbf a,\tau)$, we have $\widetilde{L^*}=\widetilde{L}$. Therefore, 
the statement follows from the fact that a single-valued joint eigenfunction corresponds to a Hermitian pairing between the spaces of solutions of the equation $L\psi=\Lambda\psi$ and $L^*\psi=\Lambda\psi$. 
\end{proof} 

Thus, eigenvalues $\Lambda$ occur at the intersection of the holomorphic curve $\rho(\Lambda)\subset X(\mbf a)$ 
and its conjugate anti-holomorphic curve $\rho(\Lambda)^\dagger$ in the complex surface $X(\mbf a)$.  
 
Now recall that since $a_p$ are imaginary, the eigenvalues of the local monodromies are real positive, so these monodromies are either hyperbolic or parabolic elements of $SL_2(\mbb C)$. Therefore, the signature of the Hermitian form preserved by $\rho(\Lambda)$ is necessarily $(1,1)$. Hence, the monodromy group is contained in $SU(1,1)$, which is conjugate to $SL(2,\mbb R)$ inside $SL(2,\mbb C)$. Conversely, if monodromy is real then the monodromy representation is Hermitian and we get a joint eigenfunction. Thus from Proposition \ref{eigecond} we obtain 

\begin{corollary}\label{realmon} Joint eigenfunctions of $L,L^\dagger$ (up to scaling) are in bijection with numbers $\Lambda\in \mbb C$ for which the equation $L\Psi=\Lambda\Psi$ has real monodromy. 
\end{corollary} 

In particular, we see that eigenvalues $\Lambda$ with real monodromy form a discrete set $\Sigma$ (since it is a real analytic set which is countable by Proposition \ref{l2prop}(v)). 

\begin{proposition}\label{l2prop1} If $\psi\in \mathcal{H}$ is nonzero and $L\psi=\Lambda\psi$ outside the four singular points then 
$\Lambda\in \Sigma$ and $L^\dagger\psi=\ol{\Lambda}\psi$. 
\end{proposition} 

\begin{proof} Locally near a non-singular point $z_0\in \mbb P^1$ we can write $\psi$ in the form 
$$
\psi(z,\ol z)=\psi_1(z)\ol{\eta_1(z)}+\psi_2(z)\ol{\eta_2(z)},
$$
where $\psi_1,\psi_2$ is a basis of holomorphic solutions of $L\psi=\Lambda\psi$ and $\eta_1,\eta_2$ are some linearly independent holomorphic functions (by irreducibility of monodromy of $L\psi=\Lambda\psi$). The functions $\eta_1,\eta_2$ branch around singular points, but since $\psi$ is single-valued, the analytic continuation of $\eta_i$ around a closed path from $\pi_1(\mbb P^1\setminus \lbrace{0,a,1,\infty\rbrace},z_0)$ is a linear combination of $\eta_1$ and $\eta_2$, and this implements the representation $\ol \rho^*$ of $\pi_1$, where $\rho$ is the monodromy representation of $L\psi=\Lambda\psi$. Thus, $\eta_1,\eta_2$ are a basis of solutions of a second order linear ODE with rational coefficients, which is smooth outside the four singular points. Also, since $\psi$ is an $L^2$ function, the functions $\eta_i$ have power  growth near the singular points, which implies that this ODE is Fuchsian. Moreover, the conditions that $\psi$ is single-valued and that it is in $L^2$ determines the characteristic exponents of this ODE at the singular point, showing that it must have the form $L^*\eta=\Lambda_*\eta$ for some $\Lambda_*\in \mbb C$. Thus $L^\dagger \psi=\ol{\Lambda_*}\psi$. Finally, by Proposition \ref{l2prop}(ii), we have $\Lambda_*=\Lambda$. Thus $\Lambda\in \Sigma$ and $L^\dagger \psi=\ol\Lambda\psi$ as claimed. 
\end{proof} 

\subsection{The special case $\mbf a=0$}\label{takh} 
In the special case $\mbf a=0$ eigenvalues $\Lambda$ of $L$ with real monodromy (also called real projective connections) were studied by Goldman \cite{G} and also by Faltings \cite{F}, who showed that these eigenvalues form a discrete set. One especially interesting point of this set gives rise to the uniformizing connection for the Riemann surface $\mbb P^1\setminus\lbrace{0,a,1,\infty\rbrace}$ (i.e., the connection coming from the representation of this surface as a quotient of the hyperbolic plane by a Fuchsian group), and other real projective connections (or, equivalently, projective structures) can be obtained from the uniformizing connection by the so-called grafting procedure (\cite{G,Tan}). Also, the paper of Takhtajan \cite{T} explains how to find infinitely many {\it real} $\Lambda$ with real monodromy if $e_1,e_2,e_3$ are real (i.e., $0<a<1$). Namely, following the work of \linebreak F. Klein, D. Hilbert and V. I. Smirnov, it is explained in \cite{T} that infinitely many real $\Lambda$ with real monodromy are provided by solutions of the following three Sturm-Liouville problems:

(1) Find eigenvalues $\Lambda$ for which there is a nonzero holomorphic eigenfunction in a neighborhood of the interval $[0,a]$.

(2) Find eigenvalues $\Lambda$ for which there is a nonzero holomorphic eigenfunction in a neighborhood of the interval $[a,1]$.

(3) Find eigenvalues for which the following property holds. If $f_0$ is a nonzero eigenfunction holomorphic at $0$ and $f_+, f_-$ are analytic continuations of $f_0$ along $[0,1]$ passing the point $a$ above (respectively below), then $f_++f_-$ is holomorphic near $1$.

We claim that there are also infinitely many eigenvalues $\Lambda$ with real monodromy that are {\bf not} given by this procedure. They can be constructed as follows. Consider a path $\gamma$ (i.e., a pure braid on three strands, with strands of $0$ and $1$ being straight line segments) that moves the point $a$ around (avoiding $0,1$) and eventually brings it back to the original position. Take some eigenvalue $\Lambda=\Lambda(a)$ that solves problem (1) and deform it along $\gamma$ so that the monodromy stays real (this can be done uniquely since real monodromy points are discrete and don't bifurcate). Denote the final value by $\Lambda_\gamma(a)$. Then $\Lambda_\gamma(a)$ solves the following problem:

(1,$\gamma$) Find eigenvalues for which there is a nonzero holomorphic eigenfunction in a neighborhood of the curve $\gamma[0,a]$, where $\gamma[0,a]$ is the path from $0$ to $a$ obtained by deforming $[0,a]$ along $\gamma$ (avoiding the point $1$).

If $\gamma$ is sufficiently complicated, this path is clearly not a straight line interval -- it is in general a complicated path from $0$ to $a$. It is easy to see that $\Lambda_\gamma(a)$ in general {\it cannot} be a solution of any of the problems (1)-(3) (simultaneously with (1,$\gamma$)); for instance, there clearly cannot be a solution holomorphic along a path connecting $0,a,1$ -- otherwise this solution would be a polynomial and the monodromy would be reducible. On the other hand, the monodromy for the eigenvalue $\Lambda_\gamma(a)$ is real by construction. 

Note that $\Lambda_\gamma(a)$ in general will {\it not} be real, since otherwise the solution holomorphic near $0$ will extend holomorphically along $\gamma$ and also along $\overline{\gamma}$, which is impossible. Thus, there are non-real eigenvalues $\Lambda$ which still give rise to real monodromy. 

\section{Essentially self-adjoint algebras of unbounded operators} 

\subsection{Preliminaries on unbounded self-adjoint operators}\label{usa}

Let us recall basics on unbounded self-adjoint operators
(\cite{AG,RS}). Let $\mathcal{H}$ be a separable Hilbert space,
$(\cdot,\cdot)$ the Hermitian inner product on
$\mcH$, $V\subset \mathcal{H}$ a dense subspace, and $A: V\to \mathcal{H}$ a linear operator. Let $\Gamma_A\subset V\times \mathcal{H}\subset \mathcal{H}\times \mathcal{H}$ be the graph of $A$. One says $A$ is {\it closed} if $\Gamma_A\subset \mathcal{H}\times \mathcal{H}$ is a closed subspace. In general, let $\overline{\Gamma}_A$ be the closure of $\Gamma_A$. Let $\overline{V}$ be the image of the projection $\pi:\overline{\Gamma}_A\to \mathcal{H}$ from $\overline{\Gamma}_A$ to the first factor $\mathcal{H}$. One says that $A$ is {\it closable} if $\pi$ is injective, i.e., $\overline{\Gamma}_A$ is the graph of a linear operator $\overline{A}: \overline{V}\to \mathcal{H}$. In this case, 
$\overline{A}$ is called the {\it closure} of $A$. Clearly, $A$ is closed iff it is closable and $\overline{V}=V$ (i.e., $\overline{A}=A$). 

Let $V^\vee=V^\vee_A$ be the subspace of all $u\in \mathcal{H}$ for which the linear functional $v\mapsto (Av,u)$ on $V$ is continuous. If so, then this linear functional extends by continuity to $\mathcal{H}$, hence by the Riesz representation theorem $(Av,u)=(v,w)$ for a unique $w\in {\mathcal H}$. Define a linear operator $A^\dagger: V^\vee\to \mathcal{H}$ by $A^\dagger u:=w$; it is called the {\it adjoint operator} to $A$. It is easy to show that $A^\dagger$ is always closed. 

The operator $A$ is called {\it symmetric} (or, in some texts, Hermitian symmetric) if $(Av,u)=(v,Au)$ for all $u,v\in V$. In this case we have a skew-Hermitian form of $V^\vee$ given by 
$$
\omega(u,w)=(A^\dagger u,w)-(u,A^\dagger w),
$$ 
and $\overline{V}={\rm Ker}\omega$. Moreover, the restriction of $A^\dagger$ to $\overline{V}$ is the closure $\overline{A}$ of $A$; thus $A$ is closable and $\overline{A}$ is symmetric. By definition $\omega$ descends to a nondegenerate skew-Hermitian form on $V^\vee/\overline{V}$ (which we will also call $\omega$).

One says that a closed symmetric operator $A$ is {\it self-adjoint} if $V=V^\vee$, i.e., $A=A^\dagger$. One says that a closable symmetric operator $A$ is {\it essentially self-adjoint} if its closure $\overline{A}$ is self-adjoint. This is equivalent to saying that $A^\dagger$ is symmetric, i.e., $\omega=0$. 

A {\it self-adjoint extension} of a symmetric operator $A$ is an extension $A'$ of $A$ defined on a subspace $\overline{V}\subset V'\subset V^\vee$ such that $A'$ is self-adjoint. It is clear that $A'$ is the restriction of $A^\dagger$ to $V'$, hence it is defined by the choice of $V'$. It is easy to show that $V'$ defines a self-adjoint extension if and only if $V'/\overline{V}$ is a Lagrangian subspace of $V^\vee/\overline{V}$ under $\omega$, i.e., 
a subspace $Y$ such that $Y^\perp=Y$. Thus the set of self-adjoint extensions of $A$ is in bijection with the Lagrangian Grassmannian of $V^\vee/\overline{V}$ (which may be empty; e.g. this happens if $V^\vee/\overline{V}$ is odd-dimensional, or more generally a finite dimensional space with the Hermitian form $i\omega$ having nonzero signature). Thus, $A$ is essentially self-adjoint if and only if it has a unique self-adjoint extension, namely its closure $\overline{A}$. So having an essentially self-adjoint operator is almost as good as having a self-adjoint operator.

An important criterion of essential self-adjointness of a symmetric operator $A$ 
is the following theorem of von Neumann: 

\begin{theorem}\label{vn} A symmetric operator $A$ is {\bf not} essentially self-adjoint if and only if the space $V^\vee$ contains a vector $v\ne 0$ such that $Av=iv$ or $Av=-iv$. 
\end{theorem} 

The notion of (essentially) self-adjoint operator is important because of the following spectral theorem. 

\begin{theorem}\label{sth} An operator $A: V\to \mathcal{H}$ is self-adjoint if and only if there exists a finite measure space $(X,\mu)$ and an isometry $U: \mathcal{H}\to L^2(X,\mu)$ 
such that $UAU^{-1}=M_f$ is the operator of multiplication by a real measurable function 
$f$ on $X$ so that $U(V)$ is the space of $g\in L^2(X,\mu)$ such that $fg$ is also in $L^2(X,\mu)$. 
\end{theorem} 

Theorem \ref{sth} allows us to define the {\it spectral resolution} of a self-adjoint operator $A$. Namely, 
for any $t\in \mbb R$ let $X_{t}\subset X$ be the subset of $x\in X$ with $f(x)\le t$ and 
define the operator $P_A(t)$ on $\mathcal{H}$ corresponding to multiplication by 
the indicator function $\chi_{X_{t}}$ on $L^2(X,\mu)$. It can be shown that $P_A(t)$ are commuting orthogonal projectors with $\lim_{t=-\infty}P_A(t)=0$, $\lim_{t\to +\infty}P_A(t)=1$, independent on the choice of $(X,\mu,U)$ in the spectral theorem, and one has 
$$
A=\int_{\mbb R}tdP_A(t).
$$
The family $P_A(t)$ is called the {\it spectral resolution} of $A$. 

Also for any $z\notin \mbb R$ one can define the operator $R_A(z):=(z-A)^{-1}=\int_{\mbb R}\frac{1}{z-t}dP_A(t)$, which corresponds to the operator of multiplication by the function $(z-f)^{-1}$ in $L^2(X,\mu)$. This operator is bounded for every such $z$ (since $|(z-f)^{-1}|\le |{\rm Im}z|^{-1})$ and is called the {\it resolvent} of $A$. Finally, any self-adjoint operator $A$ defines a unitary representation 
of $\mbb R$ on $\mathcal{H}$ (i.e., a 1-parameter group of unitary operators) given by 
$$
\rho(t)=e^{itA}:=\int_{\mbb R}e^{it}dP_A(t),
$$
the operator of multiplication by $e^{it}$ in $L^2(X,\mu)$. 

Let us say that two self-adjoint operators $A$ and $B$ (in general, defined on different domains) {\it strongly commute} if one of the following equivalent conditions hold:

(1) their spectral resolutions commute, $[P_A(t),P_B(s)]=0$;  

(2) their resolvents commute, $[R_A(z),R_B(w)]=0$; 

(3) the corresponding $1$-parameter subgroups commute, $e^{itA}e^{isB}=e^{isB}e^{itA}$; in other words, we have a unitary representation of 
$\mbb R^2$ on $\mathcal{H}$ given by $\rho(t,s)=e^{itA}e^{isB}$. 

We have the following generalization of Theorem \ref{sth} to collections of strongly commuting self-adjoint operators. 

\begin{theorem}\label{sth1} Let $V_i\subset \mathcal{H}$ be dense subspaces and 
$A_i: V_i\to \mathcal{H}$ be a collection of linear operators, $i\in 1,\ldots,n$. Then $A_i$ are pairwise strongly commuting self-adjoint operators if and only if there is a finite measure space $(X,\mu)$ and an isometry $U: \mathcal{H}\to L^2(X,\mu)$ such that $UA_iU^{-1}=M_{f_i}$ is the operator of multiplication by a {\it real} measurable function 
$f_i$ on $X$ so that for each $i$, $U(V_i)$ is the space of $g\in L^2(X,\mu)$ such that $f_ig$ is in $L^2(X,\mu)$.
\end{theorem}

However, the notion of strong commutativity of self-adjoint operators turns out to be rather subtle. 
Namely, as shown in \cite{N}, Section 10, even if $A,B$ are defined and commute on a common dense invariant domain $V$ on which any real linear combination $aA+bB$, $a,b\in \mbb R$, is essentially self-adjoint, this property may fail, i.e., the one-parameter groups $e^{itA}$ and $e^{isB}$ may fail to commute. To overcome this difficulty, we will use the following theorem, which is a special case of a remarkable result of E. Nelson (\cite{N}, Theorem 5). 

\begin{theorem}\label{ne1} Let $A_1,...,A_n$ be symmetric operators defined on a common dense invariant domain 
$V$ in a separable Hilbert space $\mathcal{H}$ and let $D=A_1^2+...+A_n^2$. If $D$ is essentially self-adjoint on $V$ then any real linear combination of the operators $A_j$ is also essentially self-adjoint on $V$, and the action of $\g=\mbb R^n$ on $V$ defined by the operators $iA_j$ integrates to a unique unitary representation $\rho$ of $G=\mbb R^n$ on $\mathcal{H}$, given by $\rho(t_1,...,t_d)=\prod_j \exp(i t_jA_j)=\exp(i\sum_jt_jA_j)$ (i.e., the differential of $\rho$ at the identity $1\in G$ restricted to $V$ is given by the natural map of $\g\to {\rm End}V$). 
\end{theorem} 

Motivated by this, we make the following definition. 

\begin{definition} (i) A densely defined linear operator $A: V\to \mathcal{H}$ is {\it normal} if there exists a finite measure space $(X,\mu)$ and an isometry $U: \mathcal{H}\to L^2(X,\mu)$ such that $UAU^{-1}=M_f$ is the operator of multiplication by a {\it complex} measurable function $f$ on $X$ so that $U(V)$ is the space of $g\in L^2(X,\mu)$ such that $fg$ is in $L^2(X,\mu)$. 

(ii) A closable operator $A: V\to \mathcal{H}$ is {\it essentially normal} if the closure of $A$ is normal. 
\end{definition} 

It is easy to see that if $A$ is normal then the domain of $A^\dagger$ coincides with the domain $V$ of $A$. 
Thus we can define the operators ${\rm Re}A=\frac{1}{2}(A+A^\dagger)$ and ${\rm Im}A=\frac{1}{2i}(A-A^\dagger)$
from $V$ to $\mathcal{H}$. It is clear that the operators ${\rm Re}A$, ${\rm Im}A$ are essentially self-adjoint and their closures strongly commute. Also, the domain of $A$ is the intersection of the domains of ${\rm Re}A,{\rm Im}A$. 
Conversely, if two self-adjoint operators $B,C$ strongly commute then by Theorem \ref{sth1} the operator 
$A:=B+iC$ defined on the intersection of the domains of $B$ and $C$ is normal, with $B={\rm Re}A$ and $C={\rm Im}A$. Thus, our definition of a normal operator 
is equivalent to that of \cite{N}, p.603. 

Let us say that two normal operators $A,B$ {\it strongly commute}  if the closures of their real and imaginary parts pairwise strongly commute. Then any normal operator $A$ strongly commutes with $A^\dagger$. Moreover, Theorem \ref{sth1} can be generalized as follows. 

\begin{theorem}\label{sth2} Let $V_i\subset \mathcal{H}$ be dense subspaces and 
$A_i: V_i\to \mathcal{H}$ be a collection of linear operators, $i = 1,\ldots,n$. Then $A_i$ are pairwise strongly commuting normal operators if and only if there is a finite measure space $(X,\mu)$ and an isometry $U: \mathcal{H}\to L^2(X,\mu)$ such that $UA_iU^{-1}=M_{f_i}$ is the operator of multiplication by a complex measurable function $f_i$ on $X$ so that $U(V_i)$ is the space of $g\in L^2(X,\mu)$ such that $f_ig$ is also in $L^2(X,\mu)$.
\end{theorem}

\subsection{Self-adjoint algebras of operators}\label{preli} Let $V\subset \mathcal{H}$ be a dense subspace of a separable Hilbert space $\mathcal{H}$. Let $\mcA$ be a commutative algebra of linear operators on $V$ with an antilinear involution $\dagger $ such that $(Av,u)=(v,A^\dagger u)$. Let $\mcA_{\mbb R}$ is the set of symmetric elements in $\mcA$, i.e., such that $A^\dagger=A$; then $\mcA=\mcA_{\mbb R}\otimes_{\mbb R}\mbb C$. 

\begin{definition}\label{es1a} The {\it Schwartz space} $S(\mcA)$ 
is the space of $u\in \mathcal{H}$ such that the linear functional 
$v\mapsto (Av,u)$ on $V$ is continuous for any $A\in \mcA$.
\end{definition}

It is clear that $V\subset S(\mcA)$. Also by the Riesz representation theorem, the 
action of $\mcA$ on $V$ extends canonically to $S(\mcA)$, via $(v,Au):=(A^\dagger v,u)$ for $v\in V$ 
and $u\in S(\mcA)$. Also note that if $\mcA=\mbb C[A]$ where $A: V\to V$ is symmetric 
then any eigenvector of $A^\dagger$ belongs to $S(\mcA)$. 

\begin{definition}\label{es2} We say that $\mcA_{\mbb R}$ is an {\it
    essentially self-adjoint algebra on} $V$ if every element
  $A\in \mcA_{\mbb R}$ is essentially self-adjoint on $S(\mcA)$.
\end{definition}

\begin{proposition}\label{esa1} (i) Let $A_1,...,A_n: V\to V$ be essentially normal operators such that $A_i^\dagger$ preserve $V$ and the normal operators $\overline{A}_i$ strongly commute. Let $\mcA=\mbb C[A_1,...,A_n,A_1^\dagger,...,A_n^\dagger]$ with $\dagger$ given by switching $A_i$ and $A_i^\dagger$. Then the algebra $\mcA_{\mbb R}$ is essentially self-adjoint. 

(ii) If $A: V\to V$ is a symmetric operator and $\mcA=\mbb C[A]$ with $\dagger$ being complex conjugation, then $\mcA_{\mbb R}$ is essentially self-adjoint on $V$ if and only if so is $A$.
\end{proposition} 

\begin{proof} (i) By Theorem \ref{sth2} there is a finite measure space $(X,\mu)$ and an isometry $U: \mathcal{H}\to L^2(X,\mu)$ which for all $i$ transforms $\overline{A}_i$ into the operator of multiplication by a complex measurable function $f_i$ on $X$. 
Moreover, for any polynomial $P\in \mbb C[x_1,...,x_n,y_1,...,y_n]$ we have 
$$
V_{P(A_1,...,A_n,A_i^\dagger,...A_n^\dagger)}^\vee=\lbrace{g\in L^2(X,\mu): P(f_1,...,f_n,\bar f_1,...,\bar f_n)g\in L^2(X,\mu)\rbrace}. 
$$
Thus, 
$$
S(\mcA)=\lbrace{g\in L^2(X,\mu): P(f_1,...,f_n,\bar f_1,...,\bar f_n)g\in L^2(X,\mu)\ \forall P\rbrace},
$$
and for every real $P$ the operator $P(A_1,...,A_n,A_1^\dagger,...,A_n^\dagger)$ (which corresponds under $U$ to multiplication by $P(f_1,...,f_n,\bar f_1,...,\bar f_n)$) is essentially self-adjoint on $S(\mcA)$, as desired. 

(ii) By (i), it suffices to establish just the ``only if" direction. Assume that $A$ is not essentially self-adjoint on $V$. Then by Theorem \ref{vn}, the domain $V^\vee$ of the adjoint operator to $A$ contains a vector $v\ne 0$ such that $Av=iv$ or $Av=-iv$. Then $v\in S(\mcA)$. This implies that $A$ is not symmetric, hence not essentially self-adjoint on $S(\mcA)$, i.e., $\mcA_{\mbb R}$ is not essentially self-adjoint on $V$. 
\end{proof} 

\begin{remark}\label{remsa} 1. It is clear that $S(\mcA)=\cap_{A\in \mcA}V_A^\vee$.  

2. Let $\mcA_{\mbb R}$ be essentially self-adjoint on $V$ and $S\supset V$ be a subspace of $\mathcal{H}$ such that each $A\in \mcA_{\mbb R}$ extends (necessarily uniquely) to an essentially self-adjoint endomorphism of $S$. Then 
$S\subset S(\mcA)$; in particular, the closure of $(A,S)$ coincides with the closure of $(A,S(\mcA))$, hence is independent on $S$, (i.e., $S$ is ``essentially unique"). Indeed, by definition $S\subset V_A^\vee$ for any $A\in \mcA_{\mbb R}$, so the statement follows from (1). 

3.  It is clear that if all $T\in \mcA_{\mbb R}$ are essentially
self-adjoint on $V$ then $\mcA_{\mbb R}$ is an essentially
self-adjoint algebra. However, the converse is false: if $\mcA_{\mbb
  R}$ is essentially self-adjoint on $V$ then an element $A\in
\mcA_{\mbb R}$ might fail to be essentially self-adjoint on $V$. For example, take $\mcA=\mbb C[T]$ where $T$ is essentially self-adjoint on $V$; then $A:=T^2$ may fail to be essentially self-adjoint on $V$, see Example \ref{exam}(3). 

4. If $\mcA_{\mbb R}$ is essentially self-adjoint then any $A_1,A_2\in \mcA_{\mbb R}$ strongly commute. This follows from Theorem \ref{ne1} for $n=2$ and $V=S(\mcA)$. Thus if $A_1,...,A_n: V\to V$ and 
the algebra $\mbb C[A_1,...,A_n]$ is essentially self-adjoint then $A_1,...,A_n$ define a unitary representation of the group $\mbb R^n$ on $\mathcal{H}$ given by $\rho(t_1,...,t_n)=e^{i(t_1A_1+...+t_nA_n)}$. Conversely, if such a representation exists then according to \cite{NS} and references therein, all elements of $\mathcal{A}_{\mbb R}$ are essentially self-adjoint on the Garding space $V$, thus so is $\mathcal{A}_{\mbb R}$ itself. 
\end{remark} 

\begin{proposition}\label{spectrumexists} If $\mcA_{\mbb R}$ is essentially self-adjoint then there is a finite measure space $(X,\mu)$, an isometry 
$U: \mathcal{H}\to L^2(X,\mu)$ and an algebra embedding $A\mapsto \psi(A)$ of $\mcA$ into the algebra 
of complex measurable functions on $X$ such that $UAU^{-1}=M_{\psi(A)}$ 
for $A\in \mcA$ and 
$$
U(S(\mcA))=\lbrace{g\in L^2(X,\mu): \psi(A)g\in L^2(X,\mu)\ \forall A\in \mcA\rbrace}
$$
In particular, all $A\in \mcA$ are essentially normal on $S(\mcA)$. 
\end{proposition} 

\begin{proof} This follows from Remark \ref{remsa}(4). 
\end{proof} 

\begin{example}\label{exaesaa} 1. Let $\mathcal{H}=L^2(\mbb R)$, $V=C^\infty_0(\mbb R)$ the subspace of smooth functions with compact support, and 
$\mcA=\mbb C[\partial]$. Then $S(\mcA)=\mathcal{H}\cap C^\infty(\mbb R)$, and it is well known that all elements of $\mcA_{\mbb R}$ are essentially self-adjoint on $S(\mcA)$ (hence $\mcA_{\mbb R}$ is essentially self-adjoint). The same holds if $\mcA=\mbb C[x]$, in which case $S(\mcA)$ is the space of $f$ such that $x^nf\in L^2(\mbb R)$ for all $n$. 

2. Let $\mathcal{H},V$ be as in (1) and $\mcA=\mbb C[H]$, where $H=-\partial^2+x^2$ is the quantum harmonic oscillator. Then it is well known that $S(\mcA)$ is the usual Schwartz space $\mathcal S(\mbb R)$, and $\mcA_{\mbb R}$ is essentially self-adjoint on $V$. 

3. Let $\mathcal{M}$ be a compact manifold, $E$ a complex vector bundle on $\mathcal{M}$ with a positive Hermitian metric,  and $\mcA$ be a commutative algebra of smooth differential operators acting on half-densities on $\mathcal{M}$ with values in $E$. Suppose that $\mathcal{A}$ is invariant under the adjunction anti-automorphism $L\mapsto L^\dagger$. Also assume that $\mcA$ is elliptic, i.e., for any $x\in \mathcal{M}$ and nonzero $p\in T^*_x\mathcal{M}$ there exists $A\in \mcA$ whose symbol does not vanish at $(x,p)$. In particular, this holds if $\mcA$ is holonomic, i.e., the common zero set (over $\mbb C$) of the symbols of all elements of $\mcA$ at each $x\in \mathcal{M}$ consists only of $0$. Let $V$ be the space of smooth half-densities and $\mathcal{H}$ the space of square integrable half-densities on $\mathcal{M}$ valued in $E$. Then we claim that $\mcA_{\mbb R}$ is essentially self-adjoint with $S(\mcA)=V$ (i.e., every element $A\in \mcA_{\mbb R}$ is essentially self-adjoint on $V$). Indeed, let $A=A_1$, then there exist $A_2,...,A_m\in \mcA$ such that $D=A_1^2+...+A_m^2$ is an elliptic operator. Then $D$ is well known to be essentially self-adjoint on $V$. Thus, by Theorem \ref{ne1}, $A$ is essentially self-adjoint on $V$. (This argument is similar to  \cite{NS}, Corollary 2.4). 
\end{example}

\begin{example}\label{exam} 1. Let $E$ be an elliptic curve over $\mbb C$ and $\mathcal{H}=L^2(E)$. Let $\mcA=\mbb C[\partial,\overline{\partial}]$ with involution $\partial^\dagger=-\overline{\partial}$. Take $V$ to be the space of smooth functions on $E$ which vanish in some neighborhood of $0\in E$. Then $S(\mcA)$ is the space of smooth functions $f$ on $E\setminus 0$ for which $\partial^m\overline{\partial}^nf\in L^2(E\setminus 0)$ for all $m,n$. 

We claim that $S(\mcA)=C^\infty(E)$, i.e. such $f$ is in fact smooth on the whole $E$. By elliptic regularity it suffices to show that $\overline\partial^n f\in L^2(E)$ as a distribution for all $n$. We prove it by induction in $n$. The base $n=0$ is known, so let us do the induction step. Let $\overline\partial^{n-1}f=g\in L^2(E)$ as a distribution. Thus as distributions on $E$ we have 
$\overline\partial g=h+\xi$, where $h\in L^2(E)$ and $\xi$ is concentrated at $0$. Locally near $0$ 
we can find $g_*\in H^1_{\rm loc}$ such that $\overline\partial g_*=h$ as distributions. Then $u:=g-g_*\in L^2_{\rm loc}$ 
and $\overline\partial u=\xi$. Thus $u(z)$ is holomorphic for $z\ne 0$ and is in $L^2_{\rm loc}$. By a version of the removable singularity theorem, this can only happen if $u$ is actually holomorphic at $0$ as well. Thus $\xi=0$ and the inductive step is complete. 

Thus, $\mcA_{\mbb R}$ is essentially self-adjoint on $V$. 

2. In contrast, let $\mathcal{H}=L^2(S^1)$ and $\mcA=\mbb C[i\frac{d}{dx}]$, where $\dagger$ is complex conjugation (with $i\frac{d}{dx}$ symmetric), and let $V$ be the space of smooth functions on the circle which vanish in some neighborhood of zero. 
Then $\mcA_{\mbb R}$ is not essentially self-adjoint by Proposition \ref{esa1} since the operator $A:=i\frac{d}{dx}$ is not essentially self-adjoint on $V$. Indeed, the function $f=e^x\in V^\vee$ satisfies the equation $Af=if$, so the claim follows from Theorem \ref{vn}. 

3. Similarly, let $\mathcal{H},V$ be as in (1) but $\mcA=\mbb C[\Delta]$, where $\Delta=4\partial\overline{\partial}$ is the Laplace operator. Let $z=x+\tau y$ and consider the function 
$$
h(z)=\sum_{m,n\in \mbb Z}\frac{e^{i(mx+ny)}}{|m+n\tau|^2+i}.
$$ 
Then $h\in L^2$ and $\Delta h-ih$ is a multiple of the $\delta$-function at $0$. 
Hence again by Theorem \ref{vn} $\Delta$ (and hence by Proposition \ref{esa1} $\mcA_{\mbb R}$) is not essentially self-adjoint on $V$.  

Thus we see that in (1) the operator $A:=\left(\begin{matrix} 0&\partial\\ -\overline{\partial}& 0\end{matrix}\right)$ 
is essentially self-adjoint on $V\otimes \mbb C^2$ while $A^2=-\frac{1}{4}\Delta$ is not. 
\end{example} 

\section{The main theorem} 

The main result of Part II of this paper is the following theorem, which, combined with Subsections \ref{realmonod} and \ref{gloeig}, confirms Conjectures \ref{basic}, \ref{first}, \ref{second}, \ref{third} in the case when $X=\mbb P^1$ with four marked points. 

\begin{theorem}\label{mainthe} (i) The normalized joint eigenfunctions $\psi_\Lambda$ of $L,L^\dagger$ (one for each $\Lambda\in \Sigma$) form an orthonormal basis of $\mathcal{H}$. 

(ii) Let $V$ be the space of smooth sections of $O(s)\otimes \overline{O(-2-s)}$ which vanish near the four singular points. 
Then the algebra $\mbb C[L,L^\dagger]_{\mbb R}$ is essentially self-adjoint on $V$, and its spectrum is naturally identified with $\Sigma$. 
\end{theorem} 

Part (i) of the theorem follows from part (ii), so the rest of the paper is mainly dedicated to the proof of part (ii). 
In the course of the proof, we define and study the Sobolev and Schwartz spaces associated to $L$, which is 
interesting in its own right.  

\section{Generalized Sobolev and Schwartz spaces attached to the operator $L$.} 

\subsection{Sobolev spaces}

Recall that for $r\in \mbb Z_+$ the {\it Sobolev space} $H^r_{\rm loc}=H^r_{\rm loc}(U)$ on a real manifold $U$ is the space of 
distributions on $U$ whose $r$-th distributional derivatives belong to $L^2_{\rm loc}(U)$. 
For basics on Sobolev spaces we refer the reader to the book \cite{A}. In particular, if $\dim U=2$ and $r\ge 2$, then we have the {\it Sobolev embedding} $H^r_{\rm loc}(U)\subset C^{r-2,\gamma}_{\rm loc}(U)$ for any $0<\gamma<1$, where $C^{m,\gamma}_{\rm loc}(U)$ is the space of $C^m$-functions whose $m$-th derivatives are locally $\gamma$-H\"older continuous (see \cite{A},p.97-98). Another Sobolev embedding in the 2-dimensional case is $H^1_{\rm loc}(U)\subset L^p_{\rm loc}(U)$ for all $1\le p<\infty$ (\cite{A}, p.97, formula (6)). 

We will often use {\it elliptic regularity}: if $E$ is an elliptic differential operator of order $m$ 
and $f$ is a distribution such that $Ef\in H^k_{\rm loc}$ then $f\in H^{k+m}_{\rm loc}$, and also if 
$Ef\in C^{k,\gamma}_{\rm loc}$ for some $0<\gamma<1$ then $f\in C^{k+m,\gamma}_{\rm loc}$ (Schauder estimates).  

\subsection{Auxiliary lemmas} 

\begin{lemma}\label{aux1} Let $0<\gamma<\beta<2$ and 
$$
G_{\gamma,\beta}(z):=\int_{|w|<1}\frac{|w|^{\gamma-2}}{|z-w|^\beta} d^2w.
$$
Then $G_{\gamma,\beta}(z)=O(|z|^{\gamma-\beta})$ as $z\to 0$. 
\end{lemma} 

\begin{proof} 
Let $r=|z|$. We have 
$$
G_{\gamma,\beta}(z)= \int_{|w|\le 2r}\frac{|w|^{\gamma-2}}{ |z-w|^\beta} d^2w+
\int_{2r\le |w|<1}\frac{|w|^{\gamma-2}}{|z-w|^\beta} d^2w.
$$
If $|u|\ge 2$ then 
\begin{equation}\label{eq3a} 
|1-u^{-1}|\ge \frac{1}{2}. 
\end{equation}
Thus, making the change of variable $u=w/z$ in the first integral and using \eqref{eq3a} in the second one, we get 
$$
|G_{\gamma,\beta}(z)|\le r^{\gamma-\beta}\int_{|u|\le 2}\frac{|u|^{\gamma-2}}{ |1-u|^\beta} d^2u+
2^{\beta}\int_{2r\le |w|<1}|w|^{\gamma-\beta-2}d^2w. 
$$
But, passing to polar coordinates, we have for any $\delta>0$: 
\begin{equation}\label{eq25}
\int_{2r\le |w|<1}|w|^{-\delta-2}d^2w=2\pi \frac{(2r)^{-\delta}-1}{\delta}=O(r^{-\delta}),\ r\to 0. 
\end{equation}
This implies the statement, by setting $\delta=\beta-\gamma$. 
\end{proof} 

\begin{lemma}\label{normform} 
Let $b,c,d$ be holomorphic functions on a disk centered at $0\in \mbb C$ with $b(0)=0$, $b'(0)=1$, $b(z)\ne 0$ for $z\ne 0$, and let 
$$
L=L(b,c,d):=\partial b(z)\partial-c(z)\partial+d(z). 
$$
Let $\alpha:=c(0)$, and suppose $\alpha\notin \mbb Z\setminus 0$. 
Then $L$ can be uniquely brought to the form $\partial z\partial-\alpha\partial$ by 
a holomorphic change of coordinate and multiplication of $L$ on the left and right by a holomorphic invertible function equal to $1$ at $0$. 
\end{lemma} 

\begin{proof} Consider the equation $L\psi=0$. Using the power series method, we see that it has a unique solution $\psi_1$ which is holomorphic near $0$ with $\psi_1(0)=1$. Multiplying $L$ on the right by $\psi_1^{-1}$, we can get to a situation where $\psi_1=1$. 

The second solution $\psi_2$ then looks like $\psi_2=\frac{(ze^{\eta(z)})^\alpha-1}{\alpha}$, where 
$\eta$ is holomorphic near $0$ with $\eta(0)=0$ and $\frac{x^\alpha-1}{\alpha}:=\log x$ if $\alpha=0$. 
Taking $\tilde z:=ze^{\eta(z)}$ as a new coordinate, we get to a situation when $\psi_2=\frac{z^\alpha-1}{\alpha}$. Then $L=a(z)(\partial z\partial-\alpha \partial)$ where $a(0)=1$, so 
we can multiply by $a^{-1}$ on the left to get to the situation where $L=\partial z\partial-\alpha \partial$. 
This shows that $L$ can be brought to the required normal form. The proof of uniqueness is straightforward. 
\end{proof} 

\begin{lemma}\label{diffeq} (i) Let $f$ be a distribution on a neighborhood of $0\in \mbb C$ and $\alpha\in i\mbb R$. 
Then $(\partial z\partial-\alpha \partial)f\in L^2_{\rm loc}$ if and only if 
$$
f(z)=f_*(z)+|z|^{2\alpha}Q(\overline{z}^{-1}),
$$
where
$$
f_*(z):=|z|^{2\alpha}(\phi(z)+h(\overline{z}))
$$
and
\begin{equation}\label{eq20}
\phi(z)=\frac{1}{\pi}\int_{|w|<\varepsilon}\frac{|w|^{-2\alpha}g(w)}{w(\overline{w}-\overline{z})}d^2w
\end{equation} 
with $g\in H^1_{\rm loc}$, $Q$ is a polynomial with $Q(0)=0$, and $h$ is holomorphic near $0$
(here $|z|^{2\alpha}Q(\overline{z}^{-1})$ is viewed as a distribution in the sense of principal value).  
Moreover, in this case $f_*(z)=O(|z|^{-\delta})$ as $z\to 0$ for any $\delta>0$.

(ii) Such $f$ is in $L^2_{\rm loc}$ if and only if $Q=0$, i.e., 
$$
f(z)=|z|^{2\alpha}(\phi(z)+h(\overline{z})). 
$$
In this case $f(z)=O(|z|^{-\delta})$ as $z\to 0$ for any $\delta>0$.  

(iii) If $(\partial z\partial-\alpha\partial)f\in L^2_{\rm loc}$ then $f$ can be uniquely written as 
$$
f(z)=f_*(z)+|z|^{2\alpha}Q(\overline{z}^{-1}),
$$
where $f_*\in L^2_{\rm loc}$ and $Q$ is a polynomial with $Q(0)=0$. 
\end{lemma}

\begin{proof} (i) We have $(\partial z\partial-\alpha\partial)f\in L^2_{\rm loc}$ iff $(z\partial-\alpha)f=g$, for some $g\in H^1_{\rm loc}$. 
Since the fundamental solution of the equation $\partial F=0$ is $\frac{1}{\pi \overline{z}}$, we also have 
$$
(z\partial-\alpha)(|z|^{2\alpha}\phi)=|z|^{2\alpha}z\partial \phi=g,
$$
for $\phi$ defined by \eqref{eq20}.
Thus, $(\partial z\partial-\alpha\partial)f\in L^2_{\rm loc}$ iff
$$
(z\partial-\alpha) (f-|z|^{2\alpha}\phi)=0. 
$$
Outside $z=0$, this yields 
$$
f-|z|^{2\alpha}\phi=|z|^{2\alpha}\widetilde{h}(\overline z), 
$$
where $\widetilde{h}$ is a meromorphic function that is holomorphic for $z\ne 0$. 

We can view $|z|^{2\alpha}\widetilde{h}(\overline z)$ as a distribution in the sense of principal value. Then we get that $f-|z|^{2\alpha}\phi-|z|^{2\alpha}\widetilde{h}(\overline z)$ is a distribution 
concentrated at $0$ and annihilated by the operator $z\partial-\alpha$. This implies that 
$$
f-|z|^{2\alpha}\phi-|z|^{2\alpha}\widetilde{h}(\overline z)=0,
$$ 
as the operator $z\partial-\alpha$ on distributions concentrated at $0$ 
is diagonalizable with eigenvalues $-n-\alpha$ for $n\in \mbb Z_{>0}$, 
hence invertible. Now writing $\widetilde{h}=h+Q$, where 
$Q$ is the principal part and $h$ is holomorphic, we get the result.  

Conversely, it is easy to see that 
for any $f$ of the form given in (i) we have 
$$(\partial z\partial-\alpha\partial)f=\partial g\in L^2_{\rm loc}.$$ 

It remains to prove the bound for $f_*$. By the H\"older inequality
 $$
 |\phi(z)|\le \frac{1}{\pi}||g||_p||w^{-1}(w-z)^{-1}||_q,
 $$
 where $\frac{1}{p}+\frac{1}{q}=1$ and $q<2$. By the Sobolev embedding theorem, 
 $g\in L^p_{\rm loc}$, so the first factor is finite. Also, by Lemma \ref{aux1}
 applied to the special case $\gamma=2-q,\beta=q$, the second factor is $O(|z|^{\frac{2}{q}-2})$ as $z\to 0$. 
 So by taking $q=\frac{2}{2-\delta}$, we get that $\phi(z)=O(|z|^{-\delta})$ 
 for every $\delta>0$. This proves (i). 
 
(ii), (iii) follow immediately from (i). 
\end{proof} 

\begin{lemma}\label{lemmm} Let $f\in L^2_{\rm loc}$ be defined on a neighborhood of $0\in \mbb C$ and $\alpha\in i\mbb R$. Then $|z|^{2+2\alpha}\Delta (|z|^{-2\alpha}f)$ is in $H^2_{\rm loc}$ 
as a distribution if and only if 
\begin{equation}\label{eq30} 
f(z)=C_2\frac{2\alpha|z|^{2\alpha}\log |z|-|z|^{2\alpha}+1}{2\alpha^2}+C_1\frac{|z|^{2\alpha}-1}{2\alpha}+f_0(z), 
\end{equation}
where $C_1,C_2\in \mbb C$ and 
$$
f_0(z)=|z|^{2\alpha}(\eta(z)+H(z))
$$ 
 with $H$ harmonic and 
$$
\eta(z)=\frac{1}{2\pi}\int_{|w|<\varepsilon}\frac{g(w)}{|w|^{2+2\alpha}}\log|z-w|d^2w,
$$
where $g\in H^2_{\rm loc}$ with $g(0)=0$. In this case, $f_0$ is $\gamma$-H\"older continuous and 
$$
\nabla f_0(z)=O(|z|^{\gamma-1})
$$ 
as $z\to 0$ for any $\gamma<1$. 
\end{lemma} 

\begin{remark} Note that 
$$
\lim_{\alpha\to 0}\frac{2\alpha|z|^{2\alpha}\log |z|-|z|^{2\alpha}+1}{2\alpha^2}=\log^2|z|. 
$$ 
So we agree that this expression stands for $\log^2|z|$ when $\alpha=0$, and similarly 
$\frac{|z|^{2\alpha}-1}{2\alpha}$ stands for $\log|z|$ in this case. 
Thus for $\alpha=0$ equation \eqref{eq30} takes the form 
$$
f(z)=C_2\log^2|z|+C_1\log|z|+f_0(z). 
$$
\end{remark} 

\begin{proof} Let $r=|z|$ and $r^{2+2\alpha}\Delta (r^{-2\alpha}f)=g$, where $g$ is in $H^2_{\rm loc}$. By the Sobolev embedding theorem, 
$g$ is continuous and moreover $\gamma$-H\"older continuous for any $\gamma<1$. 
We have 
$$
r^{2+2\alpha}\Delta\left(r^{-2\alpha} \frac{2\alpha|z|^{2\alpha}\log |z|-|z|^{2\alpha}+1}{2\alpha^2}\right)=2
$$ 
(as distributions). Thus, replacing $f(z)$ by $f(z)-\frac{1}{2}g(0)\frac{2\alpha|z|^{2\alpha}\log |z|-|z|^{2\alpha}+1}{2\alpha^2}$, we may assume without loss of generality that $g(0)=0$. In this case we will show that $f$ has the required decomposition with $C_2=C_1\alpha$, which implies the statement.  

Suppose the equation $r^{2+2\alpha}\Delta(r^{-2\alpha}f)=g$ holds in a disk $|z|<\varepsilon$. Write 
$$
f(z)=f_*(z)+h(z),
$$
where
\begin{equation}\label{eq6}
f_*(z):=\frac{|z|^{-2\alpha}}{2\pi}\int_{|w|<\varepsilon}\frac{g(w)}{|w|^{2+2\alpha}}\log|z-w|d^2w. 
\end{equation}
Note that the integral makes sense since $g(0)=0$. Since $\frac{1}{2\pi}\log|z-w|$ is the fundamental solution of Laplace's equation, we have 
$$
r^{2+2\alpha}\Delta (r^{-2\alpha}f_*)=g
$$
as distributions. Hence $r^{2+2\alpha}\Delta (r^{-2\alpha}h)=0$ as distributions. This implies that 
$$
\Delta (r^{-2\alpha}h)=2\pi R(\partial,\overline{\partial})\delta_0,
$$
where $R$ is a harmonic polynomial and $\delta_0$ is the delta-function at zero. 
Thus 
$$
h(z)=|z|^{2\alpha}(R(\partial,\overline{\partial})\log |z|+H(z)),
$$
where $H(z)$ is a harmonic (hence smooth, in fact real analytic) function. Since $h\in L^2_{\rm loc}$, the polynomial $R$ must be constant, so we get that 
$$
h(z)=|z|^{2\alpha}(C_1\log |z|+H(z)),\ C_1\in \mbb C
$$
 is a linear combination of $\log |z|$ and a harmonic function. 
 This establishes \eqref{eq30}. 

Thus, it remains to show that $f_*(z)$ is $\gamma$-H\"older continuous and 
\begin{equation}\label{eq5}
\nabla f_*(z)=O(|z|^{\gamma-1}),\ z\to 0
\end{equation}
for any $\gamma<1$. To prove \eqref{eq5}, observe that 
for any $\gamma<1$ we have 
$$
|g(z)|\le C(\gamma) |z|^\gamma
$$
 for small enough $|z|$, so we may assume that this holds in our disk $|z|<\varepsilon$. Thus, differentiating \eqref{eq6}, we get  
$$
\max(|\partial_x f_*(z)|,|\partial_y f_*(z)|)\le \frac{1}{2\pi}\int_{|w|<\varepsilon}\frac{|g(w)|}{|w|^2}\frac{1}{|z-w|}d^2w\le 
$$
$$
\le \frac{C(\gamma)}{2\pi}\int_{|w|<\varepsilon}\frac{|w|^{\gamma-2}}{|z-w|}d^2w, 
$$
so \eqref{eq5} follows from Lemma \ref{aux1} for $\beta=1$. 

Finally, let us prove that $f_*$ is $\gamma$-H\"older continuous for any $\gamma<1$. Let $|z_1|\le |z_2|$ and $z_2\ne 0$. Using \eqref{eq5} and applying the mean value theorem to the restriction of $f_*$ to the segment $z(t)=z_1t+z_2(1-t)$, we have 
$$
|f_*(z_1)-f_*(z_2)|\le K(\gamma)|z_1-z_2|\int_0^1 |z_1t+z_2(1-t)|^{\gamma-1}dt=
$$
$$
=K(\gamma)|z_2|^{\gamma-1}|z_1-z_2|\int_0^1 |z_1z_2^{-1}t+1-t|^{\gamma-1}dt.
$$
It is easy to see that $\int_0^1 |bt+1-t|^{\gamma-1}dt$ is bounded by some universal constant
when $|b|\le 1$. Thus, 
$$
|f_*(z_1)-f_*(z_2)|\le M(\gamma)|z_2|^{\gamma-1}|z_1-z_2|
$$
for some constant $M(\gamma)$. But $|z_1-z_2|\le |z_1|+|z_2|\le 2|z_2|$. 
so this implies that 
$$
|f_*(z_1)-f_*(z_2)|\le 2^{1-\gamma}M(\gamma)|z_1-z_2|^\gamma,
$$
as claimed. 

The lemma is proved. 
\end{proof} 

\subsection{The generalized local Sobolev and Schwartz spaces} 

Let $L=L(b,c,d)$ be as in Lemma \ref{normform}, with $c(0)=\alpha \in i\mbb R$. 

\begin{definition} We will say that a function $f$ defined near $0\in \mbb C$ 
is in $H^{2m}_{L,\rm loc}$ if $L^jf\in L^2_{\rm loc}$ as a distribution for all $0\le j\le m$. 
\end{definition} 

More precisely, for every neighborhood $U$ of $0$ in $\mbb C$, we have the space $H^{2m}_{L,\rm loc}(U)$; 
however, a specific region $U$ will not matter to us, so we will omit it from the notation; 
in other words, we are essentially talking about germs of functions near $0$. 

It is clear that $H^{2m}_{L,\rm loc}\subset H^{2k}_{L,\rm loc}$ if $m\le k$ and 
$L,L^\dagger: H^{2m}_{L,\rm loc}\to H^{2m-2}_{L,\rm loc}$. 

\begin{proposition}\label{chara} (i) $H^2_{L,\rm loc}=H^2_{\partial z\partial-\alpha \partial,\rm loc}$. 

(ii) $f\in H^2_{L,\rm loc}$ if and only if $f\in L^2_{\rm loc}$ and $L^\dagger f\in L^2_{\rm loc}$. 

(iii) $H^{2m}_{L,\rm loc}=H^{2m}_{\partial z\partial-\alpha \partial,\rm loc}$ for all $m$. 
Moreover, $f\in H^{2m}_{L,\rm loc}$ iff $(L^\dagger)^jf\in L^2_{\rm loc}$ for $0\le j\le m$.

(iv) $f\in H^{2m}_{L,\rm loc}$ if and only if for any polynomial $P$ of two variables of degree $\le m$ we have $P(L,L^\dagger)f\in L^2_{\rm loc}$. 

(v) If $\partial f$ or $\overline{\partial} f\in H^{2m}_{L,\rm loc}$ 
then $f\in H^{2m}_{L,\rm loc}$. 

(vi) $H^{2m}_{L,\rm loc}$ is invariant under holomorphic changes of variable and multiplication by $C^{2m}$-functions. 

(vii) For $m\ge 1$ we have $f\in H^{2m}_{L,\rm loc}$ iff $f\in L^2_{\rm loc}$ and $Lf\in H^{2m-2}_{L,\rm loc}$ iff $f\in L^2_{\rm loc}$ and $L^\dagger f\in H^{2m-2}_{L,\rm loc}$. 

(viii) Let $A_0,...,A_m$ be a sequence of differential operators such that $A_0=1$ and for each $0\le j\le m-1$ one has $A_{j+1}=LA_j$ or $A_{j+1}=L^\dagger A_j$. Then $H^{2m}_{L,\rm loc}$ is the space of $f$ such that $A_jf\in L^2_{\rm loc}$ for all $0\le j\le m-1$. 
\end{proposition} 

\begin{proof} (i) Let $f\in L^2_{\rm loc}$. Then $f\in H^2_{L,\rm loc}$ iff 
$f\in H^2_{\partial (b\partial-c) ,\rm loc}$ iff $(b\partial-c)f\in H^1_{\rm loc}$ iff 
$(z\partial-q)f\in H^1_{\rm loc}$, where $q:=zb^{-1}c$. 

We claim that in this case $zf\in H^1_{\rm loc}$; this implies that for any smooth function $h$ near $0$ 
we have $zhf\in H^1_{\rm loc}$. Indeed, 
$$
\partial(zf)=(z\partial-q)f+(1+q)f,
$$ 
and the first summand is in $H^1_{\rm loc}$ while the second one in $L^2_{\rm loc}$. Thus $\partial(zf)\in L^2_{\rm loc}$, hence 
$zf\in H^1_{\rm loc}$, as claimed. 

Thus, $(z\partial-q)f\in H^1_{\rm loc}$ iff $(z\partial-\alpha)f\in H^1_{\rm loc}$ (taking $h=z^{-1}(q-\alpha)$ and using that $q(0)=c(0)=\alpha$). So altogether we get that $f\in H^2_{L,\rm loc}$ iff 
$(z\partial-\alpha)f\in H^1_{\rm loc}$, as desired. 

(ii) By (i) it suffices to consider the case $L=\partial z\partial-\alpha\partial$. Then for $f\in L^2_{\rm loc}$ we have $\partial (z\partial-\alpha)f\in L^2_{\rm loc}$ iff $\overline{\partial} (z\partial-\alpha)f\in L^2_{\rm loc}$ iff 
$ (z\partial-\alpha)\overline{\partial}f\in L^2_{\rm loc}$ iff $z\partial |z|^{-2\alpha}\overline{\partial}f\in L^2_{\rm loc}$ iff $\overline{z}\partial |z|^{-2\alpha}\overline{\partial}f\in L^2_{\rm loc}$ iff $\partial\overline{z} |z|^{-2\alpha}\overline{\partial}f\in L^2_{\rm loc}$ iff $\overline{\partial}\overline{z} |z|^{-2\alpha}\overline{\partial}f\in L^2_{\rm loc}$ iff $(\overline{\partial}\overline{z}-\alpha)\overline{\partial}f \in L^2_{\rm loc}$, as claimed. 

(iii) We have $f\in H^{2m}_{L,\rm loc}$ iff $L^jf\in L^2_{\rm loc}$, $0\le j\le m$, which by (i),(ii) is equivalent to 
$f\in L^2_{\rm loc}$ and $(\overline{\partial}\overline{z}\overline{\partial}-\alpha\overline{\partial})L^{j-1}f\in L^2_{\rm loc}$, $1\le j\le m$, hence to $f\in L^2_{\rm loc}$ and $L^{j-1}(\overline{\partial}\overline{z}\overline{\partial}-\alpha\overline{\partial})f\in L^2_{\rm loc}$, $1\le j\le m$. Repeating this $m$ times, we see that 
$f\in H^{2m}_{L,\rm loc}$ is equivalent to $(\overline{\partial}\overline{z}\overline{\partial}-\alpha\overline{\partial})^jf\in L^2_{\rm loc}$, $0\le j\le m$. Since in particular this holds for 
$L=\partial z\partial-\alpha\partial$, the claim follows. 

(iv) It suffices to show that if $f\in H^{2m}_{L,\rm loc}$ then 
$(L^\dagger)^kL^jf\in L^2_{\rm loc}$ for any $j,k$ with $j+k\le m$, $k>0$. 
Using (ii) we see that this happens iff $(L^\dagger)^{k-1}L^{j+1}f\in L^2_{\rm loc}$. 
The claim now follows by repeating this procedure $k$ times.

(v) It suffices to prove the first statement, the second one is proved similarly. 
By (iii) we have $\partial (L^\dagger)^jf=(L^\dagger)^j\partial f\in L^2_{\rm loc}$ for all $0\le j\le m$. 
This means that $(L^\dagger)^jf\in H^1_{\rm loc}\subset L^2_{\rm loc}$ for $0\le j\le m$, i.e., 
$f\in H^{2m}_{L,\rm loc}$, as claimed. 

(vi) The first statement follows from (iii). To prove the second statement, we argue by induction in $m$ with the obvious base $m=0$. We may assume $L=\partial z\partial-\alpha\partial$. Assume the statement is known for $m-1$ and let us prove it for $m$. Let $f\in H^{2m}_{L,\rm loc}$ and $\sigma\in C^{2m}_{\rm loc}$. 
It suffices to show that $L(\sigma f)\in H^{2m-2}_{L,\rm loc}$. 
We have 
$$
L(\sigma f)=(L\sigma)f+\sigma(Lf)+2z\partial\sigma\cdot \partial f. 
$$
By the induction assumption, the first summand is in $H^{2m-2}_{L,\rm loc}$, as 
$L\sigma\in C^{2m-2}$, and so is the second summand, 
as $Lf\in H^{2m-2}_{L,\rm loc}$. So it suffices to show that 
$z\partial f\in H^{2m-2}_{L,\rm loc}$ or, equivalently, $z\partial f-\alpha f\in H^{2m-2}_{L,\rm loc}$. But we have $\partial(z\partial f-\alpha f)=Lf\in H^{2m-2}_{L,\rm loc}$, so the statement follows from (v). 

(vii) By definition $f\in H^{2m}_{L,\rm loc}$ iff $f\in L^2_{\rm loc}$ and $Lf,...,L^{m-1}(Lf)\in L^2_{\rm loc}$. 
The latter condition is equivalent to saying that $Lf\in H^{2m-2}_{L,\rm loc}$, which proves the first equivalence. 
The second equivalence follows in the same way, replacing $L$ with $L^\dagger$ and using  (iii). 

(viii) The proof is by induction in $m$. The case $m=0$ is clear, so let us justify the induction step from $m-1$ to $m$. 
Denote the space in question by $Y$. By the induction assumption, $f\in Y$ iff 
$f\in L^2_{\rm loc}$ and $Lf\in H^{2m-2}_{L,\rm loc}$ or $f\in Y$ iff $f\in L^2_{\rm loc}$ and $L^\dagger f\in H^{2m-2}_{L,\rm loc}$.
But by (vii) either of these conditions is equivalent to $f\in H^{2m}_{L,\rm loc}$, as desired. 
\end{proof} 

Thus we see that $H^{2m}_{L,\rm loc}$ does not really depend on $L$ once $\alpha$ is fixed. So we may denote $H^{2m}_{L,\rm loc}$ by $H^{2m}_{\alpha,\rm loc}$. This motivates the following definition. 

\begin{definition} (i) The {\it generalized local Sobolev space} $H^{2m}_{\alpha,\rm loc}$ is the space of $f\in L^2_{\rm loc}$ such that $(\partial z\partial-\alpha \partial)^jf\in L^2_{\rm loc}$, $0\le j\le m$. 

(ii) The {\it generalized local Schwartz space} $S_{\alpha,\rm loc}:=H_{\alpha,\rm loc}^\infty=
\cap_m H_{\alpha,\rm loc}^{2m}$ is the space of functions $f$ such that $P(L,L^\dagger)f$ 
is in $L^2_{\rm loc}$ for all polynomials $P$, or equivalently, the space of $f$ such that 
$L^jf$ is in $L^2_{\rm loc}$ for all $j\in \mbb Z_+$. 
\end{definition} 

\begin{remark} 1. The space $H^{2m}_{\alpha,\rm loc}$ is a generalization (or, rather, an analog) of the local Sobolev space $H^{2m}_{\rm loc}$, which motivates the terminology. 

2. We see that the space $H^{2m}_{\alpha,\rm loc}$ depends only on the conformal structure around $0$ and not on a specific choice of the local holomorphic coordinate. 
\end{remark} 

\subsection{Explicit description of $H^{2}_{\alpha,\rm loc}$}
Let us now describe the generalized local Sobolev space $H^{2m}_{\alpha,\rm loc}$ in the simplest nontrivial case $m=1$ more explicitly in terms of ordinary Sobolev spaces. 

\begin{proposition}\label{sobo} (i) $f\in H_{\alpha,\rm loc}^2$ if and only if it has the form 
$$
f(z)=|z|^{2\alpha}(\phi(z)+h(\overline{z})),
$$
where $\phi$ is given by \eqref{eq20} with $g\in H^1_{\rm loc}$, and $h$ is holomorphic. 
In this case $f(z)=O(|z|^{-\delta})$ as $z\to 0$ for any $\delta>0$. 

(ii) If $f$ is a distribution in a neighborhood of $0$ then $Lf\in L^2_{\rm loc}$ if and only if
$f$ can be written as 
$$
f(z)=f_*(z)+Q(\overline{z}^{-1}),
$$
where $f_*\in H^2_{\alpha,\rm loc}$ and $Q$ is a polynomial with $Q(0)=0$. Moreover, this decomposition is unique. 
\end{proposition} 

\begin{proof} (i) follows immediately from Lemma \ref{diffeq}, (i),(ii). 
To prove (ii), note that by Lemma \ref{normform}, 
by changes of variable and multiplication by a non-vanishing holomorphic function 
we can turn $f$ into a function $\widetilde{f}$ such that $(\partial z\partial-\alpha\partial)\widetilde{f}\in L^2_{\rm loc}$. 
Thus by Proposition \ref{chara}(vi) 
it suffices to prove the statement for $L=\partial z\partial-\alpha \partial$. 
But this special case is Lemma \ref{diffeq}(iii). 
\end{proof} 

\begin{corollary}\label{surj} (i) For $m\ge 0$, the map $L: H^{2(m+1)}_{\alpha,\rm loc}\to H^{2m}_{\alpha,\rm loc}$ is surjective, i.e. for any $f\in H^{2m}_{\alpha,\rm loc}$ there is $u\in H^{2(m+1)}_{\alpha,\rm loc}$ (defined on some neighborhood of $0$) 
such that $Lu=f$. 

(ii) If $P$ is any polynomial of degree $r$ in one variable and $m\ge 0$ then the operator $P(L): H^{2(m+r)}_{\alpha,\rm loc}\to H^{2m}_{\alpha,\rm loc}$ is surjective (in the same sense as in (i)). 
\end{corollary} 

\begin{proof} (i) It suffices to consider the case $L=\partial z\partial-\alpha\partial$. 
By Proposition \ref{sobo}(i), the statement holds for $m=0$. Thus if $m$ is arbitrary and 
$f\in H^{2m}_{\alpha,\rm loc}$ then there is a $u\in L^2_{\rm loc}$ with $Lu=f$. Moreover, 
$L^jf\in L^2_{\rm loc}$ for $0\le j\le m$, so $L^ju\in L^2_{\rm loc}$ for $0\le j\le m+1$, as claimed. 

(ii) This follows from (i) by factoring $P$ into linear factors. 
\end{proof} 

\subsection{Explicit description of $H^4_{\alpha,\rm loc}$}

Let $L=\partial z\partial-\alpha\partial$ and 
$D=L^\dagger L$. Let $\widetilde{H}_{\alpha,\rm loc}^4$ be the space of $f$ such that 
$f$ and $Df$ are both in $L^2_{\rm loc}$. By Proposition \ref{chara}(viii), 
we have $H_{\alpha,\rm loc}^4\subset \widetilde{H}_{\alpha,\rm loc}^4$; namely, this subspace is cut out by the condition that $Lf\in L^2_{\rm loc}$. 

\begin{proposition}\label{charfun2} (i) $f\in \widetilde{H}_{\alpha,\rm loc}^4$ if and only if 
it has the form 
\begin{equation}\label{eq40}
f(z)=C_2\frac{2\alpha|z|^{2\alpha}\log |z|-|z|^{2\alpha}+1}{2\alpha^2}+C_{1}\frac{|z|^{2\alpha}-1}{2\alpha}+f_0(z), 
\end{equation} 
where 
$$
f_0(z)=|z|^{2\alpha}(\eta(z)+H(z))
$$ 
 with $H$ harmonic and 
$$
\eta(z)=\frac{1}{2\pi}\int_{|w|<\varepsilon}\frac{g(w)}{|w|^{2+2\alpha}}\log|z-w|d^2w,
$$
where $g\in H^2_{\rm loc}$ with $g(0)=0$. In this case, $f_0$ is $\gamma$-H\"older continuous and 
$$
\nabla f_0(z)=O(|z|^{\gamma-1})
$$ 
as $z\to 0$ for any $\gamma<1$. 

(ii) Such $f$ is in $H^4_{\alpha,\rm loc}$ iff $C_2=0$. Thus, $H^4_{\alpha,\rm loc}$ is the space of functions of the form 
$$
f(z)=C\frac{|z|^{2\alpha}-1}{2\alpha}+f_0(z),
$$
where $f_0$ is as in (i), and $\dim \widetilde{H}_{\alpha,\rm loc}^4/H_{\alpha,\rm loc}^4=1$.  
\end{proposition} 

\begin{remark} For $\alpha\ne 0$ we may use a simpler (but equivalent) formula in Proposition \ref{charfun2}(i): 
$$
f(z)=C_2'|z|^{2\alpha}\log |z|+C_1'|z|^{2\alpha}+f_0(z). 
$$
\end{remark} 

\begin{proof} (i) We have $16D=\Delta r^{2+2\alpha}\Delta r^{-2\alpha}$. Since $\Delta$ is elliptic, we get that $Df\in L^2_{\rm loc}$ if and only if $r^{2+2\alpha}\Delta (r^{-2\alpha}f)\in H^2_{\rm loc}$. Thus the desired result follows from Lemma \ref{lemmm}. 

(ii) It is easy to see that 
$$
L\frac{2\alpha|z|^{2\alpha}\log |z|-|z|^{2\alpha}+1}{2\alpha^2}=\frac{|z|^{2\alpha}}{z}\notin L^2_{\rm loc},
$$
while $L\frac{|z|^{2\alpha}-1}{2\alpha}=0$. Thus it suffices to show that $Lf_0\in L^2_{\rm loc}$, i.e., that $(z\partial-\alpha)f_0\in H^1_{\rm loc}$, which is equivalent to saying that 
$|z|^{2\alpha}z\partial\eta\in H^1_{\rm loc}$, i.e., $\overline{\partial} (|z|^{2\alpha}z\partial\eta)\in L^2_{\rm loc}$. 
We have 
$$
\overline{\partial} (|z|^{2\alpha}z\partial\eta)=\overline{\partial}\left(\frac{|z|^{2\alpha}z}{4\pi}\int_{|w|<\varepsilon}\frac{g(w)}{|w|^{2+2\alpha}(z-w)}d^2w\right)=
$$
$$
\frac{\alpha |z|^{2\alpha}z}{4\pi\overline{z}}
\int_{|w|<\varepsilon}\frac{g(w)}{|w|^{2+2\alpha}(z-w)}d^2w
+\frac{g(z)}{4\overline{z}}.
$$
But $|g(w)|\le C(\gamma)|w|^\gamma$ for any $\gamma<1$, so 
$$
|\overline{\partial}(|z|^{2\alpha}z\partial\eta)|\le \frac{1}{4}C(\gamma)\left(\frac{|\alpha|}{\pi} \int_{|w|<\varepsilon}\frac{|w|^{\gamma-2}}{|z-w|}d^2w+|z|^{\gamma-1}\right).
$$
Therefore, by Lemma \ref{aux1} for $\beta=1$, we have $\overline{\partial}(|z|^{2\alpha}z\partial\eta)=O(|z|^{\gamma-1})$, as $z\to 0$, i.e. it is in $L^2_{\rm loc}$, as claimed.  
\end{proof} 

\subsection{Explicit description of $S_{\alpha,\rm loc}$}

\begin{proposition}\label{asymex} (i) Let $f\in H^{2m}_{\alpha, loc}$ for $m\ge 1$. Then 
$f$ can be written in the form 
$$
f(z)=Q_1(z,\overline{z})\frac{|z|^{2\alpha}-1}{2\alpha}+Q_2(z,\overline{z})+f_m(z),
$$
where $Q_1,Q_2$ are uniquely determined polynomials of $z,\overline{z}$ of degree $m-2$, and $f_m(z)=O(|z|^{\gamma+m-2})$ as $z\to 0$ for any $\gamma<1$, with $f_m(z)\in C^{m-2,\gamma}_{\rm loc}$ if $m\ge 2$. Thus, there exists a unique polynomial $Q_1$ of degree $m-2$ such that 
$f(z)-Q_1(z,\overline{z})\frac{|z|^{2\alpha}-1}{2\alpha}$ is in $C^{m-2,\gamma}_{\rm loc}$ for any $\gamma<1$.  

(ii) The generalized local Schwartz space $S_{\alpha,\rm loc}$ is  the space of functions of the form $Q_1(z,\overline{z})\frac{|z|^{2\alpha}-1}{2\alpha}+Q_2(z,\overline{z})$, where 
$Q_1,Q_2$ are smooth. 
\end{proposition} 

\begin{proof} (i) It is clear that the polynomials $Q_1,Q_2$ are unique if exist, so we just need to prove existence. The proof is by induction in $m$. The case $m=1$ just says that $f(z)=O(|z|^{\gamma-1})$ for any $\gamma<1$ and is Proposition \ref{sobo}(i). The case $m=2$ is Proposition \ref{charfun2}(ii). Suppose the statement is known for some $m\ge 2$ and let us prove it for $m+1$. By the induction assumption we have 
$$
(\partial z\partial-\alpha\partial)f=Q_1\frac{|z|^{2\alpha}-1}{2\alpha}+Q_2+G(z),
$$
where $Q_1,Q_2$ polynomials of degree $m-2$ and $G(z)=O(|z|^{\gamma+m-2})$ for any $\gamma<1$, with $G\in C^{m-2,\gamma}_{\rm loc}$. Fix polynomials $P_1,P_2$ of degree $m-1$ such that 
$$
(\partial z\partial-\alpha\partial)\left(P_1\frac{|z|^{2\alpha}-1}{2\alpha}+P_2\right)=
Q_1\frac{|z|^{2\alpha}-1}{2\alpha}+Q_2;
$$
it is easy to show that they exist. Let 
$$
f_*=f-P_1\frac{|z|^{2\alpha}-1}{2\alpha}-P_2.
$$
Then we have 
$$
(\partial z\partial-\alpha\partial)f_*=G.
$$ 
Hence 
$$
(z\partial-\alpha)f_*=\widetilde{g},
$$ 
where $\partial \widetilde{g}=G$. By elliptic regularity (for H\"older continuous functions) we have 
$$
\widetilde{g}(z)=g(z)+R(\overline z),
$$ 
where $g\in C^{m-1,\gamma}_{\rm loc}$ is such that $\partial g=G$ and $g(z)=O(|z|^{\gamma+m-1})$ as $z\to 0$ for any $\gamma<1$, and $R$ is a polynomial in one variable of degree $m-1$. 
By replacing $P_1$ by $P_1-R$, we can make sure that $R=0$, so $\widetilde{g}=g$. 
By Lemma \ref{diffeq}(ii), 
$$
f_*(z)=|z|^{2\alpha}(\phi(z)+h(\overline{z}))
$$
where 
$$
\phi(z)=\frac{1}{4\pi}\int_{|w|<\varepsilon}\frac{|w|^{-2\alpha}g(w)}{w(\overline{w}-\overline{z})}d^2w 
$$
and $h$ is holomorphic. Using elliptic regularity again and direct differentiation,  
it is easy to see that $\phi\in C^{m-1,\gamma}_{\rm loc}$ 
and the $m-1$-th Taylor polynomial of $\phi$ depends only on $\overline{z}$. Therefore, 
by modifying $P_1$ and $P_2$ by polynomials of $\overline{z}$ 
we can make sure that $\phi(z)+h(\overline{z})=o(|z|^{m-1})$ as $z\to 0$. Moreover, by applying 
Lemma \ref{aux1} to $m-1$-th derivatives of this function, it follows that 
moreover $\phi(z)+h(\overline{z})=O(|z|^{\gamma+m-1})$ as $z\to 0$ for any $\gamma<1$. 
Hence $f_*\in C^{m-1,\gamma}_{\rm loc}$ and $f_*(z)=O(|z|^{\gamma+m-1})$ as $z\to 0$, as claimed. 

(ii) Let $\mbb S_{\alpha,\rm loc}$ be the space of functions defined in (ii). 
It is easy to check by direct differentiation 
that the operator $L$ preserves $\mbb S_{\alpha,\rm loc}$, 
hence $\mbb S_{\alpha,\rm loc}\subset S_{\alpha,\rm loc}$. 
Now let us prove the opposite inclusion. Let 
$f\in S_{\alpha,\rm loc}$. By (i), for every integer $m\ge 2$ 
there exists a unique polynomial $Q_{1,m}$ of degree $m-2$ 
such that $f(z)-Q_{1,m}(z,\overline{z})\frac{|z|^{2\alpha}-1}{2\alpha}$ is in  
$C^{m-2}_{\rm loc}$. Hence $Q_{1,m+1}-Q_{1,m}$ is homogeneous 
of degree $m+1$. Thus there exists a smooth function $Q_1(z,\overline{z})$ whose $m-2$-th 
Taylor polynomial at $0$ is $Q_{1,m}$, for all $m\ge 2$. Then 
$f(z)-Q_1(z,\overline{z})\frac{|z|^{2\alpha}-1}{2\alpha}$ 
is of class $C^{m-2}$ for all $m$, which implies that 
it is a certain smooth function $Q_2(z,\overline{z})$, i.e., $S_{\alpha,\rm loc}\subset \mbb S_{\alpha,\rm loc}$, as claimed. 
\end{proof} 

\subsection{Global generalized Sobolev and Schwartz spaces} 
Now let us define global versions of generalized Sobolev and Schwartz spaces corresponding to the Darboux operator $L=L(\mbf a,\tau)$. 

\begin{definition} The {\it generalized Sobolev space} 
$H^{2m}_{\mbf a}$ is the space of all $f\in \mathcal{H}$ such that 
$f\in H^{2m}_{\rm loc}$ outside the points $p=0,a,1,\infty$ and 
$f\in H^{2m}_{a_p,\rm loc}$ near each of these four points. 
\end{definition} 

By Proposition \ref{chara}, $H^{2m}_{\mbf a}$ may be characterized as the space of 
$f$ such that one has $P(L,L^\dagger)f\in \mathcal{H}$ for every polynomial $P$ of degree $\le m$ or, equivalently, 
$L^jf\in \mathcal{H}$ for $0\le j\le m$. It is clear that $H^{2m}_{\mbf a}\subset 
H^{2k}_{\mbf a}$ if $m\le k$ and $L,L^\dagger: H^{2m}_{\mbf a}\to H^{2m-2}_{\mbf a}$. 

\begin{definition} The {\it generalized Schwartz space} is the intersection 
$S_{\mbf a}:=H_{\mbf a}^\infty=\cap_{m\ge 0}H^{2m}_{\mbf a}$. Equivalently, 
it is the space of $f\in \mathcal{H}$ such that 
$f$ is smooth outside the points $p=0,a,1,\infty$ and 
$f\in S_{a_p,\rm loc}$ near each of these four points. 
\end{definition} 

In other words, $S_{\mbf a}=S(\wt\mcA)$, 
where $\wt\mcA:=\mbb C[L,L^\dagger]\subset \End\wt V$ and $\wt V$ is the space of smooth 
sections of $O(s)\otimes \overline{O(-s-2)}$. 

By Proposition \ref{chara}(iv), $S_{\mbf a}$ may be characterized as the space of 
$f$ such that $P(L,L^\dagger)f\in \mathcal{H}$ for every polynomial $P$ or, equivalently, 
$L^jf\in \mathcal{H}$ for all $j$. Also, by Proposition \ref{asymex}(ii), 
$S_{\mbf a}$ may be described as the space of $f$ which are smooth outside 
$0,a,1,\infty$ and have the form $Q_1\frac{|z|^{2\alpha}-1}{2\alpha}+Q_2$ 
for smooth $Q_1,Q_2$ at each of these four points. 

The following proposition will play an important role below. 

\begin{proposition}\label{intparts} For any $f_1,f_2\in H^2_{\mbf a}$ we have $(Lf_1,f_2)=(f_1,L^\dagger f_2)$. 
\end{proposition} 

\begin{proof} We may assume without loss of generality that $f_1,f_2$ are smooth outside singularities, since 
such functions are dense in $H^2_{\mbf a}$. Also, by using partitions of unity, 
we may assume that $f_1,f_2$ are supported in a small neighborhood of the singular points. 
Thus by Lemma \ref{normform}, the calculation reduces to the case $L=\partial z\partial-\alpha \partial$.

Using Green's formula, we see that  
\begin{equation}\label{eq4a}
2i((Lf_1,f_2)-(f_1,L^\dagger f_2))=\lim_{\varepsilon\to 0}(I_1(\varepsilon)+I_2(\varepsilon)+I_3(\varepsilon)),
\end{equation} 
where 
$$
I_1(\varepsilon)=\oint_{|z|=\varepsilon}(z\partial -\alpha)f_1\cdot \overline{f_2}d\overline{z},
$$
$$
I_2(\varepsilon)=-\oint_{|z|=\varepsilon}f_1\cdot (\overline{z}\overline{\partial}-\alpha)\overline{f_2}d\overline{z},
$$
$$
I_3(\varepsilon)=\alpha\oint_{|z|=\varepsilon}f_1\overline{f_2}d\overline{z},
$$
where the circles are oriented counterclockwise. By Proposition \ref{sobo}, 
$I_3(\varepsilon)\to 0$ as $\varepsilon\to 0$, so it suffices to show that 
$I_1(\varepsilon)+I_2(\varepsilon)\to 0$.  

Performing ``Ces\'aro summation" on \eqref{eq4a} (i.e., integrating \eqref{eq4a} from $\varepsilon$ to $2\varepsilon$ and dividing by $\varepsilon$), we obtain 
$$
2i((Lf_1,f_2)-(f_1,L^\dagger f_2))=\lim_{\varepsilon\to 0}(J_1(\varepsilon)+J_2(\varepsilon)),
$$
where $J_k(\varepsilon):=\varepsilon^{-1}\int_{\varepsilon}^{2\varepsilon}I_k(t)dt$. Thus, it suffices to show that 
$$
\lim_{\varepsilon\to 0}J_k(\varepsilon)=0,\ k=1,2.
$$ 
These statements for $k=1$ and $k=2$ are clearly equivalent, so 
we will only consider the case $k=1$. Using polar coordinates, we have 
 $$
J_1(\varepsilon)=-i\varepsilon^{-1}\int_{\varepsilon\le |z|\le 2\varepsilon}(z\partial -\alpha)f_1\cdot\overline{f_2}e^{-i\theta}d\sigma,
$$
where $d\sigma$ is the area element. But we know that the function $(z\partial-\alpha) f_1$ is in $H^1_{\rm loc}$. In particular, by the Sobolev embedding theorem (\cite{A}, p.97, formula (6)) it is in $L^p_{\rm loc}$ for every $p<\infty$. 
Thus by the H\"older inequality 
$$
|J_1(\varepsilon)|\le \varepsilon^{-1}||(z\partial-\alpha)f_1||_p ||f_2||_{\frac{p}{p-1}}. 
$$
The norm $||(z\partial-\alpha)f_1||_p$ is bounded as $\varepsilon\to 0$, and 
$$
||f_2||_{\frac{p}{p-1}}\le C(\delta) \varepsilon^{\frac{2(p-1)}{p}-\delta}
$$  
for any $\delta>0$ (since by Proposition \ref{sobo} $f_1=O(|z|^{-\delta})$). Thus, 
$$
|J_1(\varepsilon)|=O(\varepsilon^{\frac{p-2}{p}-\delta}),\ \varepsilon\to 0 
$$
for any $\delta>0$. 
So taking any $p>\frac{2}{1-\delta}$ for some $0<\delta<1$, we get that $J_1(\varepsilon)\to 0$, as desired. 
\end{proof} 

\begin{corollary}\label{matrop} Consider the operator 
$$
\mathcal{L}=\left(\begin{matrix} 0& L\\ L^\dagger & 0\end{matrix}\right)
$$
acting on the space $Y\subset \mathcal{H}\otimes \mbb C^2$ of pairs of smooth sections of $O(s)\otimes \overline{O(-2-s)}$. Then 
$\mathcal{L}$ is essentially self-adjoint on $Y$.
\end{corollary} 

\begin{proof} Let $Y^\vee$ be the domain of the adjoint operator to $\mathcal{L}$ initially defined on $Y$. 
Since $\mathcal{L}$ is obviously symmetric on $Y$, it extends canonically to a linear map $\mathcal{L}: Y^\vee\to \mathcal{H}\otimes \mbb C^2$. 
We have to show that the operator $\mathcal{L}$ is symmetric on $Y^\vee$, then it is self-adjoint on $Y^\vee$ and thus essentially self-adjoint on $Y$. For this, note that $Y^\vee$ is the space of pairs $(f,g)$ of distributional sections such that 
$f,g,Lg,L^\dagger f\in L^2$. Thus $Y^\vee=H^2_{\mbf a}\otimes \mbb C^2$. 
But then by Proposition \ref{intparts}, the form $\omega(y_1,y_2):=
(\mathcal Ly_1,y_2)-(y_1,\mathcal{L}y_2)$ vanishes on $Y^\vee$, which implies that 
$\mathcal{L}$ is symmetric, as desired. 
\end{proof} 

\section{Proof of Theorem \ref{mainthe}}

\subsection{Green functions} By Proposition \ref{l2prop1}, by replacing $L$ by $L-\zeta$ for generic $\zeta\in \mbb C$, we may assume without loss of generality that $L$ is injective on $H^2_{\mbf a}$. Then the closure of $\mathcal L$ with domain $H^2_{\mbf a}\otimes \mbb C^2$
(which, abusing notation, we will also denote $\mathcal{L}$) is also injective. Also this closure is self-adjoint by Proposition \ref{matrop}. This allows us to define the inverse of this closure, a self-adjoint operator on $\mathcal{H}\otimes \mbb C^2$ which we will denote by $\mathcal L^{-1}$.
We have 
$$
\mathcal{L}^{-1}=\left(\begin{matrix} 0& (L^\dagger)^{-1}\\ L^{-1} & 0\end{matrix}\right)
$$
for suitable operators $L^{-1},(L^\dagger)^{-1}$. 
In this subsection we consider the integral kernel (the Green function) of the operator $(L^\dagger)^{-1}$ which is a function of two  complex variables. We are mostly interested in the integrability properties of this function, arising from the structure of its singularities, which will allow us to show that the operator $(L^\dagger)^{-1}$ (and therefore its adjoint $L^{-1}$) is compact. 

\begin{proposition}\label{green1} The Green function of the operator $L^\dagger$ is in $L^2$ as a function of two complex variables. 
\end{proposition} 

\begin{proof} By Lemma \ref{normform}, it suffices to consider the case $L=\partial z\partial-\alpha\partial=\partial z |z|^{2\alpha}\partial |z|^{-2\alpha}$. 
Since the Green function of $\overline{\partial}$ is $\frac{1}{\pi(z-w)}$, we get that the Green function of $L^\dagger$ behaves near the origin like 
$$
G(z,w)\sim \frac{1}{\pi^2}|z|^{2\alpha}\int_{|u|< 1}\frac{|u|^{-2\alpha}d^2u}{(z-u)\overline{u}(u-w)}.
$$
Using that $\frac{1}{(z-u)(u-w)}=\frac{1}{z-w}(\frac{1}{z-u}-\frac{1}{w-u})$, we obtain 
$$
G(z,w)\sim \frac{1}{\pi^2}|z|^{2\alpha}\frac{H(z)-H(w)}{z-w}, 
$$
where
$$
H(z)=\int_{|u|<1}\frac{|u|^{-2\alpha}d^2u}{(z-u)\overline{u}}=\frac{1-|z|^{-2\alpha}}{2\alpha}+H_0(z),
$$
with $H_0(z)$ holomorphic. Thus 
$$
G(z,w)\sim\frac{1}{\pi^2}\frac{|z/w|^{2\alpha}-1}{2\alpha}\frac{1}{z-w}.
$$
Thus, it suffices to show that the function $\frac{|z/w|^{2\alpha}-1}{2\alpha}\frac{1}{z-w}$ 
is locally in $L^2$, i.e. that the integral
$$
I:=\int_{0<|z|<|w|<1}\left|\frac{|z/w|^{2\alpha}-1}{2\alpha}\right|^2\frac{1}{|z/w-1|^2}\frac{d^2zd^2w}{|w|^2} 
$$
is finite. 

Let us pass to variables $u=z/w, w$. Then $d^2zd^2w/|w|^2=d^2ud^2w$, so we have 
$$
I=\int_{0<|u|,|w|<1}\left|\frac{|u|^{2\alpha}-1}{2\alpha}\right|^2\frac{1}{|u-1|^2}d^2ud^2w=
$$
$$
\pi \int_{0<|u|<1}\left|\frac{|u|^{2\alpha}-1}{2\alpha}\right|^2\frac{1}{|u-1|^2}d^2u.
$$
So our job is to show that the function $f(u):=\frac{|u|^{2\alpha}-1}{2\alpha}\frac{1}{u-1}$ is $L^2$ on the disk $|u|\le 1$. But this is clear since this function is smooth outside $u=0,1$ and has at most logarithmic growth at $u=0$ and is bounded near $u=1$. 
\end{proof} 

\begin{corollary}\label{disspec} 
The  operator $\mathcal{L}$ has discrete spectrum, which is a discrete subset of $\mbb R$. 
Moreover, its eigenspaces are finite dimensional. 
\end{corollary} 

\begin{proof}
By Proposition \ref{green1}, the Green function of the operator $(L^\dagger)^{-1}$ is square integrable. 
This implies that this operator is Hilbert-Schmidt, in particular compact. Thus the operator $\mathcal L^{-1}$ 
is a compact self-adjoint operator. Therefore, by the Hilbert-Schmidt theorem its spectrum is discrete and is a sequence going to zero, and eigenspaces are finite dimensional. This implies the statement. 
\end{proof} 

\subsection{Essential self-adjointness of the algebra $\mbb C[L,L^\dagger]$ on smooth functions.}

Recall that $\wt V$ denotes the space of smooth sections of the line bundle $O(s)\otimes \overline{O(-s-2)}$ and $\wt\mcA$ denotes the algebra $\mbb C[L,L^\dagger]\subset \End \wt V$.
 
\begin{theorem} \label{esaa} The algebra $\wt\mcA_{\mbb R}$ is essentially self-adjoint. 
\end{theorem} 

\begin{proof} By Corollary \ref{disspec}, the closure of the operator $\mathcal{L}$ has discrete spectrum, with finite dimensional eigenspaces and eigenvalues going to $\infty$. Moreover, by elliptic regularity these eigenspaces consist of smooth sections outside singularities. The operator $L$ then defines a normal operator on each of these eigenspaces (as $L,L^\dagger$ commute as differential operators). Thus each of these eigenspaces is an orthogonal direct sum of joint eigenspaces of $L,L^\dagger$, which we know to be 1-dimensional and correspond to elements of $\Sigma$. Thus $L$ is an essentially normal operator on $\widetilde{V}$. 
So the result follows from Proposition \ref{esa1}. 
\end{proof} 

Let $V\subset \wt V$ be the space of smooth sections which vanish in a neighborhood of the four singular points. 
Let $\mcA=\mbb C[L,L^\dagger]\subset \End V$.

\subsection{Proof of Theorem \ref{mainthe}(ii)}

\begin{corollary}\label{matrop1} The algebra $\mcA_{\mbb R}$ is essentially self-adjoint. 
\end{corollary}

\begin{proof} By Theorem \ref{esaa}, it suffices to show that $S(\mcA)$ coincides with $S(\wt\mcA)=S_{\mbf a}$. It is clear that $S_{\mbf a}\subset S(\mcA)$, so we just need to establish the opposite inclusion. By definition, $S(\mcA)$ is the space of all distributional sections $f$ of $O(s)\otimes \overline{O(-2-s)}$ on $\mbb P^1\setminus \lbrace{0,a,1,\infty\rbrace}$ such that for any $k,l$ the section $L^k(L^\dagger)^lf$ is in $L^2$ outside singularities.
In other words, as distributions on the whole $\mbb  P^1$, we have 
$$
L^k(L^\dagger)^lf=\phi_{kl}+\xi_{kl},
$$
where $\phi_{kl}\in \mathcal{H}$ and $\xi_{kl}$ is a singular distribution supported at the four singular points, and this decomposition is clearly unique (here $\phi_{00}=f$ and $\xi_{00}=0$). Thus by elliptic regularity $f$ must be smooth outside the singular points. Let $p$ be one of the four singular points and let us work in a neighborhood of $p$. Since $L: H^2_{\alpha,\rm loc}\to L^2_{\rm loc}$ is surjective by Proposition \ref{surj}, there is a function $f_*\in L^2_{\rm loc}$ such that $Lf_*=\phi_{10}$ (so $f_*\in H^2_{\alpha,\rm loc}$). Then $g:=f-f_*$ is in $L^2_{\rm loc}$ and $Lg=\xi_{10}$. So we have $g(z)=g_1(\bar z)+g_2(\bar z)\frac{|z|^{2\alpha}-1}{2\alpha}$, where $g_1,g_2$ are holomorphic. Thus, $g\in H^2_{\alpha,\rm loc}$, so $\xi_{10}=0$ and $f\in H^2_{\alpha,\rm loc}$. 

Now $L^2f=L(Lf)=\phi_{20}+\xi_{20}$. Arguing as above, we find that $\xi_{20}=0$, so $L^2f\in L^2_{\rm loc}$ and $f\in H^4_{\alpha,\rm loc}$. Continuing in this way, we see that $\xi_{m0}=0$ for all $m$, so $f\in H^{2m}_{\alpha,\rm loc}$ for all $m$, hence $f\in S_{\alpha,\rm loc}$. Thus, $f\in S_{\mbf a}$, as claimed. 
\end{proof} 

Corollary \ref{matrop1} implies Theorem \ref{mainthe}(ii). This completes the proof of Theorem \ref{mainthe}. 

\section{The operator $D=L^\dagger L$ and a second proof of Theorem \ref{mainthe}} 

In this section we will give another proof of Theorem \ref{mainthe}, based on using the operator $D=L^\dagger L$ instead of 
$\mathcal{L}$. This operator is formally symmetric on $\wt V$, but it turns out not to be essentially self adjoint (see Remark \ref{remnonsa} below). So we first describe a suitable self-adjoint extension of $D$. 

\subsection{The self-adjoint extension $D$} 

\begin{proposition}\label{selfadex} The operator $D$ is self-adjoint on the space $W:=H_{\mbf a}^4$. 
\end{proposition} 

\begin{proof} First of all, by Proposition \ref{intparts}, 
$D$ is symmetric on $W$. Let $W^\vee\subset \mathcal{H}$ be the domain of the adjoint operator to $D$, and $f\in W^\vee$. Our job is to show that $f\in W$. 

Let $\widetilde{W}$ be the space of $F\in {\mathcal H}$ such that $DF\in \mathcal{H}$. Thus, near each singular point $p$, the section $F$ has a representation \eqref{eq40} with some coefficients $C_j=C_{jp}(F)$. By Proposition \ref{charfun2}, $W$ is the subspace of of $\widetilde{W}$ of codimension $4$, consisting of $F$ for which the corresponding coefficients $C_{2p}(F)$ vanish at all the four singular points $p=0,a,1,\infty$. 

Given a test section $\phi\in C^\infty(\mbb P^1,O(s)\otimes \overline{O(-2-s)})$, we have 
$(D\phi,f)=(\phi,Df)$. This implies that $f\in \widetilde{W}$. So it remains to show that if $f\in W^\vee$ then $C_{2p}(f)=0$ for all $p$. 

Let $h_p\in W$ be a section supported near $p$ with $C_{1p}(h_p)=1$. 
By a direct calculation (using the local presentation $L=\partial z\partial-a_p\partial$), 
it is easy to show that $(Df,h_p)-(f,Dh_p)$ is a nonzero multiple of $C_{2p}(f)$. Thus, $C_{2p}(f)=0$, as desired.  
\end{proof} 

\begin{corollary} \label{selfad} The operator
$D$ is essentially self-adjoint on $S_{\mbf a}$.
\end{corollary} 

\begin{proof} This follows directly from Proposition \ref{selfadex}.  
\end{proof} 

\subsection{The Green function of $D$} 

\begin{proposition}\label{green2} The Green function of the operator $D=L^\dagger L$ is in $L^2$ as a function of two complex variables. Moreover, it defines an $L^1$ function on the diagonal. 
\end{proposition} 

\begin{proof} Again by Lemma \ref{normform} it suffices to consider $L=\partial z\partial-\alpha\partial$, in which case $D=\frac{1}{16}\Delta |z|^{2+2\alpha}\Delta |z|^{-2\alpha}$. 
Since the Green function of $\Delta$ is $\frac{1}{2\pi}\log|z-w|$, we get that 
the Green function of $D$ behaves near the origin like 
$$
K(z,w)\sim \frac{4}{\pi^2}|w|^{2\alpha}\int_{|u|<1}\frac{\log|z-u|\log|w-u|-\log|z|\log|w|}{|u|^{2+2\alpha}}d^2u.
$$
Thus $K$ has logarithmic singularities at $z=0$ and $w=0$ and no singularity on the diagonal, which implies that it is locally in $L^2$. Moreover, 
$$
K(z,z)\sim \frac{4}{\pi^2}|z|^{2\alpha}\int_{|u|<1}\frac{\log^2|z-u|-\log^2|z|}{|u|^{2+2\alpha}}d^2u.
$$
Thus, $K(z,z)$ has logarithmic growth at $z=0$, hence is in $L^1$, as claimed. 
\end{proof} 

\begin{corollary}\label{disspec1} The self-adjoint operator $D$ has discrete spectrum, which is a discrete subset of $\mbb R$. 
Moreover, all eigenspaces of $D$ are finite dimensional. 
\end{corollary} 

\begin{proof} As before, we may assume by Proposition \ref{l2prop1} that $L$ is injective. Then if $f\in H_{\mbf a}^4$ 
and $Df=0$ then $(Df,f)=(Lf,Lf)=0$, so $Lf=0$, hence $f=0$. Thus $D$ is injective on $H_{\mbf a}^4$ and also self-adjoint by Proposition \ref{selfadex}, 
so we can define the operator $D^{-1}$. By Proposition \ref{green2}, $D^{-1}$ is a compact (in fact, trace class) self-adjoint operator. Therefore, its spectrum is discrete and is a sequence going to zero, and eigenspaces are finite dimensional.
This implies the statement. 
\end{proof} 

\subsection{A second proof of Theorem \ref{mainthe}} 

Corollary \ref{disspec1} allows us to give another proof of Theorem \ref{mainthe}. Namely, we just need to give a second proof of 
Theorem \ref{esaa}, from that point the proof is the same. 

So let us give a second proof of Theorem \ref{esaa}, which is the same as its first proof but using the operator $D$ instead of $\mathcal{L}$. By Corollary \ref{disspec1}, the operator $D$ has discrete spectrum, with finite dimensional eigenspaces and eigenvalues going to $\infty$. Moreover, by elliptic regularity these eigenspaces consist of smooth sections outside singularities. The operator $L$ then defines a normal operator on each of these eigenspaces (as $L,L^\dagger$ commute as differential operators). Thus each of these eigenspaces is an orthogonal direct sum of joint eigenspaces of $L,L^\dagger$, which we know to be 1-dimensional and correspond to elements of $\Sigma$. Thus $L$ is an essentially normal operator on $\widetilde{V}$. 
So the result follows from Proposition \ref{esa1}. 

\begin{remark} Note that the above two proofs of Theorem \ref{mainthe} do not use Nelson's Theorem \ref{ne1}. Indeed, the commutativity of the spectral resolutions of $L,L^\dagger$ comes out as a consequence of the argument and does not need to be known in advance. 
\end{remark} 

\section{A third proof of Theorem \ref{mainthe}}

In this section we give a third proof of Theorem \ref{mainthe}, which uses the operator $D$ and Nelson's theorem \ref{ne1}, but avoids the use of Green functions. 

\subsection{Essential normality of $L$} 

\begin{proposition}\label{essno} 
The operator $L$ is an essentially normal operator on $S_{\mbf a}$ and a normal operator on $H_{\mbf a}^2$.
\end{proposition} 

\begin{proof} Let $A={\rm Re}L=\frac{1}{2}(L+L^\dagger)$, $B={\rm Im}L=\frac{1}{2i}(L-L^\dagger)$. Our job is to show that 
$A,B$ strongly commute on $S_{\mbf a}$. To this end, note that $D=A^2+B^2$. By Proposition \ref{intparts}, $(Lf_1,f_2)=(f_1,L^\dagger f_2)$ for $f_1,f_2\in S_{\mbf a}$, hence $A,B$ are symmetric on $S_{\mbf a}$. Finally, by Corollary \ref{selfad}, the operator $D$ is essentially self-adjoint on $S_{\mbf a}$. Hence, the assumptions of Nelson's Theorem \ref{ne1} are satisfied. Thus, Theorem \ref{ne1} implies that $A,B$ strongly commute, as desired. 
\end{proof} 

This shows that the joint spectrum of $L,L^\dagger$ is well defined. So to prove Theorem 
\ref{mainthe}, it remains to show (by Proposition \ref{esa1}) that this spectrum is discrete. In the next subsection we give a proof of 
of this fact which does not use Green functions. 

\subsection{Discreteness of the joint spectrum of $L,L^\dagger$.} Formula \eqref{eq7} implies that single-valued joint eigenfunctions of $L,L^\dagger$ belong to the space $S_{\mbf a}$. Hence each eigenvalue $\Lambda$ corresponding to such an eigenfunction is a point of ${\rm Spec}(L)$. We would like to show that in fact ${\rm Spec}(L)$ does not contain any other points. Namely, we have the following theorem, which together with Propositions \ref{essno} and \ref{esa1} implies Theorem \ref{mainthe}. 

\begin{theorem}\label{main2} We have ${\rm Spec}(L)=\Sigma$, the set of eigenvalues of $L$ which give rise to a real monodromy representation. In particular, ${\rm Spec}(L)$ is discrete. Moreover, each eigenvalue $\Lambda\in \Sigma$ is simple. 
\end{theorem} 

\begin{proof} We only need to prove that ${\rm Spec}(L)$ is contained in $\Sigma$, everything else will then follow from the results proved before. 

Fix a closed disk $\mathcal{D}\subset \mbb P^1\setminus \lbrace{0,a,1,\infty\rbrace}=\mbb C\setminus \lbrace{0,a,1\rbrace}$ with center $z_0\in \mathcal{D}$. Fix a small enough $\varepsilon>0$, and let $\mathcal{R}=\mathcal{R}_\varepsilon\supset \mathcal {D}$ be the set of points of $\mbb P^1$ at distance $\ge \varepsilon$ from $\lbrace{0,a,1,\infty\rbrace}$. Let $\gamma_i\subset \mathcal{R}$, $i=1,2,3$, be closed paths around punctures generating $\pi_1(\mbb C\setminus \lbrace{0,a,1\rbrace},z_0)$. 

Assume the contrary, i.e., that $\Lambda\in {\rm Spec}(L)$, but $\Lambda\notin \Sigma$. Then there is $K>0$ such that for any solution $\psi$ of the system $L\psi=\Lambda\psi$, $L^\dagger\psi=\overline{\Lambda}\psi$ over $\mathcal{D}$
we have 
\begin{equation}\label{eq9} 
\sum_{i=1}^3 \sup_{\mathcal D}|\gamma_i\psi-\psi|\ge K\sup_{\mathcal D}(|\psi|+|\partial \psi|+|\ol \partial\psi|+|\partial\ol \partial \psi|),
\end{equation} 
where $\gamma_i\psi$ denotes the image of $\psi$ under holonomy along $\gamma_i$; indeed, the ratio of the LHS to the RHS is a positive continuous function 
on the projectivization of the space of solutions, which is compact, hence this function is bounded away from zero. 

On the other hand, the joint spectral decomposition for $L,L^\dagger$(whose existence follows from Proposition \ref{essno}) implies that there are ``almost joint eigenfunctions". Namely, let $\mathcal H_n\subset \mathcal{H}$ be the image of the spectral projection $P_n$ for $L$ corresponding to the disk around $\Lambda$ of radius $1/n$. Note that 
all functions $\psi\in \mathcal H_n$ are smooth outside the four singular points, since $L^m(L^\dagger)^k\mathcal H_n\subset \mathcal H_n$ for all $m,k\ge 0$. Thus we have a sequence of single-valued ``almost joint eigenfunctions" $\psi_n\in \mathcal H_n$ such that $\sup_{\mathcal R}|\psi_n|=1$ but 
\begin{equation}\label{eq10}
||(L-\Lambda)\psi_n||_{C^r(\mathcal{R})}+||(L^\dagger-\overline{\Lambda})\psi_n||_{C^r(\mathcal{R})}\to 0,\ n\to \infty
\end{equation} 
for any $r\ge 0$. Thus, there exists $M>0$ such that
\begin{equation}\label{eq55}
\sup_{\mathcal D}(|\psi_n|+|\partial \psi_n|+|\ol \partial\psi_n|+|\partial\ol \partial \psi_n|)\ge M,\ n\gg 0.
\end{equation} 

Now let $\phi_n$ be the solution of the holonomic system 
$$
L\psi=\Lambda\psi,\ L^\dagger\psi=\overline{\Lambda}\psi
$$ 
over $\mathcal{D}$
with the same initial conditions as $\psi_n$ at $z_0$. Then the functions $\eta_n:=\psi_n-\phi_n$ satisfy the equations 
\begin{equation}\label{eq11}
(L-\Lambda)\eta_n=(L-\Lambda)\psi_n,\ (L^\dagger-\overline{\Lambda})\eta_n=(L^\dagger-\overline{\Lambda})\psi_n.
\end{equation} 
with zero initial conditions. This together with \eqref{eq10} implies that 
$$
\sup_{\mathcal D}(|\eta_n|+|\partial \eta_n|+|\ol \partial\eta_n|+|\partial\ol \partial \eta_n|)\to 0,\ n\to \infty,  
$$
hence using \eqref{eq55}, we get 
\begin{equation}\label{eq57}
\sup_{\mathcal D}(|\phi_n|+|\partial \phi_n|+|\ol \partial\phi_n|+|\partial\ol \partial \phi_n|)\ge M,\ n\gg 0. 
\end{equation} 
Also, solving the initial value problem with zero initial conditions for 
system \eqref{eq11} along $\gamma_i$ and using \eqref{eq10}, we see that 
$$
\sum_{i=1}^3 {\rm sup}_{\mathcal D}|(\psi_n-\gamma_i\phi_n)-(\psi_n-\phi_n)|=\sum_{i=1}^3 {\rm sup}_{\mathcal D}|\gamma_i\phi_n-\phi_n|\to 0,\ n\to \infty.
$$
By \eqref{eq57}, this contradicts \eqref{eq9}. 
\end{proof} 

\begin{remark} Here is a second proof of Theorem \ref{main2} using the theory of rigged Hilbert spaces (Gelfand triples). This theory implies that for any $\Lambda\in {\rm Spec}(L)$ there must exist a {\it generalized joint eigenfunction} 
$\psi$ of $L,L^\dagger$ with eigenvalues $\Lambda,\overline{\Lambda}$ on $\mbb C\setminus \lbrace{0,a,1\rbrace}$ (see \cite{GS}, p. 184-186, Theorems 1 and 2; in our case we need a straightforward generalization of this theory from the case of one self-adjoint operator to finite collections of strongly commuting self-adjoint operators). But then by elliptic regularity, this eigenfunction must be smooth, which implies that $\Lambda\in \Sigma$. 
\end{remark} 

\section{The spectral description of the global Sobolev and Schwartz spaces attached to $L$} 

\begin{corollary}\label{coro1} (i) $H^{2m}_{\mbf a}$ is the space of $f=\sum_{\Lambda\in \Sigma}c_\Lambda\psi_\Lambda\in 
\mathcal{H}$ such that 
$$
||f||_{H^{2m}_{\mbf a}}^2:=\sum_{\Lambda\in \Sigma} (1+|\Lambda|^{2m}) |c_\Lambda|^2<\infty.
$$
Moreover, this norm introduces a Hilbert space structure on $H^{2m}_{\mbf a}$, and the operators 
$L, L^\dagger: H_{\mbf a}^{2m+2}\to H_{\mbf a}^{2m}$ are bounded in this norm. 

(ii) For any polynomial $P$ of two variables, $P(L,L^\dagger)$ makes sense as an unbounded essentially normal operator on $\mathcal{H}$ with domain $S_{\mbf a}$. In particular, if $P$ is real (i.e., $P(z,\overline z)$ is real for all complex $z$, or, equivalently, $P(y,x)=\overline{P}(x,y)$) then the operator $P(L,L^\dagger)$ is essentially self-adjoint on $S_{\mbf a}$.

(iii) For any $f_1,f_2\in H^{2m}_{\mbf a}$ and any polynomial $P$ of degree $m$, one has
$$(P(L,L^\dagger)f_1,f_2)=(f_1,\overline{P}(L^\dagger,L)f_2).$$ 

(iv) Let $P$ be an elliptic polynomial in two variables of degree $m$, i.e., its homogeneous leading part does not vanish for $z\ne 0$ (equivalently, the operator $P(L,L^\dagger)$ is elliptic outside singularities). 
Then the operator  $P(L,L^\dagger)$ is closed on $H^{2m}_{\mbf a}$.  

(v) $S_{\mbf a}$ is the space of elements of $\mathcal{H}$ with rapidly decaying Fourier coefficients with respect 
to the basis $\psi_\Lambda$, i.e., elements of the form $\sum_{\Lambda\in \Sigma} c_\Lambda\psi_\Lambda$, where 
$$\sum_{\Lambda\in \Sigma} |\Lambda|^m |c_\Lambda|^2<\infty$$ for any $m\ge 0$. 

(vi) $S_{\mbf a}$ is the space of smooth vectors for the unitary representation $\rho$ of $\mbb R^2$ defined by the operators $A={\rm Re}L,B={\rm Im}L$, 
$\rho(t,s)=e^{i(tA+sB)}$.  
\end{corollary} 

\begin{proof} (i),(ii) are is immediate from Theorem \ref{mainthe}. 
Part (iii) then follows from (ii), and part (v) from (i). Part (iv) follows from (i) in a standard way. 
Finally, (vi) follows from Corollary 9.3 in \cite{N}. 
\end{proof} 

\begin{remark} We can extend the definition of the generalized Sobolev spaces $H^k_{\mbf a}$ to the case when $k\in \mbb R_+$ is not necessarily an even integer. Namely, $H^k_{\mbf a}$ is the Hilbert space of $f\in \mathcal{H}$ such that 
$$
||f||_{H^k_{\mbf a}}^2:=\sum_{\Lambda\in \Sigma} (1+|\Lambda|^k) |c_\Lambda|^2<\infty.
$$
It is easy to see that $H^k_{\mbf a}\subset H^\ell_{\mbf a}$ if 
$k\ge \ell$ and that the operators $L,L^\dagger: H^{k+2}_{\mbf a}\to H^{k}_{\mbf a}$ 
are bounded. 
\end{remark} 

\section{Remarks} 

\begin{remark}\label{remnonsa} The proof of Proposition \ref{selfadex} shows that, unlike the case of elliptic operators, 
the operator $D$ is {\bf not} essentially self-adjoint on the space $\wt V=C^\infty(\mbb P^1,O(s)\otimes \overline{O(-2-s)})$. In fact, this proof shows that the domain $\wt V^\vee_D$ of the adjoint operator for $(D,\wt V)$ is $\wt H_{\mbf a}^4$, on which $D$ is not symmetric. More precisely, $(\wt H_{\mbf a}^4)^\perp\subset \wt H_{\mbf a}^4$ is the space of 
$f\in \wt H_{\mbf a}^4$ such that $C_{1p}(f)=C_{2p}(f)=0$. Thus $\wt H_{\mbf a}^4/(\wt H_{\mbf a}^4)^\perp$ has dimension $8$ and carries a hyperbolic skew-Hermitian form of rank $8$, given by 
$$
\omega(f,g)=(Df,g)-(f,Dg),
$$ 
which is proportional to $\sum_p (C_{1p}(f)\overline{C_{2p}(g)}-C_{2p}(f)\overline{C_{1p}(g)})$. The subspace $H_{\mbf a}^4/(\wt H_{\mbf a}^4)^\perp\subset \wt H_{\mbf a}^4/(\wt H_{\mbf a}^4)^\perp$ is Lagrangian 
with respect to this form. 

2. This shows that $(D,\wt V)$ in fact has many self-adjoint extensions, and the extension $(D,H_{\mbf a}^4)$ 
considered above is just one of them. Namely, as explained in Subsection \ref{usa}, such extensions are parametrized by  
maximal isotropic subspaces $Y\subset \wt H_{\mbf a}^4$, i.e., by the Lagrangian Grassmannian of the 8-dimensional hyperbolic skew-Hermitian space $\wt H_{\mbf a}^4/(\wt H_{\mbf a}^4)^\perp$. However, none of these subspaces except $H_{\mbf a}^4$ contain a common invariant domain for $L,L^\dagger$, so Nelson's Theorem \ref{ne1} does not apply to them. In fact, the maximal isotropic subspace $H_{\mbf a}^4$ is the only one giving rise to a unitary representation of $\mbb R^2$, $\rho(t,s)=e^{i(tA+sB)}$. Indeed, if $W$ is another such subspace then the above argument shows that the corresponding unitary representation has discrete spectrum (i.e., is a direct sum of 1-dimensional representations). But then the eigenfunctions form a basis of $\mathcal H$, which has to be the same as for $H_{\mbf a}^4$. Thus $W=H_{\mbf a}^4$. 

3. It may be shown that the self-adjoint extension of $D$ constructed above is the Friedrichs extension (cf. \cite{AG}). 
\end{remark}

\begin{remark} The properties of the generalized Schwartz space asserted in Corollary \ref{coro1}, in particular the symmetry of $P(L,L^\dagger)$ on $S_{\mbf a}$, are nontrivial and not at all automatic. For example, let $\mbf a=0$ and consider a similar space $\widetilde{S}$ of $f$ such that $D^nf\in \mathcal{H}$ as distributions for any $n\ge 0$. Then $\widetilde{S}$ contains functions which near singularities behave like $C_2\log^2|z|+C_1\log|z|+...$ for any $C_1,C_2$ (as we no longer require that $Lf\in \mathcal{H}$). Thus $D$ is not symmetric on $\widetilde{S}$ and an analog of Corollary \ref{coro1} fails. This can also be seen from Theorem \ref{vn}. 
\end{remark} 

\begin{remark} Self-adjoint elliptic differential operators $A$ of order $r$ on a compact $d$-dimensional manifold $X$ satisfy the {\it Weyl law}, which says that the number of eigenvalues of $A$ in the interval $[Na,Nb]$ behaves as 
$(2\pi)^{-n}{\rm Vol}(\sigma(A)\in [a,b])N^{d/r}$ as $N\to \infty$, where $\sigma(A)$ is the symbol of $A$, and ${\rm Vol}(\sigma(A)\in [a,b])$ is the symplectic volume of the subset in $T^*X$ defined by the condition $\sigma(A)\in [a,b]$. We expect that there is a similar Weyl law in our situation. Namely, we expect that for any region $R$ in $\mbb C$ (say, with piecewise smooth boundary), we have 
$$
|\Sigma\cap NR|\sim (2\pi)^{-n}{\rm Vol}(\sigma(L)\in R)N.
$$
Note that the volume on the right hand side is finite, which corresponds to the fact that the spectrum is discrete. 
Also in the case when $a\in [0,1]$ and $\mbf a=0$ (considered in \cite{T}), this suggests that $L$ should have many other eigenvalues with real monodromy than the real ones considered in \cite{T}; indeed, the number of these real eigenvalues with absolute value $\le N$ behaves like $\sqrt{N}$ rather than $N$, as they are given by a Sturm-Liouville problem for a second order differential operator in one real variable. 
This agrees with the discussion of Subsection \ref{takh}. 
\end{remark} 

\section{Degenerate cases} 

\subsection{The reducible case} 

Let us now discuss the degenerate cases (a) and (b) of Darboux operators (as in Proposition \ref{da1}), 
assuming as before that $s+1\in i\mbb R$. In case (a), we have 
$$
L=(z\partial+1/2)^2+c(z\partial+1/2),\ c\in \mbb C, 
$$ 
so the spectral theory of $L$ reduces to the spectral theory of the operator $z\partial+1/2$, 
which is very easy to describe. Namely, $z\partial +1/2$ is an unbounded normal operator 
initially defined on the dense domain $C^\infty(\mbb P^1,O(s)\otimes \overline{O(-2-s)})$, 
whose real and imaginary part are essentially self-adjoint with commuting spectral resolutions and 
describe the derivative of the natural unitary action of $\mbb C^{\times}$ on the Hilbert space
$\mathcal{H}$. This operator has continuous spectrum, with generalized 
eigensections $|z|^{-1+2i\gamma}e^{in{\rm arg}(z)}$ (written as ordinary functions on the affine chart $\mbb C^{\times}\subset \overline{\mbb C}=\mbb P^1$, using the invariant nonzero section of $O(s)\otimes \overline{O(-2-s)}$ over $\mbb C^{\times}$ to trivialize this bundle) and eigenvalues $\lambda=\frac{n}{2}+i\gamma$, with $\gamma\in \mbb R$ and $n\in \mbb Z$. Thus the spectrum of $L$ consists of numbers $\Lambda=(\frac{n}{2}+i\gamma)^2+c(\frac{n}{2}+i\gamma)$, and is independent of $s$. Moreover, the spectral multiplicity is $2$ since we have two generalized eigenfunctions for each spectral point. 

We note that the corresponding Schwartz space $S$, i.e., the space of $f\in {\mathcal H}$ such that 
$P(L,L^\dagger)f\in {\mathcal H}$ for any polynomial $P$ in this case is larger 
than the space $C^\infty(\mbb P^1,O(s)\otimes \overline{O(-2-s)})$. Namely, it is the space of smooth functions 
on $\mbb C^{\times}$ whose all derivatives in the coordinates $\rho=\log|z|$ and $\theta={\rm arg}z$ 
are in $L^2(\mbb C^{\times},|dz/z|)$ (the Sobolev space $H^\infty(\mbb C^{\times})$ of infinite order).

\subsection{The trigonometric case}
The trigonometric case (b) of Proposition \ref{da1} is more interesting, as the monodromy is irreducible. 
In this case up to symmetries $L=\partial (z^2-1)\partial+(c_0z+c_1)\partial$, and the holomorphic eigenfunctions of $L$ are hypergeometric. Let $\alpha_\pm$ be the characteristic exponents of $L$ at the points $\pm 1$ 
(depending of $c_0,c_1$), and assume that they are imaginary. Similarly to the elliptic case, one may show 
that $L$ is an unbounded essentially normal operator on the Schwartz space $S_{\alpha_+,\alpha_-}$ of $f$ such that $P(L,L^\dagger)f\in \mathcal{H}$ for any polynomial $P$. 
The space $S_{\alpha_+,\alpha_-}$ consists of the sections $f$ which are smooth outside of $-1,1,\infty$, have the form 
$$
f(z)=g_\pm(z)\frac{(z\pm 1)^{2\alpha_\pm}-1}{2\alpha_\pm}+h_\pm(z)
$$
with smooth $g_\pm,h_\pm$ near $z=\pm 1$, and have derivatives of all orders in the coordinates $\rho=\log|z|$ and $\theta={\rm arg}z$ all $L^2_{\rm loc}(\mbb C^{\times},|dz/z|)$ near $z=\infty$. 
Thus, we have two singular points ($z=1,-1$) of the same type as in the elliptic case and one ($z=\infty$) 
of the same type as in the reducible case (a). 

Let us now study the spectrum of $L$. We give a sketch of the argument, assuming for simplicity that $\alpha_\pm=0$, i.e, $L=\partial (z^2-1)\partial$ (with $s=-1$). Then the monodromy operators $M_+$, $M_-$ of the equation $L\psi=\Lambda \psi$ at the points $1,-1$ are unipotent. Let $\Lambda=\mu^2-1/4$. Then the monodromy operator $M_\infty=(M_-M_+)^{-1}$ at infinity has eigenvalues $-e^{\pm 2\pi i\mu}$. Let $\beta=2+2\cos 2\pi \mu$, so that ${\rm Tr}M_\infty=2-\beta$. We can find a basis in which $M_-, M_+$ have the form 
$$
M_-=\left(\begin{matrix}1& i \\ 0 & 1\end{matrix}\right),\ M_+=\left(\begin{matrix}1& 0 \\ i\beta & 1\end{matrix}\right).
$$

Consider first the case when $\mu$ is real. Then it is easy to see that these matrices preserve the Hermitian form $z_1\overline{z_2}+z_2\overline{z_1}$. So if $f_1,f_2$ is the eigenbasis of $M_\infty$ in the space of local solutions such that $(f_1,f_1)=1$ and $(f_2,f_2)=-1$ then the function $|f_1|^2-|f_2|^2$ is a single-valued real eigenfunction. But since $f_1$ behaves as $z^{-1/2+\mu}$ and $f_2$ as 
$z^{-1/2-\mu}$, this function does {\it not} define a distribution on $S_{\alpha_+,\alpha_-}$ for real $\mu\ne 0$ (i.e., $0<\beta<4$), as it grows exponentially in the logarithmic coordinate $\rho=\log|z|$. 

On the other hand, if $\mu=\frac{n}{2}+i\gamma$ for integer $n$ and real $\gamma$ 
then the eigenvalues of $M_\infty$ are $\pm e^{2\pi \gamma}$ and $\pm e^{-2\pi \gamma}$, and $\beta=2\pm 2\cosh \gamma$ (so $\beta\le 0$ or $\beta\ge 4$). So in this case the single valued solution is $f_1\overline{f_2}+f_2\overline{f_1}$, and it gives a generalized eigenfunction of $L$ (a distribution on $S_{\alpha_+,\alpha_-}$). Hence $L$ has continuous spectrum $(\frac{n}{2}+i\gamma)^2$.  
Thus the spectrum is the same as in the reducible case (with $c=0$). However, unlike the reducible case, the spectral multiplicity is now $1$, as the monodromy of the equation $L\psi=\Lambda\psi$ is irreducible. 

\begin{remark} Thus we see that unlike the elliptic case (c), in the trigonometric case (b) a single-valued joint eigenfunction of $L,L^\dagger$ may {\it not} give rise to a spectral point.  
\end{remark}

\end{document}